\documentclass[reqno, 11pt,usenames,dvipsnames]{amsart}
\usepackage{graphicx, enumitem} 
\usepackage{caption}
\usepackage{subcaption}
\usepackage{amsmath,amssymb,amsthm, mathrsfs, mathtools}
\usepackage{cite}
\usepackage[foot]{amsaddr}
\usepackage{thmtools}
\usepackage{hyperref}
\usepackage{bm}
\usepackage{xcolor}
\usepackage{float}
\usepackage{appendix}
\hypersetup{
    colorlinks,
    linkcolor={red!60!black},
    citecolor={green!60!black},
    urlcolor={blue!60!black}
}
\usepackage[T1]{fontenc}
\usepackage{lmodern}
\usepackage[babel]{microtype}
\usepackage[english]{babel}
\usepackage{comment}
\usepackage{multirow}

\usepackage{geometry}
\numberwithin{equation}{section}
\numberwithin{figure}{section}

\usepackage{enumitem}
\usepackage{xcolor,colortbl}
\usepackage{enumitem}

\theoremstyle{plain}
\newtheorem{thm}{Theorem}[section]
\newtheorem{theorem}[thm]{Theorem}

\newtheorem{corollary}[thm]{Corollary}
\newtheorem{conj}[thm]{Conjecture}
\newtheorem{lemma}[thm]{Lemma}

\newtheorem{problem}[thm]{Problem}

\theoremstyle{definition}
\newtheorem{definition}[thm]{Definition}

\newtheorem{secclaim}{Claim}[subsection]
\newtheorem{claim}{Claim}[section]
\newtheorem{fact}[thm]{Fact}
\newtheorem{construction}[thm]{Construction}

\theoremstyle{remark}
\newtheorem{rem}[thm]{Remark}

\newcommand{\ext}{\mathop{\mathrm{ext}}}

\newcommand{\eps}{\varepsilon}

\newcommand{\ba}{\bm{\alpha}}
\renewcommand{\b}{\bm}
\newcommand{\opt}{\mathop{\textsc{opt}}}

\newcommand{\feas}{\mathop{\textsc{feas}}}

\newcommand{\maxs}{10^7}
\newcommand{\ER}{Erd\H{o}s-Rothschild }
\newcommand{\sm}{\setminus}

\newcommand{\mc}[1]{\mathcal{#1}}
\newcommand{\mb}[1]{\mathbb{#1}}

\newcommand{\x}{\mathcal{X}}

\newcommand{\aA}{\alpha}
\newcommand{\bB}{\beta}
\newcommand{\bb}{\bm{\beta}}
\newcommand{\gG}{\gamma}
\newcommand{\dD}{\delta}

\newcommand{\sS}{\sigma}

\newcommand{\DD}{\Delta}

\newcommand{\ff}{f}
\newcommand{\hf}{\tilde{g}}
\newcommand{\gf}{g}
\newcommand{\frelax}{\tilde{f}}
\newcommand{\rev}{r_2}
\newcommand{\rmax}{r}
\newcommand{\Rmax}{R}
\newcommand{\Rev}{R_2}
\newcommand{\gmax}{g^{\rm max}}

\title[A framework for the generalised Erd\H{o}s-Rothschild problem]{A framework for the generalised Erd\H{o}s-Rothschild problem and a resolution of the dichromatic triangle case}

\author[P. Gupta]{Pranshu Gupta$^*$}
\thanks{$^*$ \scriptsize Faculty of Computer Science and Mathematics, University of Passau, 94030 Passau, Germany.}
\author[Y. Pehova]{Yani Pehova$^\dagger$}
\thanks{$^\dagger$ \scriptsize Department of Mathematics, London School of Economics and Political Science, London WC2A 2AE, UK, supported by EPSRC Fellowship EP/V038168/1.}
\author[E. Powierski]{Emil Powierski$^\ddagger$}
\thanks{$^\ddagger$ \scriptsize Mathematical Institute, University of Oxford, Oxford OX2 6GG, UK}
\author[K. Staden]{Katherine Staden$^\S$}
\thanks{$^\S$ \scriptsize School of Mathematics and Statistics, The Open University, Milton Keynes MK7 6AA, UK, supported by EPSRC Fellowship EP/V025953/1.}
\thanks{\scriptsize Contact: \href{mailto:pranshu.gupta@uni-passau.de}{\texttt{pranshu.gupta@uni-passau.de}}, \href{mailto:y.pehova@lse.ac.uk}{\texttt{y.pehova@lse.ac.uk}}, \href{mailto:powierski@maths.ox.ac.uk}{\texttt{powierski@maths.ox.ac.uk}}, \href{mailto:katherine.staden@open.ac.uk}{\texttt{katherine.staden@open.ac.uk}}}

\begin{document}

\begin{abstract}
The \ER problem from 1974 asks for the maximum number of $s$-edge colourings in an $n$-vertex graph which avoid a monochromatic copy of $K_k$, given positive integers $n,s,k$. In this paper, we systematically study the generalisation of this problem to a given forbidden family of colourings of $K_k$.
This problem typically exhibits a dichotomy whereby for some values of $s$, the extremal graph is the `trivial' one, namely the Tur\'an graph on $k-1$ parts, with no copies of $K_k$;
while for others, this graph is no longer extremal and determining the extremal graph becomes much harder. 

We generalise a framework developed for the monochromatic \ER problem to the general setting and work in this framework to obtain our main results, which concern two specific forbidden families: triangles with exactly two colours, and improperly coloured cliques. We essentially solve these problems fully for all integers $s \geq 2$ and large $n$. In both cases we obtain an infinite family of structures which are extremal for some $s$, which are the first results of this kind.

A consequence of our results is that for every non-monochromatic colour pattern, every extremal graph is complete partite.
Our work extends work of Hoppen, Lefmann and Schmidt and of Benevides, Hoppen and Sampaio.
\end{abstract}

\maketitle


\section{Introduction}

This paper concerns an extremal problem on edge colourings of graphs with certain forbidden structures: given a collection $\x$ of $s$-edge coloured cliques on $k$ vertices, which $n$-vertex graph has the maximum number of colourings which do not contain any elements of $\x$? More formally,
let $k \geq 3$ be a clique size, $s \geq 2$ a number of colours,  
and let $\x$ be a family of $s$-edge colourings of $K_k$. 
Given a graph $G$, we say that an $s$-edge colouring of $G$ is \emph{$\x$-free} (or \emph{valid} if $\x$ is clear from the context) if it avoids all coloured copies of $K_k$ which lie in $\x$, and define $F(G;\x)$ to be the number of $\x$-free $s$-edge colourings of $G$. 
The goal is to determine
$$
F(n;\x) := \max_{|V(G)|=n}F(G;\x),
$$
as well as the set of \emph{$\x$-extremal} graphs $G$, i.e.\ those satisfying $F(G;\x)=F(n;\x)$. By considering all colourings of the Tur\'an graph $T_{k-1}(n)$,
which is the $n$-vertex complete $(k-1)$-partite graph with part sizes as equal as possible,
we can see that
\begin{equation}\label{eq:trivial}
s^{t_{k-1}(n)} \leq F(n;\x) \leq s^{\binom{n}{2}},
\end{equation}
where $t_{k-1}(n)=(1-\frac{1}{k-1}+o(1))\binom{n}{2}$ is the number of edges in $T_{k-1}(n)$. We refer to $s^{t_{k-1}(n)}$ as the \emph{trivial lower bound} for $F(n;\x)$ and to $T_{k-1}(n)$ as the \emph{trivial example}.
The inequalities in \eqref{eq:trivial} give the order of magnitude of $F(n;\x)$, so an asymptotic solution to the problem above would be to determine the limit
$$
\lim_{n \to \infty}\frac{\log F(n;\x)}{\binom{n}{2}},
$$
which was shown to exist in~\cite{alon2004number} for forbidden monochromatic cliques using an entropy inequality of Shearer; the same proof applies for general families $\x$ (see~\cite[Lemma~7]{balogh2006remark}).

\subsection{The monochromatic Erd\H{o}s-Rothschild problem}
\label{sec:ERintro}
The problem of determining $F(n;\x)$ was first posed by Erd\H{o}s and Rothschild in 1974~\cite{erdos1974some}, in the case when $\x$ is the family $\x_{k,s}^{(1)}$ of the $s$ monochromatic colourings of $K_k$.
They conjectured that the trivial lower bound is tight for $(k,s)=(3,2)$.
This was verified by Yuster~\cite{yuster1996number} for all $n \geq 6$ and by Alon, Balogh, Keevash and Sudakov~\cite{alon2004number}, the latter of whom showed that in fact for both $s=2,3$ and all $k \geq 3$ (and large $n$), the trivial lower bound is tight, but for $s \geq 4$ it is very far from tight. It was shown in \cite[Theorem 1.2]{alon2004number} that for $s=4$, the complete 4-partite graph $T_4(n)$ has exponentially more colourings free of monochromatic triangles than $T_2(n)$. The construction can naturally be extended to $s\ge 4$ colours and inserted in place of three vertex classes of $T_{k-1}(n)$ to show that $F(n;\x_{k,s}^{(1)})$ exceeds $s^{t_{k-1}(n)}$ by an exponential factor. 
The authors also showed \cite[Theorem 1.3]{alon2004number} that $\lim_{n\to \infty}\log F(n;\x_{k,s}^{(1)})/\binom{n}{2} \to \left(1-\frac{1}{k-1}\right)\log(s)$ as $\max\{k,s\}\to\infty$. 

For $s \geq 4$, the only cases where $F(n;\x_{k,s}^{(1)})$ has been determined for large $n$ are $(k,s)$ equal to $(3,4)$, $(3,5)$ (asymptotically), $(3,6)$, $(3,7)$ and $(4,4)$~\cite{botler2019maximum,katherine_exact,pikhurko2012maximum}.
There is no conjecture for general $(k,s)$.
In all known exact results, there is a unique $\x_{k,s}^{(1)}$-extremal graph which is a Tur\'an graph, with more than $k-1$ parts for $s \geq 4$. For example, the extremal graph for $(4,4)$ is $T_9(n)$. However, for $(3,5)$ there is a large family of graphs which are \emph{almost} extremal, though we do not know which of these are extremal; nevertheless, these graphs are all complete multipartite.
It is not known whether for all pairs $(k,s)$ every extremal graph is complete multipartite.
 
\subsection{The generalised \ER problem}\label{sec:intro:gen_er}
In this paper we study the \emph{generalised \ER problem} of determining $F(n;\x)$ and the $\x$-extremal graphs for a general family $\x$ of forbidden $s$-edge colourings of $K_k$. We develop a general proof framework for the generalised \ER problem, which constitutes a powerful reduction to an optimisation problem, and give two applications to forbidden $2$-edge coloured triangles and to counting colourings of graphs in which every $k$-clique is properly coloured. 
Existing results on the generalised \ER problem, which we will shortly survey, are sporadic in nature and any sort of comprehensive solution has so far been out of reach.

Balogh~\cite{balogh2006remark} was the first to extend the \ER problem to forbidden non-monochromatic cliques. He considered the variant where $\x$ contains a single colouring $\chi:E(K_k)\to[2]$ 
and showed that, unless $\chi$ is monochromatic, the trivial lower bound is tight. Subsequent results in the literature consider predominantly what we call \emph{symmetric} families $\x$ which are invariant under permutations of the colours. Equivalently, $\x$ is symmetric if there is a family $\mathcal{P}$ of partitions of $E(K_k)$ such that $\x = \bigcup_{P \in \mathcal{P}}\x_P$ where
each $\x_P$ is obtained by assigning all possible disjoint sets of colours to the classes of $P$ (an example is given in \autoref{fig:dichrom}). We refer to such partitions as \emph{colour patterns}.
When $\mathcal{P}$ contains a single colour pattern $P$, we write $(P,s)$ as shorthand for $\x_P$ and say that an $s$-edge colouring of a graph $G$ is \emph{$P$-free} if it is $\x_P$-free. From now on, let $K_k^{(\ell)}$ denote the family of all patterns on $K_k$ with $\ell$ classes. With this notation, the monochromatic pattern $\x_{k,s}^{(1)}$ is $(K_k^{(1)},s)$.

The generalised \ER problem for symmetric $\x$ has been studied in \cite{alon2004number,balogh2006remark, balogh2019typical,  BASTOS2019gallai,bastos2023graphs, benevides2017edge, hoppen2015edge, hoppen2019remarks, hoppen2021extension, hoppen2017rainbow, hoppenstars,hoppen2022,katherine_exact,katherine_stability,pikhurko2012maximum}; known results and the corresponding $\x_{\mc P}$-extremal graphs are summarised in \autoref{table:known}.
Observe that if the number $s \geq 2$ of colours is less than the number of parts in $P$, then the unique $(P,s)$-extremal graph is the complete graph since every $s$-edge colouring of every graph is $P$-free. 
Benevides, Hoppen and Sampaio 
gave a simple proof~\cite[Theorem 1.1]{benevides2017edge}, using Zykov symmetrisation, to show that for every number $s \geq 2$ of colours and colour pattern $P$ with at most $s$ parts, there exists a complete partite extremal graph. It was much less clear whether there can be other extremal graphs, but the authors of~\cite{benevides2017edge} proved that for all but very few exceptional colour patterns, \emph{every} extremal graph is complete partite.

\begin{theorem}[\!\!\cite{alon2004number,balogh2006remark,benevides2017edge,botler2019maximum,katherine_exact}]\label{th:extknown}
Let $s \geq 2$ and $k \geq 3$ and let $P$ be a colour pattern of $K_k$. Suppose that none of the following hold.
\begin{enumerate}
    \item $P=K_k^{(1)}$ is the monochromatic pattern and $s \not\in\{2,3,4,6,7\}$;
    \item $k=3$, $s\ge 4$ and $P=K^{(2)}_3$ is the (unique) pattern on $K_3$ with $2$ colour classes.
\end{enumerate}
Then for sufficiently large $n$, every $n$-vertex $(P,s)$-extremal graph is complete partite. 
\end{theorem}

\begin{figure}[H]
    \centering
\includegraphics[scale=0.5]{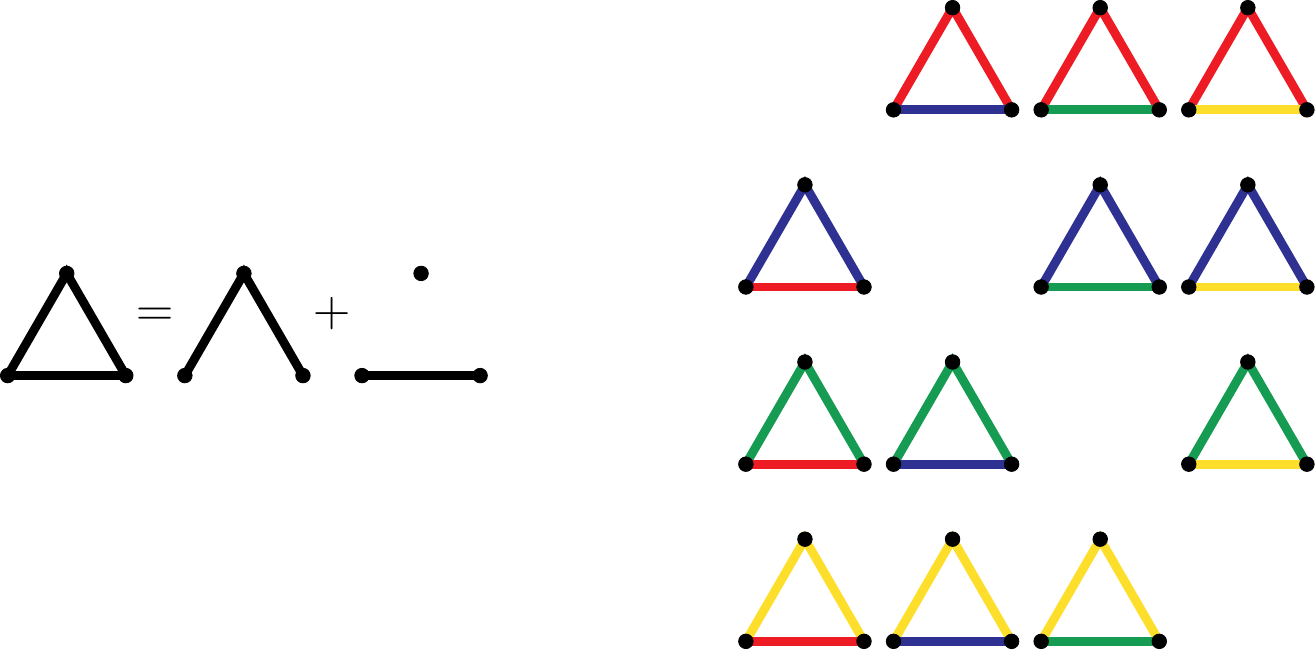}
\caption{The pattern $K^{(2)}_3$ and the family $(K^{(2)}_3,4)$.}
\label{fig:dichrom}
\end{figure}

Note that when the number of parts in $P$ is larger than $s$, the problem is degenerate since every $s$-edge colouring of any graph is valid and therefore (the complete multipartite graph) $K_n$ is the unique $(P,s)$-extremal graph on $n$ vertices. Otherwise, the proof in~\cite{benevides2017edge} works for all non-monochromatic patterns apart from case (2) and a certain pattern on $K_4$ with two colours. It was also shown in \cite{benevides2017edge} that $T_{k-1}(n)$ is the unique $(P,s)$-extremal graph when $s=3$ and $P$ is a colour pattern on $K_k$ with two colour classes such that the graph formed by one colour class has Ramsey number at most $k$. In particular, this gives $T_2(n)$ as the unique extremal graph for the pattern $K_3^{(2)}$ and $s=3$ colours. The other cases covered by \autoref{th:extknown} follow from~\cite{alon2004number,botler2019maximum,katherine_exact} (monochromatic) and~\cite{balogh2006remark} (two colours).

As mentioned in \autoref{sec:ERintro}, there are a handful of other sporadic (monochromatic) cases for which the conclusion of \autoref{th:extknown} holds.
Pikhurko and Staden~\cite{katherine_exact} proved that extremal graphs for the monochromatic case are complete partite, provided the solutions of a certain finite optimisation problem have a particular `extension property'.
As part of our proof framework we show that extremal graphs for the generalised \ER problem are complete partite and described by solutions to the corresponding optimisation problem, provided $\x$ satisfies a strong analogue of the `extension property'. The optimisation problem is formally stated in~\autoref{sec:opt} and our main results linking it to the generalised \ER problem are stated in~\autoref{sec:general}.

The generalised \ER problem for patterns on $K_3$ is the most widely studied so far. A $K^{(3)}_3$-free colouring of a graph (that is, rainbow triangles are forbidden) is also known as a \emph{Gallai colouring},
a term introduced by Gy\'arf\'as and Simonyi~\cite{gyarfas2004edge} following an earlier work of Gallai~\cite{gallai1967transitiv} on comparability graphs.
Balogh and Li~\cite{balogh2019typical} and Hoppen, Lefmann and Odermann~\cite{hoppen2017rainbow} solved the forbidden rainbow triangle problem completely for large $n$.
Their results imply that for $s=3$, the complete graph $K_n$ is the unique extremal graph,
but for $s \geq 4$, $T_2(n)$ is the unique extremal graph.  Hoppen, Lefmann and Schmidt \cite{hoppen2019remarks,hoppen2022} studied the \emph{dichromatic triangle} pattern $K_3^{(2)}$, which is the focus of our first main result. We discuss the dichromatic triangle pattern in more detail in~\autoref{sec:main:dichromatic}.

The generalised \ER problem has been studied for families of colourings that combine multiple patterns (these results are included in \autoref{table:known}). Recall that $K_k^{(\ell)}$ denotes the family of patterns on $K_k$ with exactly $\ell$ colour classes.
Analogously, let $K_k^{(\le\ell)}$ and $K_k^{(\ge\ell)}$ be the families of patterns with at most $\ell$ and at least $\ell$ classes. Hoppen, Lefmann and Nolibos~\cite{hoppen2021extension} showed that for integers $k \geq 4$ and $2 \leq \ell \leq \binom{k}{2}$ and $s > s_0(k,\ell)$, the unique $(K_k^{(\ge \ell)},s)$-extremal graph on $n$ vertices is the trivial example $T_{k-1}(n)$, for large $n$. In particular, for sufficiently large $s$ and every symmetric family containing the rainbow pattern, the trivial lower bound is tight. On the other hand, as noted in~\cite{bastos2023graphs}, a construction by Hoppen, Lefmann and Odermann~\cite[Remark 4.2]{hoppen2017rainbow} implies that for any symmetric family not containing the rainbow pattern, $T_{k-1}(n)$ is not extremal for $s \geq s_2(k)=\binom{k+1}{2}^{k^2} $ and sufficiently large $n$. For the family $K^{(\leq \binom{k}{2}-1)}_k$ of all non-rainbow colourings, their $s_2(k)$ was complemented by Bastos, Hoppen, Lefmann, Oertel and Schmidt~\cite{bastos2023graphs} showing that whenever $s \le s_1(k) \approx (k/2)^{k/2}$, the trivial example $T_{k-1}(n)$ is the unique $(K^{(\leq\binom{k}{2}-1)}_k,s)$-extremal graph for sufficiently large $n$. 

This tells us exactly for which symmetric families the trivial bound is tight for large $s$ and, in line with known cases of the generalised \ER problem, we may conjecture that for any symmetric family $\x$ there is a threshold $s^*=s^*(\x)$ that governs the tightness of the trivial bound for $\x$. That is, we suspect the following dichotomy: when $\x$ does (not) contain the rainbow pattern, the trivial bound is (not) tight if and only if $s>s^*$.
Results for the more interesting `non-trivial' range are particularly scarce. Our two main results which we present in Sections~\ref{sec:main:dichromatic} and \ref{sec:main:improper} show that this dichotomy holds for $K_3^{(2)}$ and when $\x$ is the family of all improper colourings of $K_k$. For both problems, solutions in the non-trivial range exhibit an infinite sequence of phase transitions, giving an infinite family of extremal graphs. This is the first such result in the literature.

\begin{table}
\bgroup
\def\arraystretch{1.3}
\begin{tabular}{|l|ll|l|}
\hline
\def\arraystretch{1}\begin{tabular}{l}Forbidden pattern(s)\\ on $K_k$, $k \geq 3$ \end{tabular}                                                                                                                              & \multicolumn{2}{l|}{Extremal graph on $n$ vertices, $n$ large}                                             & Reference                                                                    \\ \hline
\multirow{2}{*}{\begin{tabular}{l}$K^{(1)}_k$\end{tabular}}                           & \multicolumn{1}{l|}{$s = 2,3$}                  & $T_{k-1}(n)$                             & \multirow{2}{*}{\begin{tabular}{l}\hspace{-0.1cm}\cite{alon2004number}\end{tabular}} \\ \cline{2-3} 
                                                                                                                                                 & \multicolumn{1}{l|}{\cellcolor{Blue!15}$s \geq 4$}             & \cellcolor{Blue!15}non-trivial                          & \\ \hline
        \multirow{4}{*}{\begin{tabular}{l}$K^{(1)}_3$\end{tabular}}                     & \multicolumn{1}{l|}{\cellcolor{Blue!15}$s=4$}             & \cellcolor{Blue!15} \hspace{-0.1cm}$T_4(n)$                               & \cite{pikhurko2012maximum}                                             \\ \cline{2-4} 
                                                            & \multicolumn{1}{l|}{\cellcolor{Blue!15}$s=5$}                  & \cellcolor{Blue!15}\def\arraystretch{1}\begin{tabular}{l}\hspace{-0.1cm}asymptotically optimal family\\\hspace{-0.1cm}of complete partite graphs\\\hspace{-0.1cm}with $4$, $6$ or $8$ parts\end{tabular}     & \multirow{2}{*}{\begin{tabular}{l}\hspace{-0.1cm}\cite{botler2019maximum}\end{tabular}}                                               \\ \cline{2-3} 
                                                                                                                                                   & \multicolumn{1}{l|}{\cellcolor{Blue!15}$s=6$}                  & \cellcolor{Blue!15}$T_8(n)$                          &                                                \\ \cline{2-4} 
                                                            & \multicolumn{1}{l|}{\cellcolor{Blue!15}$s=7$}                  & \cellcolor{Blue!15}$T_8(n)$                          & \cite{katherine_exact}                                     \\ \hline
\begin{tabular}{l}$K^{(1)}_4$\end{tabular}                                          & \multicolumn{1}{l|}{\cellcolor{Blue!15}$s=4$}                  & \cellcolor{Blue!15}$T_9(n)$                          & \cite{pikhurko2012maximum}                                             \\ \hline

\multirow{2}{*}{\begin{tabular}{l}$K^{(2)}_3$\end{tabular}}                           & \multicolumn{1}{l|}{$s \leq 26$}                  & $T_2(n)$                             & \cite{hoppen2019remarks,hoppen2022} \\ \cline{2-4} 
        & \multicolumn{1}{l|}{\cellcolor{Blue!15}$s \geq 27$}             & \cellcolor{Blue!15}$T_r(n)$ for some $r \in \Rev(s)$                          & \hspace{0.1cm}\autoref{th:K32} \\ \hline
\begin{tabular}{l}any family $\subseteq K^{(2)}_k$ \end{tabular}                               & \multicolumn{1}{l|}{$s=2$}                & $T_{k-1}(n)$             & \cite{balogh2006remark}                                    \\ \hline
\begin{tabular}{l}any $P \in K^{(2)}_k$ s.t. some\\ class $J$ has $R(J,J) \leq k$\end{tabular} & \multicolumn{1}{l|}{$s=3$}                  & $T_{k-1}(n)$                      & \cite{benevides2017edge}                                    \\ \hline
\multirow{2}{*}{\begin{tabular}{l}$K^{(3)}_3$\end{tabular}}                           & \multicolumn{1}{l|}{\cellcolor{Blue!15}$s=3$}                  & \cellcolor{Blue!15}$K_n$                             & \cite{balogh2019typical,BASTOS2019gallai} \\ \cline{2-4} 
        & \multicolumn{1}{l|}{$s \geq 4$}             & $T_2(n)$                          & \cite{balogh2019typical,hoppen2017rainbow} \\ \hline
\begin{tabular}{l}any $P \in K^{(\geq 2)}_k$\end{tabular}                               & \multicolumn{1}{l|}{ \cellcolor{Blue!15} \hspace{-0.1cm}$s \geq 2$}                & \cellcolor{Blue!15} \hspace{-0.1cm}complete multipartite             & \cite{balogh2006remark,benevides2017edge}, \autoref{th:ext}                                    \\ \hline
\begin{tabular}{l}$K^{(\geq \ell)}_k$\end{tabular}                                                                                                                    & \multicolumn{1}{l|}{$s > s_0(k,\ell)$} & $T_{k-1}(n)$                      & \cite{hoppen2021extension}                                  \\ \hline
\multirow{2}{*}{\begin{tabular}{l}$K_k^{(\leq\binom{k}{2}-1)}$\end{tabular}}      & \multicolumn{1}{l|}{$s \leq s_1(k)$}        & $T_{k-1}(n)$                      & \cite{hoppen2017rainbow}                                                                            \\ \cline{2-4} 
        & \multicolumn{1}{l|}{\cellcolor{Blue!15}$s \geq s_2(k)$}        & \cellcolor{Blue!15} \hspace{-0.1cm}non-trivial &  \cite{bastos2023graphs}                          \\ \hline
\multirow{2}{*}{\begin{tabular}{l}\{all improper $P$\}\end{tabular}}      & \multicolumn{1}{l|}{$s \leq s(k)$}        & $T_{k-1}(n)$                      & \multirow{2}{*}{\begin{tabular}{l}\hspace{-0.1cm}\autoref{th:proper}\end{tabular}}                                                                              \\ \cline{2-3} 
        & \multicolumn{1}{l|}{\cellcolor{Blue!15}$s > s(k)$}        & \cellcolor{Blue!15}$T_r(n)$ for some $r \in \Rev(s)$ &                                                                    \\ \hline
\end{tabular}
\egroup

\vspace{6pt}
\caption{All results on the generalised \ER problem for symmetric families. Results where the extremal graph is non-trivial are highlighted.}
\label{table:known}
\end{table}

Perhaps surprisingly, the requirement that $n$ is large cannot be removed in general.
Indeed,~\autoref{th:extknown} does not hold for small $n$. In \cite{alon2004number} the authors showed that for $n\le s^{\frac{k-2}{2}}$ the complete graph $K_n$ has more $s$-edge colourings free of monochromatic $K_k$ than $T_{k-1}(n)$. In \cite{han2018improved} this bound on $n$ was shown to be of the correct order for $s=2,3$.

The \ER problem has been studied 
in many other contexts. The results of~\cite{katherine_stability,katherine_exact,katherine_asymptotic} about the monochromatic pattern apply in the more general setting where the size of the forbidden monochromatic cliques depends on the colour.
Patterns on graphs other than the clique $K_k$ have also been studied, e.g.~for
bipartite graphs including stars and matchings~\cite{improvedstars,HOPPEN2014linearturan,hoppen2015edge,hoppen2017rainbow,hoppenstars},
and for hypergraphs~\cite{fanoplane2021,linearhypergraphs,FanoPlaneLefmann,HypergraphEr-LefmannPersonSchacht}.
There are many variants of the problem for different discrete structures, that is, determining the maximum number of colourings of a discrete structure in which certain local substructures are forbidden.
Some examples include
sum-free sets~\cite{schurtriple,sumfree2021}, set systems~\cite{CLEMENS2018EKR, hoppen2012hypergraphs}, linear vector spaces~\cite{CLEMENS2018EKR, HOPPENvector}, and partial orders~\cite{DASchains}.

A notable variant is another question of Erd\H{o}s~\cite{erdos1974some} posed in the same paper as the \ER problem: given an oriented graph $F$, what is the maximum number of $F$-free orientations of an $n$-vertex graph $G$?
This was solved for $F$ a tournament by Alon and Yuster in~\cite{alonorientations}: this is closely related to the $2$-colour monochromatic \ER problem, and the unique extremal graph is the trivial example $T_{v(F)-1}(n)$.
Recently, Buci\'c, Janzer and Sudakov~\cite{bucic2023counting} resolved all remaining cases in an asymptotic sense and fully resolved for large $n$ the case when $F$ is an oriented odd cycle, extending the work of Ara\'ujo, Botler and Mota~\cite{araujo2020counting}.

\section{Main results}\label{sec:main}

\subsection{Forbidding dichromatic triangles}\label{sec:main:dichromatic}

There are exactly three patterns on $K_3$, namely the \emph{monochromatic pattern} $K^{(1)}_3$, the \emph{dichromatic pattern} $K^{(2)}_3$ (shown in \autoref{fig:dichrom}), and the \emph{rainbow pattern} $K^{(3)}_3$. The forbidden monochromatic triangle problem is notoriously difficult and is solved only for $2 \leq s \leq 7$ (asymptotically only for $s=5$)~\cite{botler2019maximum,katherine_exact,pikhurko2012maximum}.
We refer the reader to~\cite{katherine_exact} for more details, where it was speculated that for general $s$, based on the existing results, extremal graphs could be related to Hadamard matrices. Recall that colourings avoiding the rainbow triangle pattern $K_3^{(3)}$ are maximised by $K_n$ for $s=3$ colours and $T_2(n)$ for $s\ge 4$ colours \cite{balogh2019typical,hoppen2017rainbow}.
The first results on the dichromatic pattern were obtained by Hoppen and Lefmann~\cite{hoppen2019remarks} who showed that for $2 \leq s \leq 12$ the trivial lower bound is tight and $T_2(n)$ is the unique extremal graph.
The same authors together with Schmidt~\cite{hoppen2022} showed that the result extends to $s \leq 26$, but not beyond, and conjectured that the extremal graph for $s=27$ is $T_4(n)$.

\begin{theorem}[\!\!\cite{hoppen2019remarks,hoppen2022}]\label{th:26}
For every integer $2 \leq s \leq 26$ and $n$ sufficiently large, 
$$
F(n;(K^{(2)}_3,s)) = s^{t_2(n)}
$$
and the unique $n$-vertex graph with the maximum number of $K^{(2)}_3$-free $s$-edge colourings is $T_2(n)$.
For every integer $s \geq 27$, we have $F(n;(K^{(2)}_3,s)) > s^{t_2(n)}$.
\end{theorem}

The latter part of \autoref{th:26} can be seen by considering the following construction as described in~\cite{hoppen2022} and shown in~\autoref{fig:col_T4}. 
Let $s=27$, let $V_1,\ldots,V_4$ be the vertex classes of $T_4(n)$ and let $M_1,M_2,M_3$ be the unique decomposition of $K_4$ into three perfect matchings. Split the set of 27 colours into three sets $C_1,C_2,C_3$ of size 9. For each edge $ij$ in each matching $M_\ell$, colour the edges in $V_i\times V_j$ arbitrarily with colours from $C_\ell$. This generates $9^{t_4(n)}$ valid colourings of $T_4(n)$, which after taking all colour permutations yields at least $(1-o(1))\binom{27}{9,9,9}9^{t_4(n)}>27^{t_2(n)}$ colourings.
The authors of~\cite{hoppen2022} conjectured that $T_4(n)$ is in fact the unique extremal graph for $s=27$. 

Our first main result is an exact solution of the generalised \ER problem for the dichromatic triangle pattern for all values of $s$, which as a corollary confirms the above conjecture and resolves the non-monochromatic cases excluded from \autoref{th:extknown}. We show that the above construction on $T_4(n)$ is optimal for $27\le s\le 496$, followed by analogous constructions on $T_6(n)$, $T_8(n)$, etc. We now describe these constructions before we introduce relevant notation and state our main result for general $s$.

\begin{figure}[H]
    \centering
    \includegraphics[]{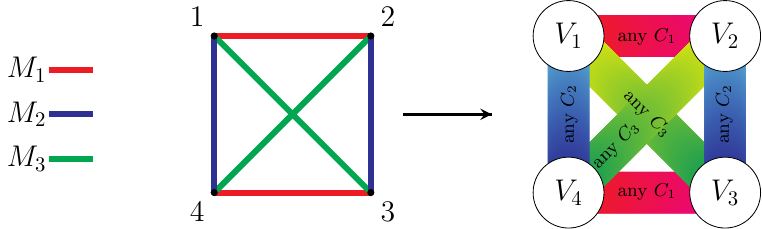}

    \caption{A lower bound on the number of $K_3^{(2)}$-free $s$-edge colourings of $T_4(n)$ is given by splitting $[s]=C_1\cup C_2\cup C_3$ and colouring as shown. 
    }
        \label{fig:col_T4}
\end{figure}

\begin{construction}\label{cons:matchings}
Given $s$ colours and an integer $r$, consider a copy of $T_r(n)$ with vertex classes $V_1,\ldots,V_r$ and the set $\mc M$ of maximum matchings of $K_r$. Note that $\mc M$ will consist of perfect matchings if $r$ is even and matchings of size $(r-1)/2$ if $r$ is odd. Partition the set of colours $[s]$ into classes $\{C_M\}_{M\in \mc M}$. 
For each $ij \in E(K_r)$,
colour the edges in $V_i \times V_j$ with colours from $\bigcup_{M \ni ij}C_M$.
\end{construction}

Every colouring of $T_r(n)$ generated in this way is free of dichromatic triangles as the only monochromatic cherries (paths on three vertices) have both end vertices in the same vertex class, and thus cannot be completed to a triangle. This yields
\begin{equation}\label{eq:intro_construction}
    \prod_{ij \in E(K_r)} \left(\sum_{M\ni ij} |C_M|\right)^{|V_i||V_j|}
\end{equation}
colourings free of dichromatic triangles, which is maximised if the sums $\sum_{M\ni ij} |C_M| $ are all as equal as possible; by double counting this means that the neighbourhood of every vertex is an equipartition of $[s]$, so \eqref{eq:intro_construction} simplifies to
$$\prod_{i\in[r]} \left[\left\lfloor\frac{s}{r-1}\right\rfloor^{r-1-a} \left\lceil\frac{s}{r-1}\right\rceil^a\right]^{\frac{1}{2}(\frac{n}{r})^2+O(n)},$$
where $a$ is the remainder of $s$ when divided by $r-1$.

We are able to show that for each $s\ge 2$, there exists a set $\Rev(s)$ consisting of at most two consecutive even integers, such that \autoref{cons:matchings} is optimal only for $r\in \Rev(s)$.
We define a few key quantities that capture $\Rev(s)$ and the number of colourings generated by \autoref{cons:matchings}.
We write $\log$ for the natural logarithm.

\begin{definition}[$\gf_s(r)$, $\gf(s)$, $\Rev(s)$, $\rev(s)$]\label{def:defs}
Let $s \geq 2$ be the number of colours, and let $2\le r <s$ be an integer.
Let $z$ and $a\in\{0,\ldots,r-2\}$ be the quotient and remainder of $s$ when divided by $(r-1)$; that is, $
 z=\left\lfloor\frac{s}{r-1}\right\rfloor$ and $a=s-(r-1)z.$
Define
$$ 
\gf_s(r) := \frac{r-1-a}{r}\log(z)+\frac{a}{r}\log(z+1),
$$
and
$$
\gf(s) := \max_{\substack{r \in 2\mb{N}\\ r < s}}\gf_s(r),\quad
\Rev(s) := \{r \in 2\mb{N}: \gf_s(r) = \gf(s)\}.
$$
If $s<e^2$ we let $\rev(s):=2$; otherwise
let $\rev(s)$ be the largest even integer $r$ such that $(r-1)e^{r} \leq s$. 
\end{definition}

With this notation, \autoref{cons:matchings} will yield $O_{r,s}(1) \cdot e^{\frac{r}{r-1}\gf_s(r)t_r(n)}$ colourings free of dichromatic triangles. 
(Note that this simplifies to $O_{r,s}(1) \cdot e^{\gf_s(r)n^2/2}$ when $r | n$; otherwise these quantities differ by an exponential factor in $n$.)
By definition of $\Rev(s)$, taking any $r\in \Rev(s)$ yields $O_{r,s}(1) \cdot e^{\frac{r}{r-1}\gf(s)t_r(n)}$ colourings, which we will prove is the maximum for $s\ge 2$, and $T_r(n)$ is the extremal graph. A priori, there is no reason why $\Rev(s)$ should be easy to describe, or finite, but we are able to show that $\Rev(s)\subseteq \{\rev(s),\rev(s)+2\}$. A list of the resulting values of $\Rev(s)$ for $s\le \maxs$ is shown in~\autoref{table:1}. We are now ready to state our first main result.

\begin{table}
\centering
\centering
\begin{tabular}{ l|l }
 $s$ & $\Rev(s)$\\
\hline 
 $[2,26]$ & $\{2\}$  \\ 
 $27$ & $\{2,4\}$  \\
$[28,496]$ & $\{4\}$ \\
$[497,5856]$ & $\{6\}$ \\
$[5857,59470]$ & $\{8\}$ \\
$[59471,559116]$ & $\{10\}$ \\
$[559117, 5015852]$ & $\{12\}$ \\
$[5015853,\ge \maxs]$ & $\{14\}$
\end{tabular}
\vspace{6pt}
\caption{$\Rev(s)$ for $s \leq \maxs$ generated by the script \texttt{optr.py}.}
\label{table:1}
\end{table}

\begin{theorem}\label{th:K32}
For every integer $s \geq 2$ we have $\Rev(s) \subseteq \{\rev(s),\rev(s)+2\}$, so in particular $\gf(s) = \max\{\gf_s(\rev(s)),\gf_s(\rev(s)+2)\}$. 
For sufficiently large $n$, 
every $n$-vertex $(K_3^{(2)},s)$-extremal graph is $T_r(n)$ for some $r\in \Rev(s)$ and
$$F(n;(K^{(2)}_3,s))=(C+o(1))\cdot e^{\frac{r}{r-1}\gf(s)t_r(n)},$$
where $C$ depends only on $s,r$ and on $n ~(\!\!\!\!\mod r)$.\\
Moreover, for all $s \in [2,\maxs] \sm \{27\}$, the unique extremal graph is $T_r(n)$ where $r$ is the unique value in $\Rev(s)$ shown in~\autoref{table:1}, and the unique extremal graph for $s=27$ is $T_4(n)$. 
\end{theorem}

\autoref{th:K32} (together with~\autoref{lm:gapprox}) implies that
$$
\lim_{n \to \infty} \frac{\log F(n;(K^{(2)}_3,s))}{\binom{n}{2}} = g(s) \sim W(s/e)
$$
where $W$ is the Lambert $W$-function, defined to be the inverse of $f(x)=xe^x$,
which is closely related to the definition of $\rev(s)$.

Even though $\gf_{27}(2)=\gf_{27}(4)$, as described above, $T_4(n)$ has more valid colourings than $T_2(n)$, by the multiplicative constant $(1+o(1))\binom{27}{9,9,9} > 10^{12}$.
Thus~\autoref{th:K32} implies \autoref{th:26} from~\cite{hoppen2019remarks,hoppen2022} and confirms the conjecture of Hoppen, Lefmann and Schmidt \cite{hoppen2022} for the case $s=27$.
We conjecture that $\Rev(s)$ is a singleton unless $s=27$ but are unable to prove this.

Prior to our work, the only pattern which had been solved for all $s$ was the rainbow triangle pattern, for which the extremal graph is one of two, depending on $s$ (described above).
By combining~\autoref{th:K32} with further analysis of the set $\Rev(s)$, we show that there is an infinite family of graphs which are extremal for some $s$, in the following strong sense:
\begin{corollary}\label{cor:inf}
    For any $r \in 2\mb{N}$, there are integers $s^-(r) \le s^+(r)$ such that whenever $n$ is sufficiently large, the Tur\'an graph $T_r(n)$ is uniquely $(K^{(2)}_3,s)$-extremal for all $s^-(r) \leq s \le s^+(r)$. Moreover, $s^-(r)\in ((r-3)e^{r-2},(r-1)e^r)$ and 
    there is at most one value of $s$ between $s^+(r)$ and $s^-(r+2)$, for which every $n$-vertex $(K^{(2)}_3,s)$-extremal graph lies in $\{T_r(n),T_{r+2}(n)\}$.
\end{corollary}

Finally, we immediately obtain the following corollary by combining~\autoref{th:K32} with~\autoref{th:extknown}.

\begin{corollary}\label{th:ext}
Let $s \geq 2$ and $k \geq 3$ be integers and let $P$ be a non-monochromatic colour pattern of $K_k$.
Then every $(P,s)$-extremal graph is complete partite.
\end{corollary}

We discuss the proofs of~\autoref{th:K32} and~\autoref{cor:inf} in~\autoref{sec:K23:mainproof}.

\subsection{Forbidding improperly coloured cliques}\label{sec:main:improper}

Closely related to the forbidden triangle pattern problem is the following generalised \ER question on \emph{proper $k$-clique colourings}. Among all graphs $G$ on $n$ vertices, what is the maximum number of $s$-edge colourings of $G$ in which every copy of $K_k$ is properly coloured? Here, an edge colouring is \emph{proper} if every colour class is a matching. Equivalently, a proper $k$-clique colouring is one where two adjacent edges of the same colour must not be contained in any copy of $K_k$. Proper edge colourings of graphs have been studied extensively, in particular through the chromatic index $\chi'(G)$ which is the smallest number of colours necessary for a proper edge colouring of $G$. In e.g.~\cite{lazebnik1989greatest,LohPikhurkoSudakov,ma2015maximizing} the authors consider the related Linial-Wilf problem of finding graphs with a fixed number of vertices and edges having the maximum number of (globally) proper colourings.

Using the language we introduced so far, determining the graphs with the maximum number of proper $k$-clique colourings corresponds to finding $F(n;\x_{k,s}^\wedge)$ and classifying the $\x_{k,s}^{\wedge}$-extremal graphs for $\x_{k,s}^{\wedge}=\{\mbox{all improper }s\mbox{-edge colourings of }K_k\}$.
In particular, when $k=3$, this is equivalent to forbidding both the patterns $K_3^{(2)}$ and $K_3^{(1)}$, and is closely related to the forbidden dichromatic triangle problem. Note that for any fixed graph the sets of $\x_{k,s}^\wedge$-free colourings form an upward chain as we increase $k$.
As a second main result of our paper, we give the maximum number of proper $k$-clique colourings for all $k$ and $s$ and classify all extremal graphs. Note that a proper edge colouring of $K_k$ has at least $k-1$ colours so we only consider the problem in this range.

\begin{theorem}\label{th:proper}
For every $k\ge 3$ there exists an integer $s(k)$ such that for every $s\ge k-1$ and sufficiently large $n$ every $n$-vertex $\x_{k,s}^{\wedge}$-extremal graph is either
\begin{enumerate}
\item[$\bullet$] $T_{k-1}(n)$ if $s\le s(k)$, in which case $F(n;\x_{k,s}^{\wedge}) = s^{t_{k-1}(n)}$; or 
\item[$\bullet$] $T_r(n)$ for some $r\in \Rev(s)$ if $s>s(k)$, in which case 
$F(n;\x_{k,s}^{\wedge}) = (C+o(1))\cdot e^{\frac{r}{r-1}\gf(s)t_{r}(n)}$ where $C$ is a constant depending only on $s,r$ and on $n~(\!\!\!\!\mod r)$.
\end{enumerate}
\end{theorem}

Some precise values and estimates of $s(k)$ for small $k$ are shown in~\autoref{table:improper_patterns}. 

\begin{table}
    \centering
    \bgroup\def\arraystretch{1.3}
    \begin{tabular}{c|c|c|c|c|c|c|c}
         ~~$k$~~&3&4&5&6&7&8&9\\
         \hline
         $s(k)$&26&3124&531440&$\approx 1\cdot 10^8$&$\approx 2.6\cdot 10^{10}$&$\approx 7.5\cdot 10^{12}$&$\approx 2.5\cdot 10^{15}$
    \end{tabular}
    \egroup
    \vspace{6pt}
    \caption{Phase transitions in \autoref{th:proper} for small values of $k$ generated by the script \texttt{improper\_patterns.py} attached as an ancillary file.}
\label{table:improper_patterns}
\end{table}

\autoref{th:proper} is almost a corollary of~\autoref{th:K32}, which explains the similarity in the attained values of $F(n;\x)$. We prove \autoref{th:proper} similarly using the same reduction to an optimisation problem as promised for the proof of \autoref{th:K32}, and it turns out that both forbidden families reduce to optimising a version of \autoref{cons:matchings}. Indeed, note that in \autoref{cons:matchings} every clique is properly coloured.  The reason for this equivalence is somewhat technical and is discussed together with the proof of \autoref{th:proper} in \autoref{sec:proper}.

\subsection{Methods}\label{sec:introproofs}

In the literature, results on the generalised \ER problem are proved by considering a valid colouring of an $n$-vertex extremal graph and approximating it using a coloured version of Szemer\'edi's regularity lemma, or occasionally the hypergraph container method, e.g.~\cite{balogh2019typical,botler2019maximum,han2018improved}.
Then, it suffices to solve a corresponding
problem where edges may be given a set of colours, and every assignment from one of these sets is counted.
In the regularity lemma approach, this is a combinatorial optimisation problem whose feasible solutions are vertex-weighted edge coloured multigraphs which contain no forbidden patterns and whose size does not depend on $n$.
This problem was explicitly stated for monochromatic patterns in~\cite{katherine_asymptotic} but has implicitly appeared in all earlier and most later works.
Though this optimisation problem can be solved by brute force in time $O_{s,k}(1)$, it is very difficult to solve in practice. Most solutions have been obtained via a linear relaxation which only yields feasible solutions in a few cases, hence the sporadic results.
A different optimisation problem arises from the container method (see~\cite{botler2019maximum}); again it appears difficult to solve in general.
 
As mentioned in \autoref{sec:intro:gen_er}, we prove our main results from \autoref{sec:main:dichromatic} and \autoref{sec:main:improper} by first deriving a proof framework for the generalised \ER problem. Our framework is inspired by existing theory for monochromatic patterns developed in~\cite{katherine_asymptotic,katherine_exact, katherine_stability}. The gist of our general results is that the generalised \ER problem reduces to the optimisation problem sketched above. 
In both cases, we are able to solve it for all $s$. For this, we do not use any linear programs, but mainly use analytic tools, as well as combinatorial considerations and some computer assistance for small $s$. 
We hope the ideas in the proof will be useful for solving the generalised \ER problem for other colour patterns.

\subsection{Organisation and notation}

The first part of the paper contains some general theory about forbidden colour patterns in cliques. In \autoref{sec:opt} we introduce the optimisation problem $Q(\x)$ which underpins the paper, and then in \autoref{sec:general} state some general results which show the strong connection between the generalised \ER problem and this optimisation problem. One of these is a new `exact' result which we will use to prove the main results of this paper. 
\autoref{sec:K23} concerns the dichromatic triangle problem, and in it we use the theory developed in the first part of the paper and solve the appropriate instances of $Q(\x)$ to prove our first main result \autoref{th:K32}. \autoref{cor:inf} and~\autoref{th:ext} will follow immediately.
In \autoref{sec:proper} we apply the same methods to the problem of maximising the number of proper $k$-clique colourings and prove our second main result \autoref{th:proper}. The proof of \autoref{th:proper} is closely related to \autoref{th:K32} and will use our tools for optimising a version of \autoref{cons:matchings} when $\x=(K_3^{(2)},s)$. 
We prove the general results of \autoref{sec:general} in \autoref{sec:proofs}. These proofs are adaptations of analogous results in~\cite{katherine_asymptotic,katherine_exact,katherine_stability} so we provide sketches only.
\autoref{sec:conclude} contains some concluding remarks on the generalised \ER problem and further avenues for investigation.

\medskip
In this paper, $\log$ always denotes the natural logarithm $\log_e$ (this is in contrast to many other papers on this topic, but useful here as we make extensive use of calculus).
We also set $\log 0 := 0$ to ease notation, in particular, given a multiset $\mc{A}$ of sets $A$, we write $\sum_{A \in \mc{A}}\log(|A|)$ instead of $\sum_{A \in \mc{A}: A \neq \emptyset}\log(|A|)$. We write $x=a\pm b$ as short-hand for $x\in[a-b,a+b]$. The $L^1$-norm of a vector $\bm{x}$ is $\|\bm{x}\|_1=\sum_{i}|x_i|$. 
We write $2\mb{N}$ to denote the set of positive even integers. 
We write $0 < a \ll b \ll c < 1$ to mean that we can choose the constants $a, b, c$ from right to left.
More precisely, there exist non-decreasing functions $f : (0, 1] \to (0, 1]$ and $g : (0, 1] \to (0, 1]$ such
that for all $a \leq f (b)$ and $b \leq  g(c)$ our calculations and arguments in our proofs are correct. Larger
hierarchies are defined similarly. 
We suppress rounding where it does not affect the validity of the argument.
Most of the specific notation will be introduced in the next section.

\section{An optimisation problem}\label{sec:opt}

In~\cite{katherine_asymptotic}, Pikhurko, Staden and Yilma showed that the \ER problem is asymptotically solved by complete partite colouring constructions such as the one shown in \autoref{fig:col_T4}. That is, it is enough to maximise the number of colourings of a complete multipartite graph with parts $V_1,\ldots,V_r$ such that each edge in $V_i\times V_j$ is coloured arbitrarily from a list of colours $\phi(ij)$. This can be stated more precisely as an optimisation problem with parameters $r$, the sizes of the $V_i$, and $\phi$. They also proved a stability result which states that any almost extremal graph for the \ER problem resembles a solution to this optimisation problem. Pikhurko and Staden then showed that, under certain conditions on solutions to the optimisation problem, a stronger stability theorem holds~\cite{katherine_stability}, and also an exact result~\cite{katherine_exact}.
They solved a new case of the \ER problem by solving the corresponding optimisation problem.

An analogous optimisation problem, Problem $Q(\x)$, can be formulated for the generalised \ER problem for a family $\x$ of forbidden colourings. As we mentioned before, we consider graphs with a partite structure, and count only the $\x$-free colourings such that the colours of edges are determined entirely by the parts they lie between. Thus these colourings are defined by a function $\phi$ which maps pairs of parts to sets of colours. 
This function $\phi$ must be such that every colouring generated by $\phi$ will be $\x$-free.

So suppose we have $r \in \mb{N}$, which will be the number of parts $V_1,\ldots,V_r$, and a function $\phi : \binom{[r]}{2} \to 2^{[s]}$, which maps pairs of vertices to sets of colours.
Given a colouring $\sS:E(K_k)\to[s]$, we say that $\phi$ is \emph{$\sS$-free} if there is no injective map $\psi:[k]\to[r]$ with $\sS(ij) \in \phi(\psi(i)\psi(j))$ for all $ij\in \binom{[k]}{2}$.
We say that $\phi$ is \emph{$\x$-free} if it is $\sS$-free for every $\sS\in \x$. 

Before we state our optimisation problem, we define
$\Phi_{\x,t}(r)$ to be the set of $\x$-free colourings $\phi: \binom{[r]}{2} \to 2^{[s]}$ such that $|\phi(ij)|\geq t$ for all $i\neq j$. 
We call such $\phi$ \emph{colour templates} and, given $c \in [s]$, also write $\phi^{-1}(c) := \{ij \in \binom{[r]}{2}: c \in \phi(ij)\}$ which can be considered as (the edge-set of) a graph on the vertex set $[r]$.
We write $\phi_{ij} := |\phi(ij)|$ for brevity and often call this the \emph{multiplicity} (of $ij$), which comes from thinking of $\phi$ as the multigraph made up of colour graphs $\phi^{-1}(1),\ldots,\phi^{-1}(s)$.
We also define $\DD^r$ to be the set of non-negative vectors $\ba=(\aA_1,\ldots,\aA_r)$ such that $\sum_{i\in[r]} \alpha_i=1$.
By convention, given $\ba_j \in \DD^r$, we write $\ba_{j}=(\alpha_{j,1},\ldots,\alpha_{j,r})$.

\medskip
\noindent\fbox{%
    \parbox{\textwidth}{%
\textbf{Problem $Q_t(\x)$.}

Maximise
$$
q(\phi,\ba) := 2\sum_{ij \in \binom{[r]}{2} 
}\aA_i\aA_j\log\phi_{ij}
$$
subject to
$r \in \mb{N}$, $\phi \in \Phi_{\x,t}(r)$
and $\ba \in \DD^r$.
}}
\medskip

Denote by $\feas_t(\x)$ the set of feasible triples $(r,\phi,\ba)$ for this problem, and by $\opt_t(\x)$ the set of optimal triples. The maximum of the objective function we denote by $Q_t(\x)$. 

For ease of notation, we sometimes omit the subscript $t$ when $t=0$, for example, $Q(\x):=Q_0(\x)$ and $\feas(\x):=\feas_0(\x)$.

We now show how this optimisation problem gives lower bounds on $F(n;\x)$.
Let $(r,\phi,\ba) \in \feas_0(\x)$.
Let $G$ be an $n$-vertex graph with parts $V_1,\ldots,V_r$, where $|V_i|\in\{\lfloor\aA_i n \rfloor,\lceil\aA_i n\rceil\}$, such that $G[V_i]$ is empty for all $i$; if $\phi(ij)=\emptyset$ for some $ij\in \binom{[r]}{2}$ then $G[V_i,V_j]$ is empty, and otherwise it is a complete bipartite graph.
Consider the set of all $s$-edge colourings $\chi$ of $G$ where, for every 
$ij \in \binom{[r]}{2}$
and every $xy \in G[V_i,V_j]$, we have $\chi(xy) \in \phi(ij)$.
Every such $\chi$ is $\x$-free since $\phi$ is $\x$-free.
Counting the number of such colourings, we see that
$$
F(n;\x) \geq F(G;\x) \geq \prod_{ij \in \binom{[r]}{2}}\phi_{ij}^{|V_i||V_j|}.
$$
Thus
\begin{equation}\label{eq:lowerbd}
\log F(n;\x) \geq \sum_{ij \in \binom{[r]}{2}}|V_i||V_j|\log\phi_{ij} = q(\phi,\ba)\binom{n}{2}+O(n)
\end{equation}
and therefore we have the lower bound
\begin{equation}
F(n;\x) \geq e^{Q_0(\x)\binom{n}{2}+O(n)}.\label{eq:lb}
\end{equation}
Note that
$$
\textstyle \frac{k-2}{k-1}\log(s) \leq Q_0(\x) \leq \log(s),
$$
which should be compared to our bounds in~\autoref{eq:trivial} for graphs.
Here, the upper bound comes from $2\sum \alpha_i\alpha_j=1-\sum \alpha_i^2 \leq 1$ and $\phi_{ij} \leq s$,
while the lower bound -- the trivial bound -- comes from taking $r=k-1$, $\ba$ uniform and $\phi \equiv [s]$.

A central result of our paper, presented in \autoref{sec:general}, states that, under certain conditions, the bound in \eqref{eq:lb} is tight, thus reducing finding $\log(F(n;\x))/\binom{n}{2}$ asymptotically to finding $Q_0(\x)$.

The set $\opt_0(\x)$ of solutions is rather degenerate, because we can extend every $(r,\phi,\ba)\in \opt_0(\x)$ by adding a new part $r+1$ with $\alpha_{r+1}=0$, or by splitting any part into two. It will be helpful to consider the set of solutions in which parts cannot be merged or deleted to obtain a smaller optimal solution. 
We define the set of \emph{basic optimal solutions} $\opt^*(\x)$ to be the set of $(r^*,\phi^*,\ba^*)\in \opt_2(\x)$ such that $\aA_i^*>0$ for all $i \in [r^*]$.
We now show that the set $\opt_0(\x)$ contains a basic optimal solution. The following fact also justifies our omission of subscripts when referring to $Q(\x)$.

\begin{fact}\label{fact:opts}
Let $s \geq 2$ and $k \geq 3$ be integers and let $\x$ be a family of $s$-edge colourings of $K_k$.
If $\opt_0(\x) \neq \emptyset$, then $\opt^*(\x) \neq \emptyset$. Moreover, $Q_{2}(\x) = Q_{1}(\x) = Q_0(\x)$ and $\opt_{2}(\x)\subseteq \opt_{1}(\x)\subseteq \opt_0(\x)$.
\end{fact}

\begin{proof}
    Let $(r,\phi,\ba)\in \opt_0(\x)$ be an optimal solution with minimum number of parts $r$. It holds that $\alpha_i>0$ for all $i\in [r]$. Indeed, otherwise we would obtain another optimal solution with fewer parts. Let us, for the sake of contradiction, assume that $(r,\phi,\ba) \not\in \opt_{2}(\x)$, say, $\phi_{12}\leq 1$. For any $\eps$ with $-\alpha_2 \leq \eps \leq \alpha_1$ define $\tilde{\alpha_1}=\alpha_1-\varepsilon$, $\tilde{\alpha_2}=\alpha_2+\varepsilon$, and $\tilde{\alpha_i}=\alpha_i$ for all other $i$. 
    (This is a version of `Zykov symmetrisation', see~\autoref{sec:proofs}.)
    Then $q(\phi,\tilde{\ba})= Q(\x)- \eps\sum_{i=3 }^r \alpha_i (\log \phi_{1i}-\log\phi_{2i})$ since the contribution $\alpha_1\alpha_2\log \phi_{12}$ is zero. By optimality of $(r,\phi,\ba)$ we must have that $\sum_{i=3}^r \alpha_i (\log \phi_{1i}-\log\phi_{2i})=0$, for otherwise either a positive or negative $\eps$ would yield a better solution. Thus setting $\eps=\alpha_1$ and removing part $1$ (of size $\tilde\alpha_1$) yields a solution $(r-1,\phi|_{[2,r]},\tilde\ba)\in \opt_0(\x)$ with fewer parts, a contradiction to the minimality of $r$.

    For the second part, as $\feas_{2}(\x)\subseteq \feas_{1}(\x)\subseteq \feas_{0}(\x)$ we have $Q_{2}(\x) \le Q_{1}(\x) \le Q_{0}(\x)$. However, from our proof above, $q(\phi,\ba)=Q_0(\x)=Q_{2}(\x)$, so we have equality throughout. Since the optimum is the same for $t=0,1,2$, we also have the inclusion $\opt_{2}(\x)\subseteq \opt_1(\x)\subseteq \opt_0(\x)$.
\end{proof}

\begin{fact}\label{fact:max}
Let $s \geq 2$ and $k \geq 3$ be integers and let $\x$ be a family of $s$-edge colourings of $K_k$.
Then for all $(r^*,\phi^*,\ba^*) \in \opt^*(\x)$, the colour template $\phi^*$ is maximal. That is, if one obtains a new colour template $\phi$ by adding to $\phi^*$ a new colour in $[s]$ between any pair $ij \in \binom{[r^*]}{2}$, then $\phi$ is not $\x$-free.
\end{fact}

\begin{proof}
    Suppose that there is some $ij \in \binom{[r^*]}{2}$ and $c \in [s]$ such that $\phi$ as described above is $\x$-free.
    Then $(r^*,\phi,\ba^*) \in \feas_{2}(\x)$ and $q(\phi,\ba^*) - q(\phi^*,\ba^*) = \alpha^*_i\alpha^*_j(\log\phi_{ij}-\log\phi^*_{ij}) \geq \log(\frac{s}{s-1})\alpha^*_i\alpha^*_j > 0$, a contradiction to the optimality of $(r^*,\phi^*,\ba^*)$. 
\end{proof}

\subsection{Contributions}

Given $(r,\phi,\ba)\in\feas(\x)$ and $i \in [r]$, we define 
$$
q_i(\phi,\ba) := \sum_{j \in [r]\sm\{i\}}\aA_j\log\phi_{ij}
$$
and refer to $q_i$ as the \emph{contribution of part $i$}, as $q(\phi,\ba)=\sum_{i\in[r]} \alpha_i q_i(\phi,\ba)$.

The following lemma, a version of~\cite[Proposition~2.1]{katherine_stability},  is crucial in our proof of~\autoref{th:K32} as well as in the other general results of the next section.
It states that every vertex in an optimal solution contributes optimally to $q$. Indeed, if two vertices had differing contributions, we can move weight from one to the other to increase $q$, contradicting optimality.
This can be proved via the method of Lagrange multipliers. 

\begin{lemma}\label{lm:cont1}
    Let $s \geq 2$ and $k \geq 3$ be integers and let $\x$ be a family of $s$-edge colourings of $K_k$.
    Then for any $(r,\phi,\ba)\in\opt(\x)$, we have
    $$
    q_i(\phi,\ba)=Q(\x)
    \quad\text{whenever }\aA_i>0.
    $$
\end{lemma}

\subsection{Properties of families and solutions}

Now we introduce some properties of $\x$ -- or, more specifically, the set $\opt^*(\x)$
of basic optimal solutions of Problem $Q_2(\x)$ -- which are required to state our general results.
Given a description of $\opt^*(\x)$, these properties all tend to be easy to check.

Given $r \in \mathbb{N}$ and $\phi \in \Phi_{\x,0}(r)$, we say that $i \in [r]$ is
\begin{itemize}
    \item a \emph{clone of $j \in [r]\setminus \{ i \}$} \emph{(under $\phi$)} if $\phi(i\ell) = \phi(j\ell)$ for all $\ell \in [r] \setminus \{ i,j\}$ and $\phi_{ij} \leq 1$.
    \item a \emph{strong clone of $j$ (under $\phi$)} if $i$ is a clone of $j$ and $\phi_{ij} = 0$.
\end{itemize}
For $(r, \phi, \ba) \in \feas(\x)$ and $\phi' \in \Phi_\x(r+1)$ such that $\phi'|_{\binom{[r]}{2}} = \phi$, we define $${\rm ext}(\phi',\ba) := q_{r+1}(\phi',(\alpha_1,\ldots,\alpha_{r},0)) = \sum_{i \in [r]}\alpha_i\log|\phi'(\{ i,r+1 \})|.$$

\begin{definition}[Bounded, (strong) extension property, stable inside, hermetic]\label{extprop}
Let $s \geq 2$ and $k \geq 3$ be integers and let $\x$ be a family of $s$-colourings of $K_k$. We say that $\x$
\begin{itemize}
\item is \emph{bounded} if $\opt^*(\x)$ is non-empty and there is some $R>0$ such that $r^* \leq R$ for all $(r^*,\phi^*,\ba^*)\in\opt^*(\x)$,
    \item
has the \emph{(strong) extension property} if for any $(r^*,\phi^*,\ba^*) \in \opt^*(\x)$ and $\phi \in \Phi_\x(r^*+1)$ with $\phi|_{\binom{[r^*]}{2}} = \phi^*$ and
${\rm ext}(\phi,\ba^*) = Q(\x)$, there exists $j \in [r^*]$ such that $r^*+1$ is a (strong) clone of $j$ under $\phi$, 
\item is \emph{stable inside} if for every $(r^*,\phi^*,\ba^*) \in \opt^*(\x)$ the following holds: if we make a non-strong clone of a vertex $x \in [r^*]$ 
(i.e.~define $\phi'$ on $r^*+1$ parts by setting $\phi'(\{r^*+1,i\})=\phi(xi)$ for all $i \in [r^*]\setminus\{x\}$ and $\phi'(\{r^*+1,x\})=\{c\}$ for some $c \in [s]$), then there is a forbidden configuration (i.e.\ $\phi' \notin \Phi_\x(r^*+1)$),
\item is \emph{hermetic} if, 
for all $(r^*,\phi^*,\ba^*) \in \opt^*(\x)$, there is no $\phi\in \Phi_{\x,1}(r^*+1)$ such that $\phi|_{\binom{[r^*]}{2}}=\phi^*$.
\end{itemize}
\end{definition}

Observe that (i) if $\x$ is hermetic, then it is stable inside, since adding a non-strong clone to a basic optimal solution is a specific extension of that basic optimal solution where the new vertex sends at least one colour to all existing vertices;
and (ii) if $\x$ has the extension property and is hermetic, then it has the strong extension property, since being hermetic implies that a non-strong clone cannot have optimal contribution.

\subsection{Examples and specific cases}

In this section we derive properties of $\opt^*(\x)$ for some important patterns $P$ and resulting symmetric families $(P,s)$ of $s$-edge colourings of $K_k$.

\subsubsection{Forbidden monochromatic cliques: $P=K^{(1)}_k$}\label{sec:optk1}
For every $(r,\phi,\ba) \in \feas_1(K^{(1)}_k,s)$, we have $r<R_s(k)$, the $s$-colour Ramsey number of $K_k$.
Thus $\feas_1(K^{(1)}_k,s)$ is compact, and, since $q$ is continuous, we have that $\opt_1(K^{(1)}_k,s) \neq \emptyset$.
In particular, by~\autoref{fact:opts}, $(K^{(1)}_k,s)$ is bounded.
\autoref{fact:max} implies that for $(r^*,\phi^*,\ba^*) \in \opt^*(K^{(1)}_k,s)$, for every colour $c \in [s]$, the graph ${\phi^*}^{-1}(c)$ is maximally $K_k$-free.
Every known exact result for the (monochromatic) \ER problem can be deduced via auxiliary results from the corresponding optimisation problem; in all cases, every ${\phi^*}^{-1}(c)$ is a $(k-1)$-partite Tur\'an graph and $\bm{a}^*$ is uniform.

\subsubsection{Forbidden dichromatic triangles: $P=K^{(2)}_3$}\label{sec:opt32}

\begin{lemma}\label{lm:match}
Let $s\ge 2$ be an integer. Then the following hold.
\begin{enumerate}[label=\emph{(\roman*)}]
    \item For every $(r,\phi,\ba) \in \feas_{2}(K^{(2)}_3,s)$, we have that $\phi^{-1}(c)$ is a matching for all $c \in [s]$.
    In particular, $\sum_{j \in [r]\sm\{i\}}\phi_{ij} \leq s$ for all $i \in [r]$.
    \item For every $(r^*,\phi^*,\ba^*) \in \opt^*(K^{(2)}_3,s)$, we have that ${\phi^*}^{-1}(c)$ is a maximal matching for all $c \in [s]$.
    \item $(K^{(2)}_3,s)$ is bounded.
\end{enumerate}
\end{lemma}

\begin{proof}
We first prove~(i). Let $(r,\phi,\ba) \in \feas_{2}(K^{(2)}_3,s)$. If $\phi^{-1}(c)$ is not a matching for some $c\in [s]$, then there exist two adjacent edges, say $hi$ and $hj$, such that $c\in \phi(hi)$ and $c\in \phi(hj)$. Since all edges receive at least two colours, the edge $ij$ receives at least one colour different from $c$. The multicoloured graph induced on the vertices $h,i,$ and $j$ contradicts the $K^{(2)}_3$-free property of the function $\phi$.
The last part follows from the identity $\sum_{j \in [r]\sm\{i\}}\phi_{ij} = \sum_{c \in [s]}d_{\phi^{-1}(c)}(i)$.

Part~(ii) follows from part~(i) and \autoref{fact:max}.

For~(iii), let $(r,\phi,\ba) \in \feas_{2}(K^{(2)}_3,s)$.
Similarly to~\autoref{sec:optk1}, it suffices to show that there is some $R$ such that $r<R$ for every $(r,\phi,\ba) \in \feas_{2}(K^{(2)}_3,s)$.
    As noted, $\phi^{-1}(c)$ is a matching for each $c \in [s]$, therefore it contains at most $r/2$ edges. Since $2\binom{r}{2}/\left(\frac r2 \right)= 2(r-1)$ and together the matchings are required to cover the edges of $K_{r}$ at least twice over, we have that $s \geq 2(r-1)$. In other words, $r \leq s/2+1$.
\end{proof}

\subsubsection{Forbidden $2$-edge colourings}

In the case of $s=2$, the fact that we may restrict to $\phi$ with $\phi_{ij}\ge 2$ for all $ij$ (see \autoref{fact:opts}) means we can find $\opt^*(\x)$ for all $\mc X$.

\begin{lemma}
Let $k \geq 3$ be an integer and let $\x$ be a family of $2$-edge colourings of $K_k$. Then $(r^*,\phi^*,\ba^*) \in \opt^*(\x)$ if and only if $r^*=k-1$, $\phi^*(ij)=[2]$ for all $ij \in \binom{[r^*]}{2}$, and $\ba^*$ is uniform.
\end{lemma}

\begin{proof}
    Take some $(r^*,\phi^*,\ba^*) \in \opt^*(\x)$. Elements of $\opt^*(\x)$ see at least two colours everywhere so $\phi^*(ij)=[2]$ for all $ij$. This contains any pattern on $r^*$ vertices so we must have $r^*\le k-1$. Then $q(\phi^*,\ba^*)=2\log(2)\sum_{ij}\alpha^*_i\alpha^*_j$ which for fixed $r^*$ is maximised by uniform $\ba^*$. Indeed, if some $\alpha^*_i>\alpha^*_j$, perturbing $\alpha^*_i$ down by a sufficiently small $\eps$ and $\alpha^*_j$ up by $\eps$ yields an increase proportional to $\eps(\alpha^*_i-\alpha^*_j-\eps)>0$ in $q$, contradicting optimality. Therefore $\ba^*$ is uniform and $q(\phi^*,\ba^*)= \log(2)\frac{r^*-1}{r^*}$. This function is increasing in $r^*$ so it is maximised by $r^*=k-1$.
\end{proof}

\subsubsection{Forbidden rainbow cliques: $P=K^{(\binom{k}{2})}_k$}\label{sec:rainbowclique}
Let $k \geq 3$ and $s \geq \binom{k}{2}$.
Now, for $(r,\phi,\ba) \in \feas_1(\x)$, the number $r$ of parts can be arbitrarily large since any solution with $\phi$ using less than $\binom{k}{2}$ colours is feasible.
Note that this does not imply $(K^{(\binom{k}{2})}_k,s)$ is not bounded.

\begin{lemma}\label{lm:rainbowclique}
    Let $s \geq 3$ be an integer. 
    Let $\x = (K^{(3)}_3,s)$.
    If $s=3$, then $\opt^*(\x)=\emptyset$, so in particular, $\x$ is not bounded. If $s \geq 4$, then $\opt^*(\x)$ contains the unique element $(2,\phi^*,(\frac{1}{2},\frac{1}{2}))$ where $\phi^* \equiv [s]$, so in particular, $\x$ is bounded.
\end{lemma}

\begin{proof}
Consider the feasible solution $(r,\phi,\ba)$ such that there is $A \subseteq [s]$ of size two for which $\phi(ij)=A$ for all pairs $ij$ in $[r]$.
By convexity, $\ba$ is uniform (otherwise $q(\phi,\ba)$ can be increased), and we have $q(\phi,\ba) = (1-\frac{1}{r})\log(2)$. However, we can find a better solution simply by adding one to $r$, so $(r,\phi,\ba)$ is not optimal, and we obtain a sequence of feasible solutions whose $q$ value tends to $\log(2)$ from below.
Thus $Q^*(\x) \geq \log(2)$.

Suppose that $\opt^*(\x) \neq \emptyset$.
Let $(r^*,\phi^*,\ba^*) \in \opt^*(\x)$. By the above, at least three colours appear in $\phi^*$.

Suppose first that they never appear on a single pair. Then there are distinct $h,i,j \in [r^*]$ such that, without loss of generality, $\phi^*(hi)=\{1,2\}$ and $\phi^*(hj)=\{1,3\}$.
But then $\phi^*(ij)=\emptyset$, a contradiction.

Thus there is some $ij$ with $|\phi^*(ij)| \geq 3$. If there is any other $h \in [r^*] \sm \{i,j\}$, since there are at least two colours on $hi$ and $hj$, we see there is a rainbow triangle on $h,i,j$, a contradiction. Thus $r=2$ and, optimising, $\phi^*(12)=[s]$ and $\ba^*$ is uniform, with $Q^*(\x) = q(\phi^*,\ba^*)=\frac{1}{2}\log(s)$.
So $\frac{1}{2}\log(s) \geq \log(2)$.
This implies that $s \geq 4$.
\end{proof}

This example suggests that we should expand our set of feasible solutions to allow colours inside parts, that is, expand the domain of $\phi$ to both pairs and singletons. Then the above argument would show that every optimal solution for the rainbow triangle problem with $2$ or $3$ colours corresponds to a $2$-edge coloured clique, that is, a one-part solution where the part receives two colours only. This aligns with the results in \cite{balogh2019typical}, which show that the maximum number of $3$-colour Gallai colourings is attained by $K_n$ and is equal to $(3+o(1))2^{\binom{n}{2}}$. Many of our general results can be formulated in this setting but we chose not to do this, as in our two applications, giving colours to singletons is not necessary. We do however discuss this extension in \autoref{sec:conclude}.

\section{Results for general colour patterns}\label{sec:general}

This section contains some general results for bounded families $\x$ of forbidden edge colourings. 
Their culmination is~\autoref{th:exact1}, an `exact' result for hermetic families with the extension property. We prove our main result,~\autoref{th:K32}, by combining this with the solution to the optimisation problem for dichromatic triangles, which we obtain in~\autoref{sec:K23}.

\subsection{Asymptotic upper bound}
We saw in \autoref{sec:opt} that optimal solutions of $Q(\x)$ give a lower bound for $F(n;\x)$ (see~(\ref{eq:lb})). Using Szemer\'edi's regularity lemma, one can prove that there is a matching upper bound.

\begin{theorem}\label{th:asymptotic}
Let $s \geq 2$  and $k \geq 3$ be integers and let $\x$ be a family of $s$-edge colourings of $K_k$
for which $\opt(\x) \neq \emptyset$. Then
$$
F(n;\x) = e^{Q(\x)\binom{n}{2}+o(n^2)}.
$$
\end{theorem}

This result and its proof has appeared for specific colour patterns many times in the literature, starting with~\cite{alon2004number} which considered the monochromatic pattern $K^{(1)}_k$ and $s \in \{2,3\}$ colours.
A general monochromatic version appeared in~\cite{katherine_asymptotic}.
The corresponding result for general colour patterns is no harder, so we don't prove it separately. A proof is obtained en route to proving \autoref{th:stability_graphs} (see \autoref{sec:proofs:stability}).

\subsection{Stability for optimal solutions to \texorpdfstring{$Q(\x)$}{}}

The next key theorem is a generalisation of the main result of~\cite{katherine_stability} which is for monochromatic patterns.
It is a `stability' theorem which states that every almost optimal feasible solution to Problem $Q(\x)$ has a similar structure to a basic optimal solution.

\begin{theorem}\label{lm:stability_opt}
Let $s \geq 2$ and $k \geq 3$ be integers and let $\x$ be a bounded
    family of $s$-edge colourings of $K_k$ which  has the extension property.
Let $\nu>0$. Then there exists $\eps>0$ such that for every $(r,\phi,\ba) \in \feas(\x)$ with
$$
q(\phi,\ba)>Q(\x)-\eps,
$$
there exist $(r^*,\phi^*,\ba^*) \in \opt^*(\x)$ and a partition $[r]=Y_0 \cup \ldots \cup Y_{r^*}$ such that, defining $\beta_i=\sum_{i' \in Y_i}\alpha_{i'}$ for all $i \in [r^*]$, the following hold.
\begin{enumerate}[label={\rm (SO\arabic*)}]
\item\label{SOone} $\|\bm{\beta}-\ba^*\|_1 < \nu$. (In particular, $\sum_{i' \in Y_0}\alpha_{i'}<\nu$.)
\item\label{SOtwo} For all $ij \in \binom{[r^*]}{2}$, $i' \in Y_i$ and $j' \in Y_j$, we have that $\phi(i'j') \subseteq \phi^*(ij)$.
\item\label{SOthree} For all $i \in [r^*]$ there is a colour $c_i \in [s]$ such that for every $i'j' \in \binom{Y_i}{2}$, we have $\phi(i'j') \subseteq \{c_i\}$. If the pattern is stable inside, then for all $i \in [r^*]$ and every $i'j' \in \binom{Y_i}{2}$, we have $\phi(i'j')= \emptyset$.
\end{enumerate}
\end{theorem}

The proof follows~\cite{katherine_stability} closely, so we provide a sketch in~\autoref{sec:proofs:stability_opt}.

\subsection{Stability for graphs}\label{sec:stabgraphs}

The next general result states that almost optimal graphs have a very similar structure to the blow-up of an optimal solution (and almost all valid colourings follow an optimal colour template).
Again, a monochromatic version appears in~\cite{katherine_stability} and the proof is very similar, so we sketch it in~\autoref{sec:proofs:stability}.

Given a graph $G$, disjoint $A,B \subseteq V(G)$ 
and $0 \leq d \leq 1$, we say that $G[A,B]$ is \emph{$(\delta,d)$-regular} if $d_G(A,B) := e_G(A,B)|A|^{-1}|B|^{-1} \in (d-\delta,d+\delta)$, and $|d_G(X,Y)-d_G(A,B)|<\delta$ for all $X\subseteq A$, $Y \subseteq B$ with $|X|/|A|, |Y|/|B| \geq \delta$.
We suppress the subscript $G$ when it is clear from the context.

\begin{theorem}\label{th:stability_graphs}
Let $s \geq 2$ and $k \geq 3$ be integers and let $\x$ be a bounded
    family of $s$-edge colourings of $K_k$ which has the extension property.
Then for all $\delta>0$ there exist $n_0 \in \mathbb{N}$ and $\eps>0$ such that the following holds.
If $G$ is a graph on $n \geq n_0$ vertices such that
$$
\frac{\log F(G;\x)}{\binom{n}{2}} \geq Q(\x)-\eps,
$$
then for at least $(1-e^{-\eps n^2})F(G;\x)$ valid colourings $\chi: E(G) \to [s]$ of $G$ there exist $(r^*,\phi^*,\ba^*) \in \opt^*(\x)$ and a partition $V_1 \cup \ldots \cup V_{r^*}$ of $V(G)$ such that:
\begin{enumerate}[label={\rm(SG\arabic*)}]
    \item\label{SGone} For all $i \in [r^*]$ we have $||V_i|-\alpha^*_i n| < 1$.
    \item\label{SGtwo} For all $ij \in \binom{[r^*]}{2}$ and $c \in \phi^*(ij)$, we have that $\chi^{-1}(c)[V_i,V_j]$ is $(\delta,|\phi^*(ij)|^{-1})$-regular. 
    \item\label{SGthree} For each $i \in [r^*]$ there is a colour $c_i \in [s]$ such that $\sum_{i\in[r^*]} |E(G[V_i])\backslash \chi^{-1}(c_i)|\le \delta n^2$. Moreover, if $\x$ is stable inside, then $\sum_{i\in[r^*]}e(G[V_i]) \leq \delta n^2$.
\end{enumerate}
\end{theorem}

Note that~\ref{SGtwo} implies that all but at most an $s\delta$ proportion of pairs in $(V_i,V_j)$ are edges in $G$, and given a colour in $\phi^*(ij)$ by $\chi$.

\subsection{An exact result}

Our final general result is `exact', in the sense that it guarantees that every extremal graph is a complete partite graph with part sizes very close to some optimal solution,
and moreover, almost all valid colourings follow an optimal colour template perfectly.
One of the main results of~\cite{katherine_exact} is a monochromatic version that assumes the strong extension property.
Here we assume both the extension property and that $\x$ is hermetic.
As observed, these together imply the strong extension property.
Being hermetic is a very strong property that, we will show, is satisfied by the families of interest in this paper, 
but we do not think it holds for any of the other families where the optimisation problem has been solved.
Assuming that the family is hermetic leads to a much simpler proof.
We discuss what might arise from dropping this assumption in \autoref{sec:conclude}.

\begin{theorem}\label{th:exact1}
 Let $s \geq 2$ and $k \geq 3$ be integers and let $\x$ be a bounded, hermetic
    family of $s$-edge colourings of $K_k$ which  has the extension property.
    Then, for all $\dD>0$, there exists $\eps$ with $0<\eps<\dD$ such that whenever $n$ is sufficiently large, the following hold for every $\x$-extremal graph $G$ on $n$ vertices: There is an integer $r^*$ and $\ba^* \in \DD^{r^*}$ such that 
        \begin{enumerate}[label={\rm (SE\arabic*)}]
        \item\label{seone} $G$ is  a complete $r^*$-partite graph whose $i$-th part $W_i$ has size $(\aA^*_i \pm \dD)n$ for all $i \in [r^*]$, and 
        \item\label{setwo} for at least $(1-e^{-\eps n})F(G;\x)$ valid colourings $\chi$ of $G$ there is $\phi^*$ such that
\begin{enumerate}[label={\rm (SE2.\arabic*)}]
        \item\label{se0} $(r^*,\phi^*,\ba^*) \in \opt^*(\x)$.
        \item\label{se1} $\chi^{-1}(c)[W_i,W_j]$ is $(\dD,|\phi^*(ij)|^{-1})$-regular for all $ij \in \binom{[r^*]}{2}$ and $c \in \phi^*(ij)$.
        \item\label{se2} $\chi$ is \emph{perfect}: that is, for all $ij \in \binom{[r^*]}{2}$ and $x \in W_i$ and $y \in W_j$, we have $\chi(xy) \in \phi^*(ij)$.
\end{enumerate}\end{enumerate}
\end{theorem}

Part~\ref{se2} is crucial in determining the extremal graph(s).
This is because, to do so, we need to compare the number of valid colourings in very similar graphs, which all have $r^*$ parts of roughly the same sizes. For this, we need to understand very accurately what most colourings look like. 

The proof of~\autoref{th:exact1} is given in~\autoref{sec:proofs:exact}.

\section{Forbidding dichromatic triangles}\label{sec:K23}

In this section, we apply our general results connecting the generalised \ER problem to the optimisation problem to study a specific forbidden pattern, the dichromatic triangle; that is, the triangle coloured by exactly two colours.
Throughout this section we set $k =3$, the forbidden pattern $P=K^{(2)}_3$ is the unique partition of $K_3$ into two non-empty classes, and
$\x = (P,s)$.
Only $s$ will vary.
Thus we write $Q(s) := Q(\x)$, $\opt(s) := \opt(\x)$, $F(n;s) := F(n;\x)$ and so on, and call an $\x$-extremal graph \emph{$s$-extremal}.
We will prove 
 \autoref{th:K32}, which determines, for $s \geq 2$ and large $n$, the value of $F(n;s)$ up to a multiplicative error of $1+o(1)$, as well as the extremal graphs, which are indeed complete partite, as required to prove~\autoref{th:ext}.
In fact, we show that there are at most two extremal graphs for each $s \geq 2$ and large~$n$.

\subsection{A lower bound construction}\label{sec:lower}

We first recall \autoref{def:defs}.
Let $s \geq 2$ be the number of colours, and let $r\ge 2$ be an integer.
Let $z$ and $a\in\{0,\ldots,r-2\}$ be the quotient and remainder of $s$ when divided by $(r-1)$; that is, $
 z=\left\lfloor\frac{s}{r-1}\right\rfloor$ and $a=s-(r-1)z.$
Define
$$ 
\gf_s(r) := \left(\frac{r-1-a}{r}\right)\log(z)+\left(\frac{a}{r}\right)\log(z+1),\quad
\gf(s) := \max_{r \in 2\mb{N}}\gf_s(r),
\quad\text{and}\quad
\gmax(s) := \max_{r \in \mb{N}}\gf_s(r)
$$
as well as 
$$
\Rev(s) := \{r \in 2\mb{N}: \gf_s(r) = \gf(s)\}
\quad\text{and}\quad
\Rmax(s) := \{r \in \mb{N}: \gf_s(r) = \gmax(s)\}.
$$
Note that $\gf_s(r)=0$ when $r-1\geq s$, and so $\gf(s)$ and $\gmax(s)$ exist (and hence $\Rmax(s)$ and $\Rev(s)$ are non-empty), and this definition agrees with~\autoref{def:defs} where we restricted to $r < s$. 
Recall that we listed $\Rev(s)$ for $s \leq 10^7$ in~\autoref{table:1}.~\autoref{table:2} shows values of $\Rmax(s)$ for $s \leq 90000$.

\begin{table}
\centering
\begin{tabular}{ l|l } 
 $s$ & $\Rmax(s)$\\
\hline 
 $[2,16]$ & $\{2\}$ \\ 
 $[17,76]$ & $\{3\}$   \\
$[77,299]$ & $\{4\}$ \\
$[300,1058]$ & $\{5\}$\\
$[1059,3544]$ & $\{6\}$\\
$[3545,11443]$ & $\{7\}$ \\
$[11444,36023]$ & $\{8\}$ \\
$[36024,\ge 90000]$ & $\{9\}$ 
\end{tabular}
\vspace{6pt}
\caption{$R(s)$ for $s \leq 90000$ generated by the script \texttt{optr.py}.}
\label{table:2}
\end{table}

Let also
$$
\hf_s(x) := \frac{x-1}{x}\log\left(\frac{s}{x-1}\right),\quad\text{ for } x > 1\text{ and }s\in \mathbb{R}^+,
\quad\text{and let}\quad
\hf(s) := \sup_{x >1}\hf_s(x).
$$
The function $\hf_s$ is a fractional approximation of $\gf_s$ obtained by removing the rounding from $s/(r-1)$. So when $s$ is divisible by $r-1$, we have $\hf_s(r)=\gf_s(r)$.
Finally, let
$$
e_s(r) := \hf_s(r)-\gf_s(r).
$$

We will show in \autoref{lm:esr} that $e_s(r)$ is very small and non-negative for near-optimal $r$,
so $\hf_s$ is a good approximation to $\gf_s$.

\begin{figure}
\centering
\begin{subfigure}{.5\textwidth}
  \centering
  \includegraphics[width=1\linewidth]{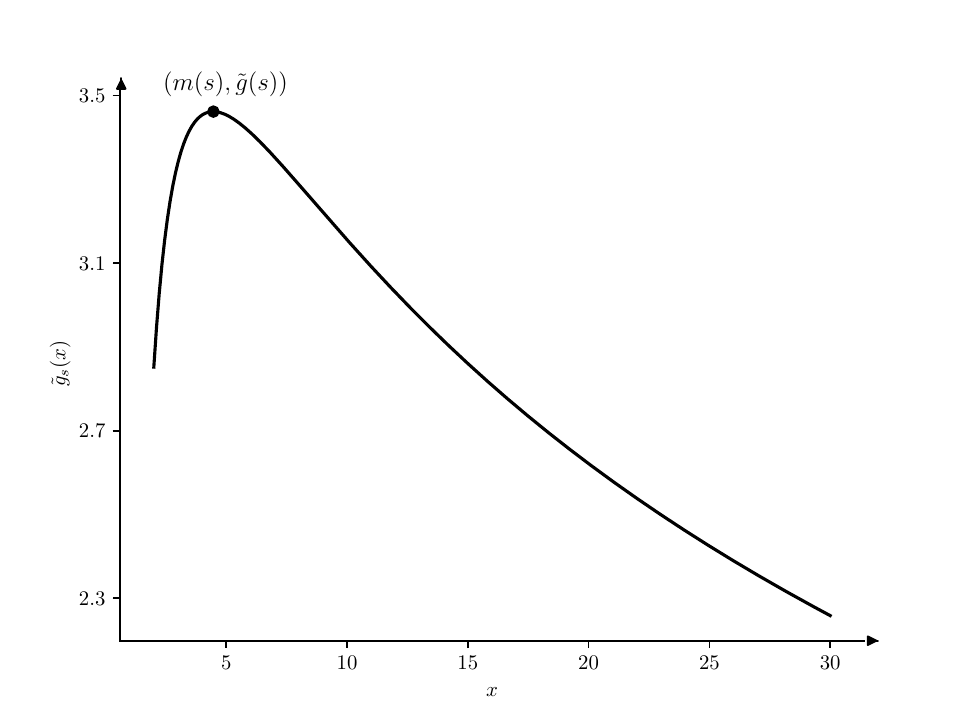}
  \caption{}
\end{subfigure}%
\begin{subfigure}{.5\textwidth}
  \centering
  \includegraphics[width=1\linewidth]{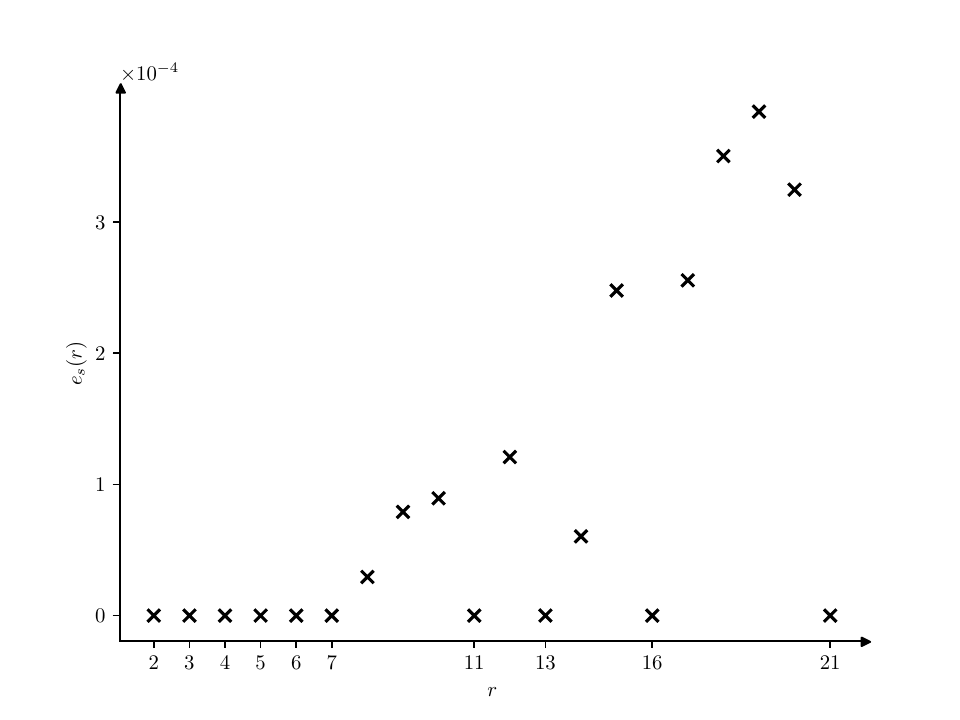}
  \caption{}
\end{subfigure}
\caption{Plot (A) shows $\hf_s(x)$ for $s=300$ with maximum attained at $m(s)=W(s/e)+1$. Plot (B) displays the difference $e_s(r)=\hf_s(r)-\gf_s(r)$ for $s=300$. Note that the difference is $0$ whenever $s$ is divisible by $r-1$.} 
\label{fig:gngtilde}
\end{figure}

Next we state two constructions of feasible triples with $r$ parts for $r \in 2\mb{N}$,
which we will show achieve a $q$-value of $\gf_s(r)$.
The colourings generated by \autoref{cons:general} are equivalent to our earlier, less formal, description in \autoref{cons:matchings}.

\begin{construction}\label{cons:general}
Let $r=2t$ be even and let $M_1,\ldots,M_T$ be a list of all the perfect matchings of $K_r$, where $T=r!/(t!2^t)$.
Let $\mc{A} = \{A_1,\ldots,A_T\}$ be a partition of $[s]$ (where some parts could be empty)
with the property that
$|b_e-b_{e'}| \leq 1$ for all $e,e' \in E(K_r)$,
where
$\phi_{\mc{A}}(e) := \bigcup_{\ell \in [T]: e \in M_\ell}A_{\ell}$
and
$b_e := |\phi_{\mc{A}}(e)|$.
Let $\bm{u} \in \DD^r$ be uniform.
This defines $(r,\phi_{\mc{A}},\bm{u})$.
\end{construction}
As we will show shortly, this implies every $e \in E(K_r)$ satisfies
\begin{equation}\label{eq:consmult}
|\phi_{\mc{A}}(e)| \in \left\lbrace \left\lfloor\frac{s}{r^*-1}\right\rfloor, 
\left\lceil\frac{s}{r^*-1}\right\rceil \right\rbrace.
\end{equation}

The second construction is a special case of the first and shows that the set of partitions $\mc A$ satisfying the property above is non-empty.

\begin{construction}\label{cons:dec}
Let $r=2t$ be even and let $E(K_r)=M_1 \cup \ldots \cup M_{r-1}$ be a decomposition of the $r$-clique into perfect matchings and write $\mc{M}:=\{M_1,\ldots,M_{r-1}\}$.
Let $[s] = A_1 \cup \ldots \cup A_{r-1}$ be an equipartition.\footnote{An \emph{equipartition} of a set has part sizes whose pairwise differences are at most 1.}
For each $ij \in \binom{[r]}{2}$, let $\phi_{\rm dec}(ij):=A_\ell$ where $ij \in M_\ell$.
Let $\bm{u} \in \DD^r$ be uniform.
This defines $(r,\phi_{\rm dec},\bm{u})$.
\end{construction}

In fact, these constructions are identical when $r\in\{2,4\}$, since all pairs of perfect matchings of $K_r$ are edge-disjoint. For all larger even $r$,~\autoref{cons:general} yields strictly more $\phi$.

We claim that the triples defined in both constructions are feasible. As the second is a special case of the first, it suffices to see why $(r,\phi_{\mc{A}},\bm{u})$ in \autoref{cons:general} is feasible.
For this, we just need to check that $\phi_{\mc{A}}$ is $K^{(2)}_3$-free.
It suffices to show that there are no distinct $h,i,j \in [r]$ and $c \in [s]$ with 
$c \in \phi_{\mc{A}}(hi) \cap \phi_{\mc{A}}(hj)$. Let $\ell \in [T]$ be such that $c \in A_\ell$.
Then $\phi_{\mc{A}}^{-1}(c) = M_\ell$, a matching, so we are done.

Let us calculate $q(\phi_{\mc{A}},\bm{u})$.
With $a$ and $z$ as defined earlier in this section, for each $i\in [r]$, since $|b_{ij}-b_{ij'}|\le 1$ for any $j,j'\in [r]\setminus\{i\}$ and  $\sum_{j\in[r]\setminus\{i\}} b_{ij}=s$, we have that $a$ of the sets $\{\phi_{\mc A}(ij)\}_{j\in[r]\setminus\{i\}}$ have size $z+1$, while the remaining $r-1-a$ have size $z$
(and, in particular,~\eqref{eq:consmult} holds). Thus
\begin{align*}
q(\phi_{\mc{A}},\bm{u}) &= 2\left(\frac{1}{r}\right)^2 \frac{r}{2}
\left((r-1-a)\log(z) + a\log(z+1)
\right)
= \gf_s(r).
\end{align*}
Thus also $q(\phi_{\rm dec},\bm{u})=\gf_s(r)$ for any $\phi_{\rm dec}$
as in \autoref{cons:dec}.
We have proved the following lemma.

\begin{lemma}\label{lm:glower}
For all integers $s \geq 2$, we have $Q(s) \geq \gf(s)$. \hfill\qed
\end{lemma}

We end this section by introducing some terminology for $s$-colour templates $\phi$ on $r$ parts that behave as in \autoref{cons:general}. We say that $\phi$ is \emph{uniform at $i$} for some $i\in [r]$ if the set $\{\phi(ij)\}_{j\in [r]\sm \{i\}}$ is an equipartition of $[s]$. We further say that $\phi$ is \emph{uniform} if it is uniform at $i$ for all $i\in [r]$.

\begin{rem}\label{rem:uniform}
    A quick check shows that $\phi \in \Phi_{P,s}(r)$ is uniform if and only if $r$ is even and $\phi=\phi_{\mc A}$ for some $\mc A$ as in \autoref{cons:general}. 
\end{rem}

\subsection{Main optimisation result and proof sketch}\label{sec:disketch}

The machinery we have developed for general colour patterns 
means that our main task here is to determine $\opt^*(s)$
for all integers $s \geq 2$.
We will show that, for all $s \geq 2$, every element of $\opt^*(s)$
is of the form
$(r,\phi_{\mc{A}},(\frac{1}{r},\ldots,\frac{1}{r}))$,
where $r \in \Rev(s)$ and $\phi_{\mc{A}}$ is as in \autoref{cons:general}.
This will show the following theorem, which we will then combine with~\autoref{th:exact1} to prove~\autoref{th:K32}.

\begin{theorem}\label{th:K23opt}
    Let $s \geq 2$ be an integer. Then $Q(s) = \gf(s)$,
    and $(r^*,\phi^*,\ba^*) \in \opt^*(s)$
    if and only if $r^* \in \Rev(s)$, $\ba^*=(\frac{1}{r^*},\ldots,\frac{1}{r^*})$ is uniform, and $\phi^*=\phi_{\mc{A}}$
    for some $\mc{A}$ as in \autoref{cons:general}.
    
\end{theorem}

\begin{rem}\label{rm:reallyproving}
    We are actually proving the following, 
    where instead of maximising $q$ over $\phi$ without dichromatic triangles, we maximise $q$ over $\phi$ in which each colour class is a matching.
    Together with \autoref{lm:match}(ii), this implies \autoref{th:K23opt}.

    Let $s \geq 2$ be an integer. Let $(r,\phi,\ba)$ be such that $r$ is an integer, $\phi:\binom{[r]}{2} \to 2^{[s]}$ and $\phi^{-1}(c)$ is a matching for all $c \in [s]$, and $\ba \in \DD^r$. Then $q(\phi,\ba) \leq \gf(s)$ with equality if and only if $r \in \Rev(s)$, $\ba$ is uniform and $\phi=\phi_{\mc{A}}$ for some $\mc{A}$ as in \autoref{cons:general}.
\end{rem}

We now sketch the proof of \autoref{th:K23opt}. Let $(r,\phi,\ba)$ be a basic optimal solution. We recall from~\autoref{lm:cont1} that
every part has optimal contribution, that is, $q_i(\phi,\ba)=\sum_{j \in [r]\sm\{i\}}\aA_j \log \phi_{ij} = Q(s)$. We will make use of this fact by analysing contributions in a number of ways.
The first way is the following averaging trick for the contributions: for an optimal solution we have
\begin{equation}\nonumber
q(\phi,\ba) = \frac{1}{r}\sum_{i \in [r]}q_i(\phi,\ba)= \sum_i\frac{1}{r}\sum_{j\neq i}\alpha_j\log\phi_{ij}
=\frac{1}{r}\sum_j\alpha_j\sum_{i\neq j}\log\phi_{ij} \leq \sum_j\alpha_j \gf_s(r)=\gf_s(r),\end{equation}
 where the last inequality uses a simple fact about maximising the sum of logarithms (\autoref{lm:maxoflog}).
This shows that solutions on an even number of parts have $q(\phi,\ba) \leq g(s)$ and some further analysis shows that this bound is only attained by the $(r,\phi,\ba)$ described in~\autoref{th:K23opt}.

It then only remains to rule out the case that $r$ is odd for which we needed to work much harder.
By~\autoref{lm:match}(ii), every $\phi^{-1}(c)$ is a maximal matching, and hence we might intuitively expect $r$ to be even so that the matchings can span the vertex set. When $1$ is a largest part, again averaging over contributions, we see that
$$
q(\phi,\ba)=\frac{1}{r}\sum_{i \in [r]}q_i(\phi,\ba) \leq \frac{2\alpha_1}{r}\sum_{ij \in \binom{[r]}{2}}\log \phi_{ij}.
$$
But if $r$ is odd, then every matching $\phi^{-1}(c)$ contains $\frac{r-1}{2}$ edges, and hence the sum of all $\phi_{ij}$ is at most $s\cdot\frac{r-1}{2}$
(if $r$ is even, this sum is $s\cdot\frac{r}{2}$). If then $\alpha_1$ is not significantly larger than $\frac{1}{r}$, this yields a contradiction to $(r,\phi,\ba)$ being optimal
(see~\autoref{lm:largepart} for details).

In \autoref{lm:contribution}, we maximise $q(\phi,\ba)$ under the condition that $1$ is a largest part of some fixed size $x$. A weight shifting process shows that the contribution $q_1(\phi,\ba)$ is maximised when we have as many parts of size $x$ as possible and possibly one smaller part. It turns out that when $x$ is significantly larger than $\frac{1}{r}$, $q_1(\phi,\ba)$ is not only smaller than $\gf_s(r)$, but actually smaller than either $\gf_s(r-1)$ or $\gf_s(r+1)$ (see \autoref{lm:fcomp}).  Thus the $r$-part solution is not optimal (and we do better by taking an even solution with one fewer part or one more part).

We give the full proof of \autoref{th:K23opt} in \autoref{sec:K23:proof}. Before we can do this, in the following subsection we will investigate which $r$ maximise $g_s(r)$, which will allow us to assume that $r\leq \log(s)+2$ later.

\subsection{Analytic estimates}\label{sec:analytic}

In this section, we prove some estimates concerning the functions $\gf_s$, $\hf_s$, $e_s$, $\gf$, $\Rmax$ and $\Rev$. 
These are mainly proved via calculus and hence the more standard proofs are deferred to the appendix.

\begin{lemma}\label{lm:esr}
For all integers $s,r$ satisfying $2 \leq r \leq s-1$, we have
$$
0 \leq e_s(r) \leq \frac{1}{4}\left\lfloor \frac{s}{r-1}\right\rfloor^{-2}.
$$
\end{lemma}

\begin{proof}
The fact that $e_s(r) \geq 0$ is a consequence of the concavity of $\log(x)$ and Jensen's inequality. Indeed, if we denote $s=(r-1)z+a$ for $0\le a\le r-2$, we have
\begin{align*}\gf_s(r)&=\frac{r-1}{r}\left(\frac{r-1-a}{r-1}\log (z)+\frac{a}{r-1}\log (z+1)\right)\\
&\le \frac{r-1}{r}\log\left(\frac{r-1-a}{r-1}z+\frac{a}{r-1}(z+1)\right)=\hf_s(r).\end{align*}
For the upper bound, it is also easy to see that $e_s(r)=0$ when $r-1$ divides $s$, while otherwise we have
\begin{align*}
&r \cdot e_s(r) = a\log\left(\frac{s/(r-1)}{z+1}\right)+(r-1-a)\log\left(\frac{s/(r-1)}{z}\right)\\
&=a\log\left(1+\frac{a-(r-1)}{(r-1)(z+1)}\right)+(r-1-a)\log\left(1+\frac{a}{(r-1)z}\right)
\\
&\leq 
\frac{a (r-1-a)}{r-1}\left(\frac{1}{z}-\frac{1}{z+1}\right),
\end{align*}
where for the inequality we used that $\log(1+x)\leq x$ for $x>-1$.
Since $a(r-1-a)\leq \frac{1}{4}(r-1)^2$, we get
$$
e_s(r) \leq \frac{1}{4} \left(\frac{1}{z}-\frac{1}{z+1}\right)=\frac{1}{4z(z+1)} \leq \frac{1}{4}\left\lfloor \frac{s}{r-1}\right\rfloor^{-2},
$$
as required.
\end{proof}

\begin{lemma}\label{lm:gsr}
For all integers 
$r,s$ with $2 \leq r < (s/2)^{1/4}$ we have
 $$\gf_{s+1}(r+1)-\gf_{s+1}(r) > \gf_s(r+1)-\gf_s(r).$$
\end{lemma}

\begin{proof}
\autoref{lm:esr} implies that
$$
    -e_{s+1}(r+1)+e_s(r+1)-e_s(r)+e_{s+1}(r) \geq
-e_{s+1}(r+1)-e_s(r) \geq - \frac{2(r-1)^2}{4(s-r+1)^2}\geq - \frac{r^2}{s^2},
$$
since $r\leq s/4$ which follows from $2 \leq r < (s/2)^{1/4}$. Using that $\log(1+x) \geq x/(1+x)$ for $x>0$, we have
$$
\hf_{s+1}(r+1)-\hf_s(r+1)+\hf_s(r)-\hf_{s+1}(r) = \frac{1}{r(r+1)}\log\left(1+\frac{1}{s}\right) \geq \frac{1}{r(r+1)(s+1)} \geq \frac{1}{2r^2s}.
$$
Combining the two previous bounds, we get
$$
\gf_{s+1}(r+1)-\gf_{s}(r+1)+\gf_s(r)-\gf_{s+1}(r) \geq \frac{1}{2r^2s} - \frac{ r^2}{s^2} > 0,
$$
since $2r^4 < s$,
completing the proof.
\end{proof}

The \emph{Lambert $W$-function} is the inverse of $f(x) = xe^x$.
That is, $y=xe^x$ for $y \geq 0$ if and only if $x=W(y)$.

By applying $\log$ to $y=W(y)e^{W(y)}$ and using standard bounds on $\log$ we have
\begin{equation}\label{eq:W}
\tfrac{1}{2}\log y <\log y-\log\log y < W(y) < \log y  \quad\text{for }y > e.
\end{equation}
For $s<e^2$, let $\rmax(s) = \rev(s) := 2$, and for $s\geq e^2$, let
$$
\rmax(s) :=
\max\{r \in \mb{N}: (r-1)e^r \leq s\}
\quad\text{and}\quad
\rev(s) := 
\max\{r \in 2\mb{N}: (r-1)e^r \leq s\}.
$$ 
(We already defined $\rev(s)$ in Definition~\ref{def:defs}.)

Then for $s \geq e^2$, we have $\rev(s) = 2\lfloor (W(s/e)+1)/2\rfloor$
and~(\ref{eq:W}) implies that $\rmax(s) = \lfloor W(s/e)+1\rfloor \leq \log(s)$. Thus, for all $s \geq 2$,
\begin{equation}\label{eq:rs}
 \rev(s) \leq \rmax(s) \leq \max\{2,\log(s)\}
\end{equation} and, for $s \geq 92$,
\begin{equation}\label{eq:rs_lower}
\rmax(s)\geq \log s -1-\log(\log(s)-1) \geq \log(s)/2.
\end{equation}
Moreover, we note that $W'(y)=\frac{W(y)}{y(1+W(y))}$ for all $y > 0$ and hence 
\begin{equation} \label{eq:W'}
    W'(y)=|W'(y)|\leq 1/y.
\end{equation}
The next lemma concerns the analytic properties of the function $\hf_s$.
Since its proof follows by elementary calculus, we defer it to the \hyperref[appendix]{Appendix}.

\begin{lemma}\label{lm:hanalytic}
The following holds for any integer $s \geq 2$.
\begin{enumerate}[label=\emph{(\roman*)}]
    \item $\hf_s$ has a unique maximum at $m(s)=W(s/e)+1$
    (so for $s \geq e^2$ we have $\rmax(s) = \lfloor m(s) \rfloor$).
    Moreover, $\hf_s$ is strictly increasing on $(1,m(s))$ and strictly decreasing on $(m(s),\infty)$, and
    $$
    \hf'_s(x) = \frac{\log\left(\frac{s}{x-1}\right)-x}{x^2} \quad\text{for }x >1,
    $$
    \item $\hf(s)=W(s/e)$,
    \item  $\hf_s(m(s)+a)-\hf_s(m(s)+b) \geq \frac{1}{16}(\log(s)+5/2)^{-2}$ whenever $s \geq e^2$; $m(s)+a,m(s)+b>0$; $ab \geq 0$; $|a|<2$ and $|a|+1/2 \leq |b|$, 
    \item $\hf_s(m(s))-\hf_s(m(s)+b) \leq 8|b|^2(\log(s)-4)^{-2}$ whenever $|b|\leq \min\{2,m(s)-2\}$ and $s\ge 55$.
\end{enumerate}    
\end{lemma}

We are now ready to give a precise description of $\Rmax(s)$ and $\Rev(s)$. While parts (i) and (ii) are used in later subsections to prove~\autoref{th:K32}, part (iii) 
is used in the proof of~\autoref{cor:inf} as well as providing additional structural information that tells us for fixed $r$ what the sets $S(r) := \{s:~ r \in \Rmax(s)\}$ and $S_2(r):=\{s:~ r \in \Rev(s)\}$ look like.

\begin{lemma}\label{lm:Rsprecise}
Let $s \geq 2$ be an integer.
Then
\begin{enumerate}[label=\emph{(\roman*)}]
\item $\Rmax(s) \subseteq \{\rmax(s),\rmax(s)+1\}$ and
$\Rev(s) \subseteq \{\rev(s),\rev(s)+2\}$.
In particular, $\max \Rmax(s), \max \Rev(s) \leq \log(s)+2$.
\item
$g(s) > \min\{g_s(p),g_s(q)\}$ for any two distinct odd integers $p,q$.
\item 
There are increasing sequences $(s_2,s_3,s_4,\ldots)$
and $(\tilde s_2,\tilde s_4,\tilde s_6,\ldots)$ with
$$
2=s_2 < e^2 < s_3 < 2e^3 <
\ldots < (r-2)e^{r-1} < s_r < (r-1)e^r < s_{r+1} < re^{r+1} < \ldots
$$
and $s_{r-1} \leq \tilde s_r \leq s_r$
such that the set $S(r):=\{s:~ r \in \Rmax(s)\}$ is either the interval $[s_r,s_{r+1}-1]$ or $[s_r,s_{r+1}]$ for all $r \in \mb{N}$;
and the set $S_2(r):=\{s:~ r \in \Rev(s)\}$ is either the interval $[\tilde s_r,\tilde s_{r+2}-1]$ or $[\tilde s_r,\tilde s_{r+2}]$ for all $r \in 2\mb{N}$.
Moreover, for all $r \in 2\mb{N}$ there exists $s \in \mb{N}$ with $\Rev(s)=\{r\}$.
\item The values in~\autoref{table:1} and~\autoref{table:2} hold.
\end{enumerate}
\end{lemma}

Before proving the lemma, we make some remarks on~(ii).
If $\gf_s(r)$ is maximised by an even $r$,
then part~(ii) follows from~(i) since then $\Rmax(s)$ contains an even $r$ and at most one of $p,q$.
However, if $\gf_s(r)$ is only maximised by an odd $r$, then we could have $\Rmax(s)=\{p\}$ and $\gf_s(r)<\gf_s(p)$ where $r \in \Rev(s)$,
and must show $\gf_s(r)>\gf_s(q)$ for every odd $q \neq p$.

\begin{proof}
We will prove these assertions analytically for large $s$ and computationally for small $s$.
First we prove parts~(i) and (ii).
Suppose first that $s < 196$. We compute all values $\gf_s(r)$ for $r < s$ which determines $\Rmax(s), \Rev(s)$ and verifies parts~(i) and~(ii).
A script for this calculation is in the ancillary file {\tt optr.py}.

Thus we may suppose that $s \geq 196$.
We claim that for distinct integers $r_a=m(s)+a, r_b=m(s)+b \geq 2$ with $ab\geq 0$, $|a|<2$ and $|a|<|b|$, we have $g_s(r_a)>g_s(r_b)$. To see this, note that \autoref{lm:hanalytic}(iii) gives $\hf_s(r_a)-\hf_s(r_b) \geq 1/(16(\log(s)+5/2)^2).$  Equation~\eqref{eq:rs} shows that we have $r_a\leq \log(s)+2$. Thus \autoref{lm:esr} gives that
\begin{align}\label{eq:gh}
e_s(r_a) \leq \frac{1}{4\lfloor s/(r_a-1)\rfloor^2} \leq \frac{(r_a-1)^2}{4(s-r_a+1)^2} \leq \frac{(\log(s)+1)^2}{s^2}.
\end{align}
Putting the above observations together, we get
\begin{align*}
    g_s(r_a)-g_s(r_b)&\geq \hf_s(r_a)-\hf_s(r_b) - e_s(r_a) \geq \frac{1}{16(\log(s)+5/2)^2}- \frac{(\log(s)+1)^2}{s^2}>0,
\end{align*}
where the last inequality is true for $s \geq 196$. This proves the claim.

Parts (i) and (ii) follow easily from the claim. Indeed, to prove that $\Rmax(s) \subseteq \{\rmax(s),\rmax(s)+1\}$, it suffices to show that
for all integers $r,r' \geq 2$ with $r<\rmax(s)<\rmax(s)+1<r'$, we have $\gf_s(\rmax(s)) > \gf_s(r)$ and $\gf_s(\rmax(s)+1)>\gf_s(r')$.
Given such $r$, let $r_a=\rmax(s)$, $r_b=r$. Since $\rmax(s)\leq m(s)$, $a$ and $b$ are both non-positive and the other conditions are also easy to verify.
Given such $r'$, the second inequality follows similarly since $\rmax(s)+1 \geq m(s)$.

For the statement regarding $\Rev(s)$, the claim implies that $ g_s(r(s)+2)>g_s(r')$ when $r'>r(s)+2$. 

For part (ii), 
suppose that $p,q$ are both odd integers and note that there is an even integer $r$ that satisfies $|r-m(s)|<2$ and lies between $m(s)$ and, without loss of generality, $p$. By our claim applied with $r=r_a$ and $p=r_b$ we have that $g(s) \geq g_s(r)>g_s(p)\ge \min\{g_s(p), g_s(q)\}$ as desired.

Next, we prove part (iii).
First we will consider $S^0(r) := S(r) \cap [2,90000]$
and $S^0_2(r) := S_2(r) \cap [2,90000]$.
Suppose that $s<90000$.
By part~(i) we can find $\Rmax(s)$ and $\Rev(s)$ by computing $\rmax(s)$ and $r_2(s)$ as approximations of $W(s/e)$ and comparing the two possible values of $\gf_s$. Then we can manually check that 
$S^0(2)=[2,16],\ldots,S^0(8)=[11444,36023],
S^0(9)=[36024,90000]$
and 
$S^0_2(2)=[2,27],\ldots,S^0_2(8)=[5857,59470],S^0_2(10)=[59471,90000]$
are as shown in~\autoref{table:1} and~\autoref{table:2}. The two tables contain the optimal $r$ for $s\le \maxs$ (for $\Rev(s)$) and $s\le 90000$ (for $R(s)$) for illustrative purposes.
A script for this calculation is provided in the ancillary file {\tt optr.py}.

Suppose now that $s \geq 90000$.
We first show that
$\min\{\Rmax(s+1)\} \geq \max\{\Rmax(s)\}$ and $\min\{\Rev(s+1)\} \geq \max\{\Rev(s)\}$. Let $r^*=\max\{\Rmax(s)\}$ and $r_2^*=\max\{\Rev(s)\}$. 
By~\eqref{eq:rs} and part (i) of this lemma, we have $r^*,r_2^*\leq r(s)+2\le \log(s)+2<(s/2)^{1/4}+1$, so \autoref{lm:gsr} implies that for $1 \leq p< \max\{r^*,r_2^*\}$ we have $\gf_{s+1}(p+1)-\gf_{s+1}(p)>\gf_{s}(p+1)-\gf_s(p)$.
For any even $r$ with $r<r_2^*$, summing these equations over $r\le p \leq r_2^*-1$ yields
$$
\gf_{s+1}(r_2^*)-\gf_{s+1}(r)>\gf_{s}(r_2^*)-\gf_s(r) \geq 0.
$$
That is, for all even $r<r_2^*(s)$ we have $r \notin \Rev(s+1)$. 
So $\min\{\Rev(s+1)\} \geq \max\{\Rev(s)\}$. Summing the equations over $r\le p \le r^*$, it can be proved analogously that $\min\{\Rmax(s+1)\} \geq \max\{\Rmax(s)\}$.

Combined with the $S^0(r)$ and $S^0_2(r)$ obtained above, this yields the claimed interval structure of $S(r)$ and $S_2(r)$ respectively and, defining $s_r$ and $\tilde s_r$ as the minima of $S(r)$ and $S_2(r)$ respectively, it is easy to see that $s_{r-1} \leq \tilde s_r \leq s_r$. It now remains to show that for each integer $r \geq 2$, we have $s_r<s^*<s_{r+1}$ where $s^*\coloneqq (r-1)e^r \geq 90000$ and for all $r \in \mb{N}$, there is $s \in \mb{N}$ with $\Rmax(s) = \{r\}$ 
(we can see that this holds for $r \in [2,8]$, for which $s^*<90000$). For this, it suffices to prove that $R(\lfloor s^*\rfloor) = \{r\}$.
By definition, $r=W(s^*/e)+1$. Let $s$ satisfy $|s-s^*| \leq 1$.
Due to~\eqref{eq:W'} we have $|W'(x)|\leq 1/x$, so the mean value theorem gives $|m(s)-r|=|W(s/e)-W(s^*/e)|<1/(s^*-1)$. Then, by~\autoref{lm:hanalytic}(iv) applied with $b=r-m(s)$,
$$
\hf_s(m(s))-\hf_s(r) \leq \frac{8}{(\log(s) -4)^2(s^*-1)^2} \leq \frac{1}{s}.
$$
On the other hand, for any integer $r' \neq r$ we have $$|r'-m(s)|\geq |r'-r|- |r-m(s)| \geq 1-1/(s^*-1) \geq 1/2$$ and thus, by~\autoref{lm:hanalytic}(iii), $\hf_s(m(s))-\hf_s(r')\geq \frac{1}{16}(\log(s)+5/2)^{-2}$. 
Using~\eqref{eq:gh}  with $r$ in place of $r_a$ (which holds since $|m(s)-r|<2$) and putting everything together, we get
\begin{align*}
g_s(r') &\leq \hf_s(r') \leq  \hf_s(m(s)) - \frac{1}{16(\log(s)+5/2)^2} \leq \hf_s(r)+ \frac{1}{s}- \frac{1}{16(\log(s) +5/2)^2} \\
&\leq g_s(r) + \frac{(\log(s) +1)^2}{s^2} + \frac{1}{s^*-1}- \frac{1}{16(\log(s) +5/2)^2} < g_s(r),
\end{align*} as desired.
(This argument also shows that $R(\lfloor s^*\rfloor+1) = \{r\}$.)

Finally, for~(iv), we compute $\gf_s(r)$ for $r=\rmax(s),\rmax(s)+1,\rev(s),\rev(s)+2$ which, by~(i), determines $\Rmax(s)$ and $\Rev(s)$. See the ancillary file {\tt optr.py}.
\end{proof}

The final result of this section provides asymptotics for $\gf(s)$. Its proof is similar to the previous one and is deferred to the \hyperref[appendix]{Appendix}.

\begin{lemma}\label{lm:gapprox}
For all integers $s \geq 200$ we have
$$
0\le W(s/e)-\gf(s)\le \frac{600}{(\log(s))^2}
,\quad\text{and}\quad
\frac{e^{\gf(s)}}{(s/e)/W(s/e)} \to 1 \text{ as }s \to \infty.
$$
\end{lemma}

\subsection{Proof of \autoref{th:K23opt}}\label{sec:K23:proof}

In this section, we prove \autoref{th:K23opt}. We first need a simple lemma about maximising sums of logarithms, whose short proof is included in the \hyperref[appendix]{Appendix}.

\begin{lemma}\label{lm:maxoflog}
Suppose that $a_1,\ldots,a_n$ are positive with $a_1+\ldots+a_n=a$.
Then, for any $x>0$ and positive $x_1,\ldots,x_n$ 
with $x_1+\ldots+x_n \leq x$, we have
$$
\sum_{i=1}^n a_i \log(x_i) \leq 
\sum_{i=1}^n a_i\log\left(\frac{xa_i}{a}\right).
$$
    Moreover, if $a_1=a_2=\ldots=a_n$, and the $x_i$ are constrained to be integers, then the maximum is attained whenever $|x_i-x_j| \in \{0,1\}$ for all $i,j\in [n]$.
\end{lemma}

The following lemma bounds $Q(s)$ by $\gf_s(r)$. This essentially proves~\autoref{th:K23opt} when all optimal solutions have an even number of parts (which we will eventually show is always the case), since for even $r$ we have $\gf_s(r) \leq \gf(s).$
\begin{lemma} \label{lm:g_bd}
    Let $s\ge 2$. For any $(r,\phi,\ba) \in \opt^*(s)$ we have
    \begin{align}\label{eq:g_bd}
        q(\phi,\ba) \leq \gf_s(r).
    \end{align} Moreover,
    \begin{enumerate}[label=\emph{(\roman*)}]
        \item when $r$ is odd, the inequality is strict;
        \item when $r\in \Rev(s)$  we have equality
        if and only if $\ba$ and $\phi$ are uniform.
    \end{enumerate}
\end{lemma}
\begin{proof}
Suppose that $(r,\phi,\ba) \in \opt^*(s)$.
Then by~\autoref{lm:cont1}, for all $i\in[r]$ we have that $q_i(\phi,\ba)=q(\phi,\ba)=Q(s)$.
By~\autoref{lm:match}(i), each $\phi^{-1}(c)$ is a matching and $\sum_{j\neq i}\phi_{ij}\le s$. By \autoref{lm:maxoflog} applied with $x_j=\phi_{ij}$, $a_j=1$ and $x=s$ we have that $\sum_{j\neq i} \log \phi_{ij}\le r\gf_s(r)$ with equality if and only if $\phi$ is uniform at $i$ and $i$ sees every colour. So, averaging over contributions, we have
\begin{equation}\label{eq:averaging_trick}q(\phi,\ba) = \frac{1}{r}\sum_{i \in [r]}q_i(\phi,\ba)= \sum_i\frac{1}{r}\sum_{j\neq i}\alpha_j\log\phi_{ij}
=\frac{1}{r}\sum_j\alpha_j\sum_{i\neq j}\log\phi_{ij} \leq \sum_j\alpha_j \gf_s(r)=\gf_s(r),\end{equation}
proving the first part of the lemma. For part (ii), note that if $r$ is odd, each $\phi^{-1}(c)$ is a matching missing at least one vertex, so there is some $i \in [r]$ with $\sum_{j \neq i}\phi_{ij} < s$. The average contribution is therefore strictly less than $\gf_s(r)$. On the other hand, if $r$ is even and $\ba$ and $\phi$ are uniform, we have $q(\phi,\ba)=\gf_s(r)$, so it remains to prove the `only if' direction of part (ii) for $r \in \Rev(s)$.

Suppose that we have equality throughout \eqref{eq:averaging_trick} above. Then $\phi$ must be uniform by our application of \autoref{lm:maxoflog}. Moreover, if $r=2$, we have $\phi_{12}=[s]$ and optimising over $\ba$ yields that $\ba$ must be uniform. So from now on, we may assume that $r\geq 4$ and that $\phi$ is uniform.

Suppose, without loss of generality, that $\alpha_1=\max_i \alpha_i$ and let $a$ be such that
$0 \leq a < r-1$ with $a \equiv s \bmod r-1$.
The fact that $\phi$ is uniform at $1$
means that $\phi_{1i}=\lfloor s/(r-1)\rfloor =: z$ for $r-1-a$ values of $i$ and $\phi_{1i} = z+1$ for the remaining $a$ values of $i$.
Thus, without loss of generality,
\begin{align*}
    q(\phi,\ba)&=q_1(\phi,\ba)
    = (\alpha_2+\ldots+\alpha_{a+1})\log(z+1) + (\alpha_{a+2}+\ldots+\alpha_r)\log(z)\\
    &=(\alpha_2+\ldots+\alpha_{a+1})(\log(z+1)-\log(z))+(1-\alpha_1)\log(z)\\
    &\leq a \alpha_1 (\log(z+1)-\log(z)) +(1-\alpha_1) \log(z) =:p(\alpha_1).
\end{align*}
Note that $p(1/r)=\gf_s(r)$. We now show that $p$ is strictly monotone decreasing. Using $a \leq r-1$ and $z \geq (s-r+1)/(r-1)$, we obtain
\begin{align*}
p'(\alpha_1)&= a \log\left(1+\frac{1}{z}\right)-\log(z) 
\leq a \log \left(1+\frac{r-1}{s-r+1}\right) - \log\left(\frac{s-r+1}{r-1}\right)\\
&\leq \frac{(r-1)^2}{s-r+1}-\log\left(\frac{s-r+1}{r-1}\right).
\end{align*}
Since $r\geq 4$ we have $s\geq 27$ (recall~\autoref{lm:Rsprecise}(ii) and refer to~\autoref{table:1}) and, using~\autoref{lm:Rsprecise}(i), this yields $r-1\leq \log(s) +1\leq \sqrt{s}<s/2$. We can thus bound $$p'(\alpha_1)\leq \frac{s}{s-r+1}-\log\left(\frac{s-r+1}{\sqrt{s}}\right) \leq 1+\frac{2}{\sqrt{s}}-\log(\sqrt{s}-1)<0$$ for $s \geq 27$,
so we have shown that $p$ is indeed strictly monotone decreasing. 

To finish the proof, note that since $p$ is strictly monotone decreasing and $\alpha_1=\max_i \alpha_i\ge 1/r$, we have $q(\phi,\ba)\le p(\alpha_1)\le p(1/r)= \gf_s(r)$. We have equality throughout only if $\alpha_1=1/r$ which implies $\alpha_i=1/r$ for all $i$, as required.
\end{proof}

The following lemmas help us to show that optimal solutions cannot have an odd number of parts. 
The first lemma shows that, if there were a basic optimal solution on an odd number of parts,
the parts cannot be close to uniform, since the largest part must be significantly larger than $\frac{1}{r}$.

\begin{lemma}\label{lm:largepart}
Let $s \geq 2$ be an integer.
Let $(r,\phi,\ba) \in \opt^*(s)$ and suppose that $r$ is odd and $\alpha_1=\max_i\alpha_i$.
Then
$$
\alpha_1 \geq \frac{r}{r^2-1}-\frac{e_s(r+1)}{(r-1)\log (s/r)}.
$$
\end{lemma}

\begin{proof}
    Since $(r,\phi,\ba)$ is optimal with no zero parts, we have that $Q(s) = q_i(\phi,\ba)$ for all $i \in [r]$ by \autoref{lm:cont1}, so, averaging over contributions, we have
    $$
    Q(s) = \frac{1}{r}\sum_{i \in [r]}q_i(\phi,\ba) 
    = \frac{1}{r}\sum_{i \in [r]}\sum_{j \neq i}\alpha_j\log\phi_{ij}
    \leq \frac{2\alpha_1}{r}\sum_{ij \in \binom{[r]}{2}}\log \phi_{ij}.
    $$
    Now, $\sum_{ij \in \binom{[r]}{2}}\phi_{ij} = \sum_{c \in [s]}|\phi^{-1}(c)| = s\cdot \frac{r-1}{2}$ as by \autoref{lm:match}(ii) each colour graph $\phi^{-1}(c)$ is a maximal matching.
    Thus, by concavity of the $\log$ function,
    $$
    Q(s) \leq \frac{2\alpha_1}{r} \cdot \binom{r}{2}\log\left(\frac{\sum_{ij}\phi_{ij}}{\binom{r}{2}}\right) = \alpha_1(r-1)\log\left(\frac{s}{r}\right).
    $$
    Since $r$ is odd, and by \autoref{lm:glower}, we have that $Q(s) \geq \gf(s) \geq \gf_s(r+1) = \hf_s(r+1)-e_s(r+1) = \frac{r}{r+1}\log\left(\frac{s}{r}\right) - e_s(r+1)$
    and after rearrangement we obtain the required inequality.
\end{proof}

The next main lemma bounds from above the contribution of a largest part in any feasible solution in terms of the function $\ff_s$ as defined below.
We recall that, by \autoref{lm:cont1}, in a basic optimal solution, every part has contribution equal to $Q(s)$.

For integers $s, r\ge 2$ and $\frac{1}{r} \leq x < \frac{1}{r-1}$, let
$$
\ff_{s,r}(x) = \max_{0 \leq t \leq s}\ff_{s,r,t}(x)
\quad\text{where}\quad
\ff_{s,r,t}(x) := (r-1)x \gf_{s-t}(r-1)+(1-(r-1)x)\log t
$$
and
$$
\frelax_{s,r}(x) := (r-2)x\log\left(\frac{xs}{1-x}\right) + (1-(r-1)x)\log\left(\frac{(1-(r-1)x)s}{1-x} \right).
$$
Let also $\frelax_s$ be defined by setting $\frelax_s(x) := \frelax_{s,r}(x)$ for $\frac{1}{r} \leq x < \frac{1}{r-1}$ and similarly define $\ff_s$.
See~\autoref{fig:ftilde} for a plot of $\frelax_s$.

The function $\ff_{s,r,t}(x)$ is the contribution of a part of size $x$, say part 1, in a solution with $r-1$ parts of size $x$ and one part, say part $r$, of size at most $x$ and where $|\phi(1r)|=t$ while the other multiplicities $|\phi(1j)|$ are as equal as possible. The function $\ff_{s,r}(x)$ maximises this contribution in a solution of this form. 
The function $\frelax_{s,r}(x)$ measures the same contribution, but, like $\hf_s(r)$, disregards the fact that the multiplicities $|\phi(1j)|$ have to be integral, which simplifies the optimisation of $t$ (the optimal $t$ is $\frac{(1-(r-1)x)s}{1-x}$).
We start by giving some properties of $\ff_s$ and $\frelax_s$ that help us compare them with $\gf_s$. Ultimately, these comparisons will come together with the subsequent \autoref{lm:contribution} to show that the objective function of a basic optimal solution on an odd number of parts can be bounded in terms of $\gf(s)$.

\begin{lemma}\label{lm:fprops}
Let $s \geq 2$ be an integer.
    \begin{enumerate}[label=\emph{(\roman*)}]
        \item For all integers $r \geq 2$, on the domain $[\frac{1}{r},\frac{1}{r-1}]$, $\frelax_s$ and $\ff_s$ are convex.
        \item $\ff_{s-\ell}(x) < \ff_s(x) \leq \frelax_s(x)$ for all $x \in (0,1)$ and integers $0<\ell \leq s-2$.
        \item $\ff_s(x) \leq \max\{\gf_s(r),\gf_s(r-1)\}$ for all $x \in (0,1)$, where $r$ is the integer with $\frac{1}{r} \leq x < \frac{1}{r-1}$.
    \end{enumerate}
\end{lemma}

\begin{proof} 
For~(i),
we have $\frelax_{s,r}''(x)=\frac{r-2}{(x-1)x((r-1)x-1)}$ which is positive for $0<x<1/(r-1)$, so $\frelax_{s,r}$ is convex.
Observe that $\ff_{s,r,t}(x)$ is a linear function of $x$, so $\ff_{s,r}$ is a maximum of convex functions and hence is convex.

For the first inequality in~(ii), it suffices to show that $\gf_{s_1}(r-1) < \gf_{s_2}(r-1)$ for all $s_1<s_2$ and $r \geq 2$.
Indeed, if so, we have $\ff_{s-\ell,r,t}(x) < \ff_{s,r,t}(x)$ for all $r$, all $0\le t\le s-\ell$ and all $\frac{1}{r} \leq x < \frac{1}{r-1}$, and hence $\ff_{s-\ell}(x) < \ff_s(x)$ for all $x \in (0,1)$.
Now, $\gf_{s_1}(r-1)$ is the average of a multiset of $r-2$ numbers $\{a_1,\ldots,a_{r-2}\}$ while $\gf_{s_2}(r-1)$ is the average of a multiset of $r-2$ numbers $\{b_1,\ldots,b_{r-2}\}$ where $a_i\leq b_i$ for all $i \in [r]$ and at least one of these is strict. This follows since each multiset contains at most two different values and the sum $s_1$ of the former is strictly less than $s_2$, the sum of the latter.
Thus $\gf_{s_1}(r-1) < \gf_{s_2}(r-1)$.

The second part of~(ii) follows from the concavity of $\log$.
Indeed, let $\bm{\rho}=(\rho_1,\ldots,\rho_{r-1}) \in \mb{R}_{\geq 0}^{r-1}$ with $\rho_1+\ldots+\rho_{r-1} \leq s$,
and let
$$
q_{s,r,x}(\bm{\rho}) := x(\log\rho_1+\ldots+\log\rho_{r-2})+(1-(r-1)x)\log\rho_{r-1}.
$$
Then, subject to every $\rho_i \in \mb{N}$ and $\rho_{r-1}=t \leq s$, $q_{s,r,x}(\bm{\rho})$
is maximised by taking $\rho_1,\ldots,\rho_{r-2}$ as equal as possible, by the second part of~\autoref{lm:maxoflog}.
This maximum is precisely $\ff_{s,r,t}(x)$.
Thus the maximum of $q_{s,r,x}(\bm{\rho})$ subject to all $\rho_i \in \mb{N}$ is $\ff_{s,r}(x)$.
By the first part of~\autoref{lm:maxoflog},
the maximum of $q_{s,r,x}(\bm{\rho})$ with no additional constraints is $\frelax_{s,r}(x)$.
Thus $\ff_s \leq \frelax_s$.

Finally, to prove~(iii),
since $\ff_{s,r}$ is convex by part~(i),
any maxima lie at the extreme points of its domain, so $\ff_{s,r}(x) \leq \max\{\ff_{s,r}(\frac{1}{r}),\ff_{s,r}(\frac{1}{r-1})\}$,
and $\ff_{s,r}(\frac{1}{r})=\gf_s(r)$ and $\ff_{s,r}(\frac{1}{r-1})=\gf_s(r-1)$.
\end{proof}

\begin{lemma}\label{lm:contribution}
Let $s \geq 2$ be an integer and let $(r,\phi,\ba) \in \feas_2(s)$ with $\alpha_1=\max_i \alpha_i$. 
Then
$q_1(\phi,\ba) \leq \ff_{s}(\alpha_1)$
with equality only if $\frac{1}{r} \leq \alpha_1 < \frac{1}{r-1}$.
\end{lemma}
\begin{proof}
Let $(r,\phi,\ba) \in \feas_2(s)$ with $x=\alpha_1=\max_i \alpha_i$ be given. Without loss of generality we may further assume that $\phi_{12}\geq \ldots \geq \phi_{1r}$. Let $r'\ge 2$ be the unique integer with  $\frac{1}{r'} \leq x < \frac{1}{r'-1}$, so $r' \geq r$, and let $y:=1-(r'-1)x$. Setting $$
\bb:=(\underbrace{x,\ldots,x}_{r'-1},y,0,\ldots,0) \in \DD^r,
$$
we get $q_1(\phi,\ba)\leq q_1(\phi,\bb)$.
To see this, note that due to the assumed monotonicity of the $\phi_{1i}$, moving weight from $\alpha_j$ to $\alpha_i$ for $2 \leq i <j\leq r$ does not decrease $q_1(\phi, \ba)$. Since all parts have size at most $\alpha_1=x$ at the start, we can move weight in this way until we reach $\bb$.

\autoref{lm:match}(i) implies that
$\sum_{i \in [2,r]}\phi_{1i} \leq s$.
Fixing $\bb$, $t:= \phi_{1r'}$ and $\ell:= \sum_{i \in [r'+1,r]}\phi_{1i}$, \autoref{lm:maxoflog} implies that $\sum_{i \in [2,r'-1]}\bB_i\log(\phi_{1i})$ is maximised when the $\phi_{12},\ldots,\phi_{1,r'-1}$ are as equal as possible, with sum $s-t-\ell$.
This yields that
\begin{align*}
q_1(\phi, \bb) & 
= y\log t + x\sum_{i \in [2,r'-1]}\log\left(\phi_{1i}\right)
\le y\log t + (r'-1)x\gf_{s-t-\ell}(r'-1)\\
&= \ff_{s-\ell,r',t}(x) \leq \ff_{s-\ell,r'}(x) \leq  \ff_{s,r'}(x) = \ff_s(x).
\end{align*}
where we used \autoref{lm:fprops}(ii) for the last inequality. 

When $\alpha_1 \geq 1/(r-1)$, we have $r'<r$, so $\ell \geq \phi_{1r} \geq 2$ (and, e.g.~$s-\ell \geq \phi_{12} \geq 2$), so \autoref{lm:fprops}(ii) gives the strict inequality $\ff_{s-\ell,r'}(x) < \ff_{s,r'}(x)$ in the above calculation, as required.
\end{proof}

\begin{figure}
\centering
\includegraphics{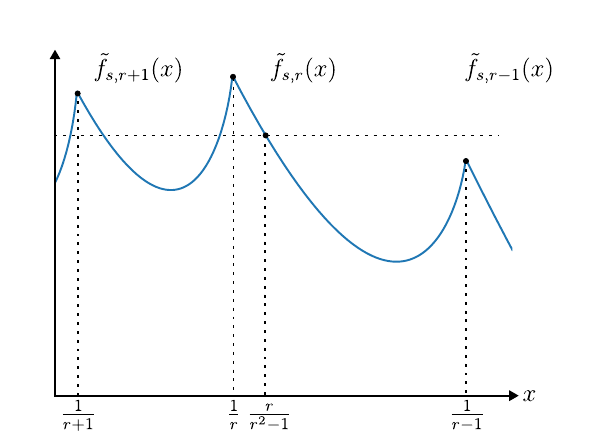}
\caption{Plot of $\frelax_{s}(x)$ illustrating the hypothesis of~\autoref{lm:fcomp} for values $s=800$ and $r=5$ (noting that $\frelax_{s}(x)$ and $f_{s}(x)$ are very close to each other).}
\label{fig:ftilde}
\end{figure}

The final main lemma in this section shows, roughly speaking, in a feasible solution on $r$ parts, that if there is a part at least almost as large as $\frac{r}{r^2-1}$, this solution cannot be optimal since it is beaten by one on $r-1$ or $r+1$ parts. As an illustration, for a fixed set of parameters, we refer the reader to~\autoref{fig:ftilde}. Combined with \autoref{lm:largepart} which says that there is always such an inflated part in an optimal solution on $r$ parts when $r$ is odd, this will imply that there cannot be an optimal solution with an odd number of parts.

\begin{lemma}\label{lm:fcomp}
Let $s \geq 2$ be an integer.
Then,
for all $r \in \Rmax(s)$ with $r \geq 3$, we have
$$
\ff_{s,r}(x) < \max\{\gf_s(r-1),\gf_s(r+1)\}
\quad\text{for all }\quad
\frac{r}{r^2-1}-\frac{e_s(r+1)}{(r-1)\log (s/r)} \leq x < \frac{1}{r-1}.
$$
\end{lemma}

\begin{proof}
    See \hyperref[appendix]{Appendix}. 
\end{proof}

We can now put everything together to prove \autoref{th:K23opt} on solutions to the optimisation problem for dichromatic triangles.

\begin{proof}[Proof of \autoref{th:K23opt}]
    Let $s \geq 2$ be an integer and let $(r^*,\phi^*,\ba^*) \in \opt^*(s)$. Now,
    \autoref{lm:glower} and~\autoref{lm:g_bd} imply that
    \begin{equation}\label{eq:base}\gf(s) \stackrel{\ref{lm:glower}}{\leq} Q(s) = q(\phi^*,\ba^*) \stackrel{\ref{lm:g_bd}}{\leq} \gf_s(r^*).
    \end{equation}
    
    \underline{If $r^*$ is even}, we have
    $\gf_s(r^*) \leq \gf(s)$,
    so we must have equality throughout. Having $\gf_s(r^*)=\gf(s)$ implies by definition that $r^*\in \Rev(s)$
    and~\autoref{lm:g_bd}(ii) implies that $\phi^*$ and $\ba^*$ are uniform. Recall that by \autoref{rem:uniform}, $\phi^*$ being uniform implies that $\phi^*=\phi_{\mc A}$ for some $\mc A$ as in \autoref{cons:general}.
    This completes the proof of the even case.

    Now suppose that \underline{$r^*$ is odd}. We claim that this never happens, so we will obtain a contradiction.
    From~\autoref{lm:g_bd}(i) we have the strict inequality
    \begin{equation}\label{eq:oddr*}
    q(\phi^*,\ba^*)< \gf_s(r^*).
    \end{equation}
    Suppose, without loss of generality, that $\alpha_1^*\ge \alpha_i^*$ for all $2\le i\le r^*$, and
    let $r'$ be the unique integer with $\frac{1}{r'} \leq \aA_1^* < \frac{1}{r'-1}$.
    Since $\alpha_1^*\ge \frac1{r^*}$, we have $r' \leq r^*$.
    By Lemmas \ref{lm:cont1}, \ref{lm:contribution} and~\ref{lm:fprops}(iii),
    we have
    \begin{equation}\label{eq:largest_part}
        g(s)\le Q(s) \stackrel{\ref{lm:cont1}}{=} q_1(\phi^*,\ba^*) \stackrel{\ref{lm:contribution}}{\leq} \ff_s(\aA^*_1)  \stackrel{\ref{lm:fprops}}{\leq} \max\{\gf_s(r'),\gf_s(r'-1)\},
    \end{equation}
    where, by definition, $\ff_s(\aA^*_1) = \ff_{s,r'}(\aA^*_1)$.

    We now claim that $r'$ must be equal to $r^*$. Indeed, suppose that $r'<r^*$, and let $r_{\text{even}},r_{\text{odd}}$ be the even and odd integer in $\{r',r'-1\}$, respectively. In this case $\{r_{\text{odd}},r^*\}$ are two distinct odd integers
    so by \autoref{lm:Rsprecise}(iii) we have $\min\{\gf_s(r^*),\gf_s(r_{\text{odd}})\}<\gf(s)$. From \eqref{eq:base} and \eqref{eq:oddr*} we have that $\gf(s)<\gf_s(r^*)$ so we must have $\gf_s(r_{\text{odd}})<\gf(s)$. Thus the right-hand side of \eqref{eq:largest_part} is at most $\gf(s)$ and again we have equality throughout.
    \autoref{lm:contribution} implies that $\aA^*_1 < 1/(r^*-1)$, which contradicts $r'<r^*$. Thus from now on we may assume that $r'=r^*$.
    
    We now proceed to compare $\gf_s(r')$ and $\gf_s(r'-1)$ which we now know are equal to $\gf_s(r^*)$ and $\gf_s(r^*-1)$, respectively. 
    From \eqref{eq:base} and \eqref{eq:oddr*} it follows that $\gf_s(r^*)>g(s)$ while $\gf_s(r^*-1)\le \gf(s)$ since $r^*-1$ is even, so $\gf_s(r^*-1) < \gf_s(r^*)$ and the maximum in \eqref{eq:largest_part} is strictly attained by $r^*$. Note that $\gf_s(r^*)>g(s)$ implies that $\Rmax(s)$ can consist solely of odd numbers.
    
    Further to the above, we must also have $r^* \in \Rmax(s)$, for otherwise there is some (odd) $r \in \Rmax(s)$ distinct from $r^*$ with both $\gf_s(r^*)$ and $\gf_s(r)$ exceeding $\gf(s)$, contradicting~\autoref{lm:Rsprecise}(iii).
    
    Finally, by~\autoref{lm:largepart} 
    we have
    $$
    \aA_1^* \geq \frac{r^*}{(r^*)^2-1}-\frac{e_s(r^*+1)}{(r^*-1)\log(s/r^*)}. 
    $$
    Then~\autoref{lm:fcomp} implies that 
    $\ff_{s,r^*}(\aA_1^*) < \max\{\gf_s(r^*-1),\gf_s(r^*+1)\} \leq \gf(s)$, where the last inequality follows from the definition of $\gf(s)$.
    Together with~\eqref{eq:largest_part}, this contradicts our initial assumption that $r^*$ is odd and finishes the proof.
\end{proof}

\subsection{Proof of~\autoref{th:K32}}\label{sec:K23:mainproof}

Now that we have determined the set of basic optimal solutions for the dichromatic triangle problem (\autoref{th:K23opt}),
we need to show that it has various properties
in order to apply our general exact result, \autoref{th:exact1}.
Most of these are easy to see; that the extension property holds is slightly more involved.

\begin{lemma}\label{lm:diprop}
    Let $s \geq 2$ be an integer. The dichromatic triangle family $\x = (K^{(2)}_3,s)$
    \begin{enumerate}[label=\emph{(\roman*)}]
        \item is bounded, 
        \item is hermetic (and hence stable inside),
        \item has the extension property,
        \item has the property that for every basic optimal solution $(r^*,\phi^*,\ba^*)$, there is some $z \geq r^*$ such that each multiplicity $\phi^*_{ij}$ equals $z$ or $z+1$. 
    \end{enumerate}
\end{lemma}

\begin{proof}
\autoref{th:K23opt} implies that for any $(r^*,\phi^*,\bm{u}_{r^*}) \in \opt^*(\x)$, we have that $r^*$ is even, $\bm{u}_{r^*}$ is uniformly $1/r^*$, all colour classes of $\phi^*$ are perfect matchings, 
the values $\phi^*_{ij}$ differ by at most one, 
and $r^* \in \Rev(s)$, which implies $r^* \leq \log(s) +2$ due to \autoref{lm:Rsprecise}(i). 
We calculated $\Rev(s)$ for small values of $s$ in~\autoref{table:1} (we showed these values hold in \autoref{lm:Rsprecise}(iv);
in particular, $r^*=2$ for $2 \leq s \leq 26$ and $r^* \leq 4$ for $s \leq 496$.

Part~(i) was proved in \autoref{lm:match}(iii).

        Next we prove~(ii). Let $(r^*,\phi^*,\bm{u}_{r^*}) \in \opt^*(\x)$ and suppose that $\phi\in \Phi_{\x,1}(r^*+1)$ extends $\phi^*$, i.e., $\phi|_{\binom{[r^*]}{2}}=\phi^*$.
        Let $c\in \phi(1,r^*+1)$. Then there is some $i\in [r^*]$ such that $c \in \phi(1i)$ since the edges of colour $c$ form a perfect matching in $\phi^*$. To avoid $\mc X$ we must have $\phi(i,r^*+1)=\{c\}$. However, since $\phi^*_{1i}\ge2$, there is some $c'\in \phi_{1i}$, giving a dichromatic triangle on $1,i, r^*+1$.
        Thus $\x$ is hermetic.

For~(iii), let $(r^*,\phi^*,\bm{u}_{r^*}) \in \opt^*(\x)$ and $\phi \in \Phi_\x(r^*+1)$ such that $\phi|_{\binom{[r^*]}{2}} = \phi^*$ and
     \begin{equation}\label{eq:ext}
     {\rm ext}(\phi, \bm{u}_{r^*})= \frac{1}{r^*}\sum_{i \in [r^*]}\log\phi_{i,r^*+1} = Q(s).
     \end{equation}
     We first note that $r^*+1$ can see each colour at most once: if colour $c$ was contained in $\phi(i, r^*+1)$ and $\phi(j, r^*+1)$, there would be a dichromatic triangle in $\phi$ on parts $i,j,r^*+1$ since $\phi^*_{ij} \geq 2$. In fact $r^*+1$ sees each colour exactly once, for otherwise $\ext(\phi,\bm{u}_{r^*})$ cannot equal $Q(s)$.
     Let $\ell:=|\{i \in [r^*]: \phi_{i,r^*+1} \neq \emptyset\}|$. Note that $\mc X$ is hermetic so $\ell\le r^*-1$.

We claim that $\ell=r^*-1$.
     This is immediate when $r^*=2$, so we may assume $s \geq 27$. 
     Equation~\eqref{eq:ext} and \autoref{lm:maxoflog} imply that
     \begin{equation}\label{eq:ext2}
     Q(s) 
     \le (\ell/r^*) \log (s/\ell).
    \end{equation}
     Suppose for a contradiction that $\ell \leq r^*-2$. 
     We have that
     \begin{equation}\label{eq:ext3}
     \frac{s}{r^*-1} \geq 6
     \quad\text{and}\quad
     \frac{(r^*-1)^2}{s-r^*+1} \leq 1/2;
     \end{equation}
     these assertions are true for $27 \leq s \leq 496$ by~\autoref{table:1} where $r^*\leq 4$ and for larger $s$ they follow from $r^* \leq \log(s) +2$ (recall \eqref{eq:rs}). 
     Our assumption implies that $(s-r^*+1)/(r^*-1) < s/\ell$.
     The mean value theorem implies that there is $\frac{s-r^*+1}{r^*-1} < u < \frac{s}{\ell}$ such that
     $$
     \ell\left(\log\left(\frac{s-r^*+1}{r^*-1}\right)-\log\left(\frac{s}{\ell}\right)\right)
     =\frac{\ell}{u}\left(\frac{s-r^*+1}{r^*-1}-\frac{s}{\ell}\right)
     >\ell-\frac{s(r^*-1)}{s-r^*+1}.
     $$
     By Lemma~\ref{lm:cont1}, we have $Q(s) = q_1(\phi^*,\bm{u}_{r^*}) = \frac{r^*-1}{r^*}\sum_{2 \leq i \leq r^*}\log\phi^*_{1i}$.
    By~\eqref{eq:consmult}, for all $ij \in \binom{[r^*]}{2}$, we have $\phi^*_{ij} \in \{z,z+1\}$ where $z := \lfloor s/(r^*-1) \rfloor$.
     Thus
     \begin{eqnarray*}
         r^*\cdot Q(s)-\ell \log\left(\frac{s}{\ell}\right) &\geq& (r^*-1) \log\left(\frac{s}{r^*-1}-1\right)-\ell \log\left(\frac{s}{\ell}\right)
         \\
         &=& (r^*-1-\ell) \log\left(\frac{s-r^*+1}{r^*-1}\right) + \ell\left(\log\left(\frac{s-r^*+1}{r^*-1}\right)-\log\left(\frac{s}{\ell}\right)\right)
         \\ 
         &\geq&  (r^*-1-\ell) \log\left(\frac{s-r^*+1}{r^*-1}\right) -\frac{s(r^*-1)}{s-r^*+1}+ \ell
         \\
         &=& (r^*-1-\ell) \left(\log\left(\frac{s}{r^*-1}-1\right)-1\right)-\frac{(r^*-1)^2}{s-r^*+1}
         \\
         &\stackrel{\eqref{eq:ext3}}{\geq}& \log(5)-1-1/2 >0,
     \end{eqnarray*}     
contradicting~\eqref{eq:ext2} and hence proving the claim. 

Let $i\in[r^*]$ be the unique index such that $\phi_{i,r^*+1}=\emptyset$. We claim that $r^*+1$ must be a clone of $i$. To see this, note that for any $j\in[r^*]\sm \{i\}$ and any $c\in \phi(j,r^*+1)$, we must have $c\in \phi^*(ij)$, for otherwise the perfect matching on $[r^*]$ in colour $c$ will contain the edge $jj'$ for some $j'\neq i$ and $r^*+1,j,j'$ will form a dichromatic triangle. Conversely, if $c\in \phi^*(ij)$, $r^*+1$ necessarily sees $c$ but we cannot have $c\in \phi(j',r^*+1)$ for any $j'\neq j$. This is because according to $\phi^*$ colour $c$ forms a perfect matching on $[r^*]$, one of whose edges must be $j'j''$ for some $j''\neq i,j$, and this would yield a dichromatic triangle on $j',j'',r^*+1$. Therefore, $r^*+1$ is a clone of $i$, as claimed.

Finally, for~(iv), recall that each multiplicity $\phi^*_{ij}$ equals $z$ or $z+1$ by~\eqref{eq:consmult}.
Since $r^* \leq \log(s) +2$, each multiplicity is therefore at least
$\lfloor s/(\log(s)+1)\rfloor$. This is at least $\log(s)+2 \geq r^*$ when $s \geq 27$.
For $2 \leq s \leq 26$, we have $r^*=2$ by~\autoref{table:1}.
In this case, $z=s \geq 2=r^*$, as required.
\end{proof}

The properties of the dichromatic triangle family guaranteed by the previous lemma now allow us to apply our general exact result, \autoref{th:exact1}.
This result immediately tells us that every extremal graph is a complete partite graph with the right number of parts with roughly the correct sizes, so it remains to prove that the sizes are in fact as equal as possible.

\begin{proof}[Proof of \autoref{th:K32}]
The statement about $\Rev(s)$ is~\autoref{lm:Rsprecise}(i). Recall that by~\autoref{th:K23opt}, $\opt^*(s)$ consists of triples $(r^*,\phi^*,\bm{u}_{r^*})$ with $r^*\in \Rev(s) \subseteq 2\mb{N}$, uniform $\bm{u}_{r^*}$ 
and $\phi^*$ where every ${\phi^*}^{-1}(c)$ is a perfect matching, and the multiplicities $\phi^*_{ij}$ over pairs $ij$ are as equal as possible.
By~\autoref{lm:Rsprecise}(i) and~\eqref{eq:rs}, we have $r^* \leq \log(s)+2$. Given $r^* \in \Rev(s)$,
let $Y_{r^*}$ consist of all $\phi^*$ such that $(r^*,\phi^*,\bm{u}_{r^*})\in \opt^*(s)$.
In other words, $Y_{r^*}$ is the set of $r^*$-partite colour templates obeying~\autoref{cons:general}. 

It remains to prove the statements about $n$-vertex extremal graphs for sufficiently large $n$. Suppose that $1/n\ll\eps\ll 1/k,1/s$ and let $G$ be an $n$-vertex extremal graph. By~\autoref{th:exact1} (which is applicable by~\autoref{lm:diprop}(i)--(iii)) there is $(r^*,\phi^*,\bm{u}_{r^*})\in \opt^*(s)$ so that $G$ is a complete $r^*$-partite graph whose parts $W_1, \dots, W_{r^*}$ all have size between $n/(2r^*)$ and $2n/r^*$, and moreover, for at least a $(1-e^{-\eps n})$ proportion of the valid colourings $\chi\in F(G;(K_3^{(2)},s))$ there exists $\phi^* \in Y_{r^*}$ such that $\chi$ follows $\phi^*$ perfectly. For a complete $r^*$-partite graph $H$ and $\phi^* \in Y_{r^*}$, denote the set of such perfect colourings of $H$ by $X_H(\phi^*)$.

Before determining the exact part sizes of $G$, we quickly show that $F(G;(K_3^{(2)},s))$ is approximately equal to $\sum_{\phi^*\in Y_{r^*}}|X_G(\phi^*)|$. Intuitively, this is because very few perfect colourings can be assigned to more than one $\phi^*\in Y_{r^*}$.
Indeed, if $\chi\in X_G(\phi^*)\cap X_G(\phi)$ for some $\phi^*,\phi\in Y_{r^*}$,
then there is $ij \in \binom{[r^*]}{2}$ where $\phi^*(ij) \neq \phi(ij)$ and $\chi(xy) \in \phi^*(ij) \cap \phi(ij)$ for all $x\in W_i,y\in W_j$,
so the number of such colourings $\chi$ for a fixed $\phi^*$ is at most $$\phi^*_{ij}\left(\frac{\phi^*_{ij}-1}{\phi^*_{ij}}\right)^{(n/(2r^*))^2}|X_G(\phi^*)| \leq s\left(1-\frac{1}{s}\right)^{(n/(2r^*))^2}|X_G(\phi^*)| \leq e^{-n}|X_G(\phi^*)|,$$
say.
Thus the number of colourings of $G$ which follow $\phi^*$ perfectly and do not follow another $\phi$ is at least
$(1-e^{-n})|X_G(\phi^*)|$,
so
\begin{equation}\label{eq:FG}
F(G;(K_3^{(2)},s)) = (1\pm e^{-\eps n})\sum_{\phi^* \in Y_{r^*}}|X_G(\phi^*)|.
\end{equation}
We now show that $|X_H(\phi^*)|$ is maximised when $H$ is the Tur\'an graph.

\begin{secclaim}
Let $H$ be a sufficiently large complete $r^*$-partite graph with parts $V_1,\ldots,V_{r^*}$. Then for every $\phi^* \in Y_{r^*}$, we have $|X_T(\phi^*)|/|X_H(\phi^*)| \geq 2$ whenever $H$ is not isomorphic to the Tur\'an graph $T := T_{r^*}(n)$.
    \end{secclaim}

    \begin{proof}
Suppose not, and, without loss of generality, we may suppose that $|V_2|-|V_1| \geq 2$ is the largest among all class size differences. Consider the complete multipartite graph $\tilde{H}$ obtained by removing a vertex from $V_2$ and adding a vertex to $V_1$. Fix $\phi^*\in Y_{r^*}$ and for each integer $m$, let $J_m := \{ij \in \binom{[r^*]}{2}: \phi^*_{ij}=m\}$.
Since $\phi^*$ is optimal,~\eqref{eq:consmult} implies that $E(K_{r^*}) = J_z \cup J_{z+1}$
where $z := \lfloor s/(r^*-1)\rfloor$.
Note that each $J_m$ is regular since the multigraph which is the union of the ${\phi^*}^{-1}(c)$ is regular.

Let $\ell := \phi^*_{12}$.
Passing from $H$ to $\tilde{H}$, $|X(\phi^*)|$ increases by a factor of
$$
D := \frac{|X_{\tilde{H}}(\phi^*)|}{|X_H(\phi^*)|}
=\ell^{|V_2|-|V_1|-1} \prod_{y\in[r^*]\sm \{ 1,2\}}\left(\frac{\phi^*_{1y}}{\phi^*_{2y}}\right)^{|V_y|}.
$$
It suffices to show that $D \geq 2$, then the claim follows by a telescoping product.

To prove this, 
suppose first that $r^*=2$. Then $D=s^{|V_2|-|V_1|-1} \geq s \geq 2$, as required.
Thus we may assume $r^* \geq 4$, and so $s \geq 27$,
and $z \geq 9$.
Note that the factor $\phi^*_{1y}/\phi^*_{2y}$ equals $\frac{z}{z+1}$ if $y \in A := N_{J_z}(1) \cap N_{J_{z+1}}(2)$, equals $\frac{z+1}{z}$ if $y \in B := N_{J_{z+1}}(1) \cap N_{J_z}(2)$, and equals $1$ otherwise.
Since $J_z$ and $J_{z+1}$ are regular graphs partitioning the edge set of $K_{r^*}$, we have $|A|=|B|$. 
Thus
$$
D \geq \left(\frac{z}{z+1}\right)^{\sum_{y \in B}|V_y|-\sum_{y \in A}|V_y|}\ell^{|V_2|-|V_1|-1} \geq \left(\frac{z}{z+1}\right)^{|A|(|V_2|-|V_1|)}\ell^{|V_2|-|V_1|-1},
$$
as $|V_2|-|V_1|$ is the maximum difference between the size of two parts.
For the next step, we note that since $\frac{r^*-2}{2z} \leq \frac{(r^*)^2} {s}\leq 1$ we have, using $\log(1+x) \geq \frac{x}{x+1}$ for $x > -1$,
$$
\ell \left(\frac{z}{z+1}\right)^{(r^*-2)/2} \geq \exp\left(\log \ell-\frac{r^*-2}{2z}\right)\geq \exp\left(\log z-\frac{r^*-2}{2z}\right) \geq \exp(\log z-1)> 1.
$$
Using $|V_2|-|V_1| \geq 2$ and $|A| \leq (r^*-2)/2$, this allows us to bound
\begin{align*}
D &\geq  \ell^{-1}\left(\ell \left(\frac{z}{z+1}\right)^{(r^*-2)/2}\right)^{|V_2|-|V_1|}\geq \ell^{-1}\left(\ell^{\frac{2}{r^*-2}}\frac{z}{z+1}\right)^{r^*-2}\\
&= \begin{cases}
    z^{r^*-1}/(z+1)^{r^*-2} & \mbox{ if } \ell=z\\
    z^{r^*-2}/(z+1)^{r^*-3} & \mbox{ if }\ell=z+1.
\end{cases}
\end{align*}

Suppose that $r^*=4$. Then $s \geq 27$ so $z \geq 9$ and hence $z^3/(z+1)^2>7$,
and $z^2/(z+1)>8$, as required.
Thus we may assume that $r^* \geq 6$ and hence $s \geq 497$.
Since $\frac{r^*-2}{z} \leq \frac{2(r^*)^2}{s} \leq 1$, we have for $i\in\{2,3\}$ that
$$
D \geq z\left(\frac{z}{z+1}\right)^{r^*-i} \geq \exp\left(\log(z)-\frac{r^*-2}{z}\right) \geq \exp(\log(z)-1) > 2,
$$
completing the claim.
\end{proof}

The maximality of $F(G;(K_3^{(2)},s))$ and the claim now imply that $G\cong T_{r^*}(n)$, that is, the parts $W_i$ form an equipartition of $V(G)$.

It remains to accurately estimate $|X_{T_{r^*}(n)}(\phi^*)|$
for every $\phi^*\in Y_{r^*}$ to obtain the formula for $F(G;(K_3^{(2)},s))$.
For this, let $m,f$ be the unique positive integers with $0 \leq f < r^*$ such that $n = r^*m + f$.
We claim that
\begin{equation}\label{eq:t_r}
    e(T_{r^*}(n)) =: t_{r^*}(n)=\binom{r^*}{2}m^2+(r^*-1)mf+\binom{f}{2}.
\end{equation}
To see this, note that $f$ parts $W_1,\ldots,W_f$ have size $m+1$ and $r^*-f$ parts $W_{f+1},\ldots,W_{r^*}$ have size $m$. Choose an arbitrary labelling $W_i=U_i \cup \{v_i\}$ for each large part $W_i$ with $i \in [f]$ and $W_i=U_i$ for each small part with $i>f$.
The terms in the required expression are, in order, the number of edges between the $U_i$, between each $v_i$ and $U_j$ for $i \neq j$, and between the $v_i$.

Fix an optimal $\phi^*\in Y_{r^*}$ and again let $z=\lfloor \frac{s}{r^*-1} \rfloor$ and $a=s-z(r^*-1)$. 
The number of colourings $\chi$ of $K_{r^*}(U_1,\ldots,U_{r^*})$ which follow $\phi^*$ is exactly
$$
z^{m^2\left(\binom{r^*}{2}-ar^*/2\right)}(z+1)^{m^2ar^*/2}.
$$
For each $v_i$ with $i \in [f]$, the number of colourings following $\phi^*$ of edges $v_ix$ with $x \in \bigcup_{j \neq i} U_j$ is exactly $z^{m(r^*-1-a)}(z+1)^{ma}$.
The number of colourings of $\{v_iv_j:ij\in\binom{[f]}{2}\}$ is a function $C_\phi$
which is at least $z^{\binom{f}{2}}$ and at most $(z+1)^{\binom{f}{2}}$, and where, due to the symmetric form of $Y_{r^*}$, $\sum_{\phi \in Y_{r^*}}C_{\phi}$ depends only on $s, r^*$ and the remainder $f$ of $n$ modulo $r^*$.
Using~\eqref{eq:t_r}, we see that the number of colourings of $T_{r^*}(n)$ following $\phi^*$ is
\begin{align*}
|X_{T_{r^*}(n)}(\phi^*)|&=(z+1)^{m^2ar^*/2}z^{m^2\left(\binom{r^*}{2}-ar^*/2\right)} \cdot \left((z+1)^{ma}z^{m(r^*-1-a)}\right)^f \cdot C_\phi\\
&= (z+1)^{\left(t_{r^*}(n)-\binom{f}{2}\right)a/(r^*-1)}z^{\left(t_{r^*}(n)-\binom{f}{2}\right)(r^*-1-a)/(r^*-1)}\cdot C_{\phi}
\\
&= e^{\frac{r^*}{r^*-1}\gf_s(r^*)\left(t_{r^*}(n)-\binom{f}{2}\right)}\cdot C_{\phi}
= C'_{\phi}e^{\frac{r^*}{r^*-1}\gf(s)t_{r^*}(n)}
\end{align*}
where $C'_{\phi}=C_{\phi}e^{-\frac{r^*}{r^*-1}\gf(s)\binom{f}{2}}$, so that, again, $C=\sum_{\phi \in Y_{r^*}} C'_\phi$ is a function of only $s,r^*,f$.
By~(\ref{eq:FG}) we have
\begin{equation}\label{eq:F_formula}
F(T_{r^*}(n);(K^{(2)}_3,s))=(1+o(1))\sum_{\phi \in Y_{r^*}}C'_{\phi}e^{\frac{r^*}{r^*-1}\gf(s)t_{r^*}(n)}=(C+o(1))e^{\frac{r^*}{r^*-1}\gf(s)t_{r^*}(n)},
\end{equation}
as desired.
The number $|Y_{r^*}|$ of optimal $\phi^*$ is a function of $r^*,s$.
In fact,
$$
|Y_{r^*}| \geq \binom{s}{\underbrace{z,\ldots,z}_{r^*-1-a}\underbrace{z+1,\ldots,z+1}_{a}}
$$
which is a lower bound on the number of $\phi$ obeying~\autoref{cons:general}.
We note for future reference that the constant $C$ in the statement of the theorem therefore satisfies
\begin{equation}\label{eq:C}
    C \geq \frac{s!}{(z+1)!^{r^*-1}}z^{\binom{f}{2}}e^{-\frac{r^*}{r^*-1}\gf(s)\binom{f}{2}}.
\end{equation}

It remains to prove the last sentence of the statement. The set $\Rev(s)$ contains a single value for all $s \in [2,\maxs] \sm \{27\}$ from~\autoref{table:1} (proved in~\autoref{lm:Rsprecise}(iv)), while $\Rev(27)=\{2,4\}$.
Thus, to complete the proof, it remains to prove the assertion about $s=27$.
Clearly $F(T_2(n);(K^{(2)}_3,27))=27^{t_2(n)}$ since every colouring is valid.

Suppose $s=27$ and $r^*=4$. Then $z=9$ and since $3$ divides $27$ we have
$C_{\phi}=9^{\binom{f}{2}}$ and $e^{\frac{3}{4}\gf(27)} = 9$ so $C'_{\phi} = 1$.
Appealing to the remark after~\autoref{cons:general}, $|Y_4|$ is the number of triples $(A_1,A_2,A_3)$ where $A_1 \cup A_2 \cup A_3 = [27]$ is an equipartition. Thus $|Y_4|=\binom{27}{9,9,9}$ and therefore $F(T_4(n);(K^{(2)}_3,27))= (1+o(1))\binom{27}{9,9,9} 9^{t_4(n)}$, which is many times greater than $F(T_2(n);(K^{(2)}_3,27))=27^{t_2(n)}$.
So $T_4(n)$ is the unique optimal graph for $s=27$
and $F(n;(K^{(2)}_3,27))=(1+o(1))\binom{27}{9,9,9} 9^{t_4(n)}$.
\end{proof}

\begin{proof}[Proof of~\autoref{cor:inf}] \autoref{th:K32} implies that for each $s$, the set of $(K_3^{(2)},s)$-extremal graphs is a subset of $\{T_r(n):r\in R_2(s)\}$, so it suffices to show that each set $S_2(r)=\{s:r\in R_2(s)\}$ for $r\in 2\mathbb{N}$ contains an interval $[s^-(r),s^+(r)]$ that is disjoint from all other $S_2(r')$, $r'\neq r$. 
\autoref{lm:Rsprecise}(iii) implies that for all $r \in 2\mb{N}$, there is $s$ such that $\Rev(s)=\{r\}$ is a singleton. We also have that the interval $S_2(r)=[\tilde s_r,\tilde s_{r+2}-1]$ or $[\tilde s_r,\tilde s_{r+2}]$ overlaps with each of $S_2(r-2)$ and $S_2(r+2)$ in at most one element, so we can take $s^-(r)=\tilde s_r$ or $\tilde s_r+1$ and $s^+(r)=\tilde s_{r+2}$ or $\tilde s_{r+2}-1$. The bounds on $s_r$ in \autoref{lm:Rsprecise}(iii) yield 
$$(r-3)e^{r-2}<s_{r-1}\le \tilde s_r\le s_r<(r-1)e^r,$$
as required.
\end{proof}

\section{Forbidding improperly coloured cliques}\label{sec:proper}

In this section, we prove~\autoref{th:proper} on the improper pattern, which follows with some extra work from the proof of~\autoref{th:K32}.
We first determine the set of basic optimal solutions and show that the pattern satisfies the hypotheses of~\autoref{th:exact1}.

\begin{lemma}\label{lm:impprop}
    Let $s\geq 2$ and $k \geq 3$ be integers and let $\x = \x_{k,s}^{\wedge}$ be the family of all improper $s$-edge colourings of $K_k$. Then $\x$
    \begin{enumerate}[label=\emph{(\roman*)}]
        \item has $Q(\x) = \max\{\frac{k-2}{k-1}\log(s), \gf(s)\}$ and basic optimal solutions $$
\textstyle\opt^*(\x) =
\begin{cases}
\{(k-1,\phi_{[s]},\bm{u})\} &\mbox{ if }  \frac{k-2}{k-1}\log(s) > \gf(s)\\
\{(k-1,\phi_{[s]},\bm{u})\} \cup \opt^*(K^{(2)}_3,s) &\mbox{ if } \frac{k-2}{k-1}\log(s) = \gf(s)\\
\opt^*(K^{(2)}_3,s) &\mbox{ if } \frac{k-2}{k-1}\log(s) < \gf(s)
\end{cases}
$$
where $\phi_{[s]} \equiv [s]$ and $\bm{u}$ is uniform;
and in particular $\x$ is bounded, 
        \item is hermetic (and hence stable inside),
        \item has the extension property.
    \end{enumerate}
\end{lemma}

\begin{proof}
Let $(r,\phi,\ba) \in \feas_{\x,2}(s)$.
Suppose first that $r \leq k-1$. Then 
$q(\phi,\ba)=2\sum_{ij}\aA_i\aA_j\log\phi_{ij} \leq 2\log(s)\sum_{ij}\aA_i\aA_j \leq (1-\frac{1}{r})\log(s) \leq (\frac{k-2}{k-1})\log(s)$. This is uniquely attained by the solution $(k-1,\phi_{[s]},\bm{u})$ where $\phi_{[s]} \equiv [s]$ and $\bm{u}$ is uniform.
Suppose instead that $r \geq k$.
Then $\phi^{-1}(c)$ is a matching for all $c \in [s]$, otherwise there would be an improperly coloured clique (with two adjacent edges of colour $c$).
By \autoref{rm:reallyproving}, we have $q(\phi,\ba) \leq \gf(s)$ with equality if and only if $r \in \Rev(s)$, $\ba$ is uniform and $\phi=\phi_{\mc{A}}$ for some $\mc{A}$ as in \autoref{cons:general}.
Thus $Q(\x)$ and the set of basic optimal solutions are as stated.
    Clearly $\x$ is bounded since the dichromatic triangle pattern is bounded.
    
    For parts~(ii) and~(iii), since the dichromatic triangle family $(K^{(2)}_3,s)$ is hermetic and has the extension property, it suffices to check the conditions for the solution $(k-1,\phi_{[s]},\bm{u})$ assuming that $Q(\x)=\frac{k-2}{k-1}\log(s)$.
    Let $\phi \in \Phi_{\x,1}(k)$ be an extension of $\phi_{[s]}$. 
    Let $c \in \phi(1k)$. Then $c \in \phi(12)=[s]$ so if $|\phi(ik)| \geq 1$ for all $i \in [k-1]$, there is a $k$-clique with adjacent edges of colour $c$, a contradiction. Thus $\x$ is hermetic.
    We have also shown that any extension $\phi$ must have $\phi(ik)=\emptyset$ for some $1\leq i \leq k-1$. Without loss of generality, suppose $\phi(1k)=\emptyset$.  Then the contribution of $k$ is
    $
    \sum_{2 \leq i \leq k-1}\frac{1}{k-1}\log|\phi(ik)| \leq \frac{k-2}{k-1}\log(s)=Q(\x)
    $
    with equality if and only if $\phi(ik)=[s]$ for all $2 \leq i \leq k$.
    In other words, if and only if $k$ is a clone of $1$ under $\phi_{[s]}$. Thus $\x$ has the extension property.
\end{proof}

Given our proof of~\autoref{th:K32}, it is now an easy task to prove our second main result.

\begin{proof}[Proof of \autoref{th:proper}]
Let $\x=\x_{k,s}^\wedge$ be the family of all improper $s$-edge colourings of $K_k$.
In~\autoref{lm:impprop} we determined $Q(\x)$ and $\opt^*(\x)$ and showed that the hypotheses of~\autoref{th:exact1} hold.

Let $0 < 1/n \ll \eps \ll \dD \ll 1/s,1/k$ and let $G$ be an $\x$-extremal graph on $n$ vertices.
~\autoref{th:exact1} implies that there is $(r^*,\phi^*,\ba^*) \in \opt^*(\x)$ such that $G$ is a complete $r^*$-partite graph whose $i$-th part $W_i$ has size $(\alpha_i^*\pm\dD)n$ for all $i \in [r^*]$, and~\ref{setwo} holds. 

If $(r^*,\phi^*,\ba^*)=(k-1,\phi_{[s]},\bm{u})$, then $G$ is $K_k$-free so every colouring of $G$ is valid, so $F(n;\x) = s^{\sum_{ij\in\binom{[k-1]}{2}}|W_i||W_j|}$ which is uniquely maximised when $||W_i|-|W_j|| \leq 1$ for all $i,j$;
that is, we must have $G \cong T_{k-1}(n)$ which has exactly $s^{t_{k-1}(n)}$ valid colourings.
On the other hand, if $(r^*,\phi^*,\ba^*) \in \opt^*(K^{(2)}_3,s)\sm \{(k-1,\phi_{[s]},\bm{u})\} $, \autoref{th:K32} implies that $G \cong T_{r^*}(n)$ for some $r^*\in \Rev(s)$ and $F(G;(K^{(2)}_3,s))=(C+o(1))e^{\frac{r^*}{r^*-1}\gf(s)t_{r^*}(n)}$ where
$C$ satisfies~\eqref{eq:C}
and $\gf(s)=\gf_s(r^*)$.
Recall that by \autoref{lm:Rsprecise}(i) and \eqref{eq:rs} we have
\begin{equation}\label{eq:rstar}
    r^* \leq \log(s)+2.
\end{equation}
We now first show that 
\begin{equation}\label{eq:t19_mono_s}\mbox{if }r^*\ge k\mbox{ and } \gf_s(r^*) \geq \frac{k-2}{k-1}\log(s)\mbox{, then }\gf_{s+1}(r^*)>\frac{k-2}{k-1}\log(s+1).\end{equation}

We verify \eqref{eq:t19_mono_s} manually for all $k\ge 3$ and $k-1\le s\le 1231$ and $r^*\in \Rev(s)\sm \{k-1\}$ by a computer search given in the ancillary file \texttt{improper\_patterns.py}.
Suppose now that $s\ge 1232$. Then
\begin{align*}
g_{s+1}(r^*)-g_s(r^*)-\frac{k-2}{k-1}\log(s+1)+\frac{k-2}{k-1}\log(s)
&\geq \left(\frac{r^*-1}{r^*}-\frac{k-2}{k-1}\right)\log\left(\frac{s+1}{s}\right)-e_{s+1}(r^*)\\
&>\left(\frac{1}{k-1}-\frac{1}{r^*}\right)\frac{1}{s+1}- \frac{1}{4}\left\lfloor \frac{s+1}{r^*-1}\right\rfloor^{-2}>0
\end{align*}
using $k \leq r^* \leq \log s + 2$ and $s\ge 1232$ in the final inequality. So \eqref{eq:t19_mono_s} holds for all $s$.
Define
$$
\overline{s}(k) := \min\left\{s \in \{2,3,\ldots\}: g(s)\ge \frac{k-2}{k-1}\log(s)
\text{ and }\min\Rev(s) \geq k\right\}.
$$
We will show that $s(k)=\overline{s}(k)-1$ satisfies the theorem.
Note that if there is $r^* \in \Rev(s)$ with $r^* < k$, then $\gf(s) \leq \frac{k-2}{k-1}\log(s)$.
Thus, by comparing the leading terms in the exponent of the formulas we gave above for $F(T_{k-1}(n);\mc X)$ and $F(T_{r^*}(n);\mc X)$, we obtain from~\eqref{eq:t19_mono_s} that 
\begin{itemize}
    \item if $s<\overline{s}(k)$, we have $g(s)\le \frac{k-2}{k-1}\log(s)$ and $G\cong T_{k-1}(n)$;
    \item if $s>\overline{s}(k)$, we have $g(s)> \frac{k-2}{k-1}\log(s)$ and thus $G\cong T_{r^*}(n)$.
\end{itemize}
It remains to show that for $s=\overline{s}(k)$, we have $G\cong T_{r^*}(n)$.
Suppose $g(s) > \frac{k-2}{k-1}\log(s)$. Then we can deduce as before that $G\cong T_{r^*}(n)$ for some $r^* \in \Rev(s)$.

So suppose $g(s) = \frac{k-2}{k-1}\log(s)$ and $\min\Rev(s) \geq k$.
First note that by our definition of $\overline{s}(k)$ we have $\overline{s}(k)\ge 27$ for all $k$ since $\Rev(s)=\{2\}$ for $s\le 26$. Thus for the calculations that follow we may assume that $s\ge 27$. By making the appropriate substitutions we can deduce from \eqref{eq:t_r} that for an integer $r$, the number of edges in the Tur\'an graph $T_r(n)$ is 
$$t_r(n)=\frac{r-1}{r}\frac{n^2}{2} - \frac{i}{2}\left(1-\frac{i}{r}\right),$$
where $0 \leq i \leq r-1$ and $i\equiv n\mod r$.
Thus, letting $f,b$ be such that $0 \leq f \leq r^*-1$ and $f\equiv n \mod r^*$ and $0 \leq b \leq k-2$ and $b \equiv n\mod (k-1)$, and setting $z:=\left\lfloor\frac{s}{r^*-1}\right\rfloor$,
we have, using $e^{\gf(s)}=s^{\frac{k-2}{k-1}}$,
\begin{eqnarray*}
\frac{F(T_{r^*}(n);\x)}{F(T_{k-1}(n);\x)}&\stackrel{\eqref{eq:C}}{\geq}&\frac{(1-o(1))\frac{s!}{((z+1)!)^{r^*-1}}z^{\binom{f}{2}}\exp\left(-\frac{r^*}{r^*-1}g(s)\binom{f}{2}\right)\exp\left(-\frac{r^*}{r^*-1}g(s)\frac{f}{2}\left(1-\frac{f}{r^*}\right)\right)}{\exp\left(-\frac{b}{2}\left(1-\frac{b}{k-1}\right)\log(s)\right)}\\
&\ge&\frac1{2}\frac{s!}{((z+1)!)^{r^*-1}}z^{\binom{f}{2}}e^{-\log(s)f^2/2}.
\end{eqnarray*}
The second inequality follows by expanding the numerator, substituting $\gf(s)\le \hf(s)\le \log(s)$ (which holds due to \autoref{lm:esr}, \autoref{lm:hanalytic}(ii) and \eqref{eq:W}), and noting that the exponent in the denominator is non-positive. Also,
using $s \geq 27$ and $z \geq \frac{s}{\log(s)+1}-1 \geq \frac{s}{2\log(s)}$ and $f\le \log(s)+1$ (which follow from~\eqref{eq:rstar}), 

$$\binom{f}{2}\log(z)-\frac{f^2}{2}\log(s)\stackrel{\eqref{eq:rstar}}{\ge}\binom{f}{2}\log\left(\frac{s}{2\log(s)}\right)-\frac{f^2}{2}\log(s) \ge -2\log^2(s)\log\log(s).$$
Putting this together with the Stirling bounds $e(\ell/e)^\ell \le \ell!\le e\ell(\ell/e)^\ell$ applied to $\ell=s,z$, and~\eqref{eq:rstar}, we obtain
$$\frac{F(T_{r^*}(n);\x)}{F(T_{k-1}(n);\x)}>1,$$
as required.
\end{proof}

\section{Proofs of general results}\label{sec:proofs}

This section contains the proofs of~\autoref{th:asymptotic}--\autoref{th:exact1}.
These proofs are mainly adaptations of proofs in~\cite{katherine_exact,katherine_stability}, which are versions of our results for monochromatic colour patterns. 
The arguments transfer almost directly from the monochromatic setting -- the key point being that in both settings, the family of forbidden colourings are on \emph{cliques}. This means that, informally, colours at two non-adjacent vertices are `independent' since they cannot be in a forbidden clique together, and thus it can be shown that the following holds:
\begin{enumerate}[label=(S),ref=(S)]
    \item\label{it:S} for any pair $u,v$ of non-adjacent vertices, either replacing $u$ by a twin of $v$ or $v$ by a twin of $u$ (the operation of \emph{symmetrisation}) gives a graph with at least as many valid colourings. 
\end{enumerate}
This is stated in e.g.~\cite[Lemma 2.7]{benevides2017edge}.
Symmetrisation was originally introduced by Zykov in the 1950s~\cite{zykov, zykov_en} to give a new proof of Tur\'an's theorem.

The original proofs in~\cite{katherine_exact,katherine_stability} are rather long. We therefore give only a sketch for those parts which are not new, and refer the reader to the relevant reference.

\subsection{Tools}\label{sec:tools}
We start by collecting some tools concerning the optimisation problem. Versions of most of these for monochromatic patterns appear in~\cite{katherine_asymptotic,katherine_exact,katherine_stability},
but usually the same proof goes through verbatim.
We limit ourselves to brief remarks where this is not quite true.

The first lemma, a generalisation of~\cite[Lemma~2.8]{katherine_stability}, states that sizes of parts of basic optimal solutions are bounded below.
It is proved via a compactness argument, which requires that there is a finite number of possible $\phi^*$. This holds when $\x$ is bounded (we recall from \autoref{sec:optk1} that the monochromatic pattern is bounded).

\begin{lemma}\label{lm:mu}
Let $s \geq 2$ and $k \geq 3$ be integers and let $\x$ be a bounded family of $s$-edge colourings of $K_k$ that has the extension property.
Then there exists $\mu>0$ such that for all $(r^*,\phi^*,\ba^*) \in \opt^*(\x)$, we have $\aA_i^*>\mu$ for all $i \in [r^*]$.
\end{lemma}

The next lemma states that $q$ is continuous in its second argument and follows from simple calculus.

\begin{lemma}[Proposition~11,\cite{katherine_asymptotic}]\label{lm:lipschitz}
Let $s \geq 2$ and $k \geq 3$ be integers and let $\x$ be a family of $s$-edge colourings of $K_k$.
Then for any $r \in \mathbb{N}$, $\phi\in \Phi_{\x,0}(r)$ and $\ba,\bm{\beta}\in \Delta^r$ we have
$$
|q(\phi,\ba)-q(\phi,\bm{\beta})|<2\log(s)\|\ba-\bm{\beta}\|_1.
$$
\end{lemma}

Next, we have that almost optimal solutions have the property that the vertex weighting can be perturbed slightly to produce an optimal solution. This is another compactness argument and generalises~\cite[Claim~15]{katherine_asymptotic}.

\begin{lemma}\label{lm:near}
Let $s \geq 2$ and $k \geq 3$ be integers and let $\x$ be a bounded family of $s$-edge colourings of $K_k$.
For all $\dD>0$ there exists $\eps>0$ such that the following holds. Let $(r,\phi,\ba) \in \feas_{1}(\x)$ be such that $q(\phi,\ba) \geq Q(\x)-\eps$. Then there exists $\ba^* \in \DD^r$ such that $\|\ba-\ba^*\|_1 < \dD$ and $(r,\phi,\ba^*) \in \opt_1(\x)$.
\end{lemma}

Our next tool is a version of~\cite[Lemma 2.10]{katherine_stability} which states that when a bounded family has the extension property, any vertex attached to a basic optimal solution with almost optimal contribution must be a clone of an existing vertex.
(Recall that the extension property (\autoref{extprop}) guarantees this to be true when the contribution is optimal rather than almost optimal.)

\begin{lemma}\label{lm:ext}
Let $s \geq 2$ and $k \geq 3$ be integers and let $\x$ be a bounded family of $s$-edge colourings of $K_k$ that has the (strong) extension property.
Then there exists $\eta>0$ such that the following holds.
    Let $(r^*,\phi^*,\ba^*) \in \opt^*(\x)$ and $\phi \in \Phi_{\x,0}(r^*+1)$ be such that $\phi|_{\binom{[r^*]}{2}} = \phi^*$ and
    ${\rm ext}(\phi,\ba^*) > Q(\x)-\eta$.
    Then $r^*+1$ is a (strong) clone of some $i \in [r^*]$ under $\phi$.
\end{lemma}

The final result is a new lemma which allows us to find a common optimal vertex weighting given several optimal solutions with similar vertex weightings. It will be used in the proof of~\autoref{th:exact1}.

\begin{lemma}\label{lem:uni_ba}
Let $s \geq 2$ and $k \geq 3$ be integers and let $\x$ be a family of $s$-edge colourings of $K_k$. Let $r \geq 2$ be an integer and suppose that $\Phi \subseteq \Phi_{\x,0}(r)$, and that for each $\phi \in \Phi$, we have $\ba_\phi \in \DD^r$ such that $(r,\phi,\ba_\phi)\in\opt(\x)$. Then for any $\eps>0$ there is $\dD>0$ such that the following holds. If for all $\phi,\psi \in \Phi$ we have $\|\ba_\phi-\ba_{\psi}\|_1<\dD$, then there is $\ba^*\in \DD^r$ such that for all $\phi \in \Phi$ we have $(r,\phi,\ba^*)\in\opt(\x)$ and $\|\ba_\phi-\ba^*\|_1<\eps.$
\end{lemma}

\begin{proof}
Suppose that the lemma does not hold for some $\eps>0$. Then for each integer $m \geq 1$, there is $\Phi_m \subseteq \Phi_{\x,0}(r)$ and a family $(r, \phi,\ba^{(m)}_\phi)_{\phi \in\Phi_m}$ with $\|\ba_\phi-\ba_\psi\|_1<\frac{1}{m}$ for all $\phi,\psi \in \Phi_m$ and such that, if there is $\ba^* \in \Delta^r$ so that $(r,\phi,\ba^*) \in \opt^*(\x)$ for all $\phi \in \Phi_m$, then $\|\ba_\phi^{(m)}-\ba^*\|_1 > \eps$ for some $\phi \in \Phi_m$.

We can restrict to a subsequence  where some $\phi$ either never appears or appears infinitely often. Since $\Phi_{\x,0}(r)$ is finite, we may iteratively do this for all $\phi$ to obtain a subsequence ${(m_\ell)}_{\ell \in \mathbb{N}}$ such that for all $\ell \geq 1$ the set $\Phi_{m_\ell}$ is equal to some fixed set $\Phi$.
Owing to the compactness of $\DD^r$, by restricting the sequence ${(m_\ell)}_{\ell \in \mathbb{N}}$ further, we may assume that for all $\phi \in \Phi$, $(\ba_\phi^{(m_\ell)})_\ell$ converges. Since for any $\phi,\phi' \in \Phi$ we have $\|\ba_\phi^{(m_\ell)}-\ba_{\phi'}^{(m_\ell)}\|_1 \to 0$  as  $\ell \to \infty$, we get the same limit $\ba^*$ of the sequences $(\ba_\phi^{(m_\ell)})_\ell$ for $\phi \in \Phi$. 
Because $(r^*,\phi,\ba_\phi^{(m_\ell)})$ is optimal for all $\phi \in \Phi$ and  $\ell \geq 1$, the continuity of $q$ implies that $(r^*,\phi,\ba^*) \in \opt(\x)$ for all $\phi \in \Phi$. This yields a contradiction, since we can choose $L$ so that $\|\ba_\phi - \ba^*\|_1 < \eps$ for all $\phi \in \Phi = \Phi_{m_L}$.
\end{proof}

\subsection{Proof of \autoref{lm:stability_opt}}\label{sec:proofs:stability_opt}
The proof of \autoref{lm:stability_opt} is largely the same as that of~\cite[Lemma~3.1]{katherine_stability} to which the reader is referred for details. 
The idea is to repeatedly `symmetrise' an almost optimal solution, and then show that actually not much changed during this process.

\begin{proof}[Sketch proof of \autoref{lm:stability_opt}]
Let $\x$ be a bounded family of $s$-edge colourings of $K_k$ which has the extension property.
Write $Q := Q(\x)$.
Let $\nu>0$ be given.
Let $\mu$ be the constant obtained from~\autoref{lm:mu}, let $0 < \gG\ll \delta \ll \mu$ and 
obtain $\eps^{1/4}$ by applying~\autoref{lm:near} with $\gG^2$.
Without loss of generality, we have $0<\eps \ll \gamma \ll \delta\ll  \mu \ll \nu,1$.
Suppose that $(r,\phi,\ba) \in \feas(\x)$ has $q(\phi,\ba) > Q-\eps$.
By splitting up parts we can assume that every part $\alpha_i$ has equal size $1/r \ll \eps$.
The first main step in the proof is the \emph{forwards symmetrisation} procedure which is the same as in \cite{katherine_stability}. The only property required of solutions $(r,\phi,\ba)$ is a version~\ref{it:S'} of~\ref{it:S} stating that symmetrising produces a feasible solution and $q$ does not decrease
(we already used this to prove~\autoref{fact:opts}).
\begin{enumerate}[label=(S$'$),ref=(S$'$)]
    \item\label{it:S'} In a feasible solution $(r,\phi,\ba)$, for any $hj \in \binom{[r]}{2}$ with $\phi_{hj} \leq 1$, 
    such that, after possibly swapping the labels $h,j$, then
    defining $\ba' \in \DD^r$ with $\aA_h':=\aA_h+\aA_j$ and $\aA_j' := 0$ and $\aA_\ell' := \aA_\ell$ for all $\ell \in [r]\sm\{h,j\}$ (the operation of \emph{symmetrisation}), we have that $(r,\phi,\ba')$ is a feasible solution and $q(\phi,\ba') \geq q(\phi,\ba)$. 
\end{enumerate}
This procedure yields~\autoref{cl:forwards}, which we will state after introducing notation.
Let $\mathcal{V}_0=\{\{1\},\ldots,\{r\}\}$ and $\phi_0=\phi$.
Given $r'\leq r$ and colour templates $\psi' \in \Phi_{\x,0}(r'),\psi \in \Phi_{\x,0}(r)$ and a partition $\mathcal{V}=\{V_1, \dots, V_{r'}\}$ of $[r]$ into $r'$ parts, we say that $\psi=_{\mathcal{V}} \psi' $ if $\psi(i'j')=\psi'(ij)$ for all $i' \in V_i$, $j' \in V_j$ and $ij\in\binom{[r']}{2}$, and $\psi(i'i'')=\emptyset$ whenever $i',i'' \in V_i$ for all $i \in [r']$.

\begin{secclaim}[See Claim~3.1.1 in~\cite{katherine_stability}]\label{cl:forwards}
There is $f \in \mathbb{N}$ 
and $2 \leq r_f \leq \ldots \leq r_0=r$
such that, after relabelling $[r]$, for all $i=0,\ldots,f$, there is a partition $\mathcal{V}_i$ of $[r]$ with $|\mc{V}_i|=r_i$
and a colour template $\phi_i \in \Phi_{\x,0}(r)$ such that the following hold.
\begin{itemize}
\item There is a single $x_i \in [r]$ such that $\mathcal{V}_i$ consists of the same elements as $\mathcal{V}_{i-1}$, except that $x_i$ has been moved from one part to another and any empty parts are deleted;
\item $\phi_i =_{\mathcal{V}_i}\psi_i$ where $\psi_i:=\phi|_{\binom{[r_i]}{2}}$
and $\psi_f \in \Phi_{\x,2}(r_f)$;
\item $q(\phi_i,\ba)-q(\phi_{i-1},\ba) \geq 0$.
\end{itemize}
\end{secclaim}
The claim is obtained by repeatedly applying \ref{it:S'}. We start with $\phi_0, \psi_0=\phi_0$ and $\ba_0=\ba$ and $\mc V_0$. While $\psi_i\not \in \Phi_{\x,2}(r_i)$, let $hj\in\binom{[r_i]}{2}$ be such that $|\psi_i(hj)|\le 1$ as in \ref{it:S'} and take $x_{i+1}$ in the $j$-th class of $\mc V_i$. If $x_{i+1}$ is a singleton, set $r_{i+1}:=r_i-1$, otherwise $r_{i+1}:=r_i$. Let $\ba_{i+1}$ be obtained by moving $\alpha_{x_{i+1}}$ from $\alpha_{i,j}$ to $\alpha_{i,h}$ as in \ref{it:S'}.  Take $\phi_{i+1}$ to be $\phi_i$ with the edges incident to $x_{i+1}$ replaced with the edges incident to any vertex in the $h$-th class of $\mc V_i$ (which are all clones of each other). It can be shown that $q(\phi_{i+1},\ba)=q(\psi_{i+1},\ba_{i+1})$, which together with \ref{it:S'} yields $q(\phi_{i+1},\ba)\ge q(\phi_{i},\ba)$.

We next describe \emph{backwards symmetrisation}, which goes backwards through the steps of forwards symmetrisation, removing parts with small contribution and assigning parts with large contribution into new groups depending on where they ended up in the forwards process.
For $(r,\phi,\ba) \in \Phi_{\x,0}(r)$ and $P \subseteq [r]$, write
$q_x(P,\phi) := \frac{1}{r}\sum_{
y \in P\setminus\{x\} 
}\log\phi_{xy}$
(recalling each $\alpha_y=\frac{1}{r}$).
Let $\mathcal{V}_i:=\{V_{i,j}: j \in [r_i]\}$.
We do the following (backwards symmetrisation).
Let us define $\mathcal{U}_f:=\{U^0_f,\ldots,U^{r_f}_f\}$ by setting $U^j_f:=V_{f,j}$
for all $j \in [r_f]$ and $U^0_f:=\emptyset$,
and let $U_f := [r]$.
For each $i=f-1,\ldots,0$ define $U_i$ and $\mathcal{U}_i=\{U^0_i,\ldots,U^{r_f}_i\}$ inductively as follows.
Initialise with $U_i=[r]$ and $U^0_i=\emptyset$.
\begin{enumerate}
    \item \emph{Add vertices with small contribution to the exceptional set}:
    \begin{enumerate}
        \item If $\frac{r}{|U_i|}q_{x_i}(U_i,\phi_i)<Q-\sqrt{\eps}$, move $x_i$ from $U_i$ into $U_i^0$.
        \item Repeat the following until no longer possible, updating $U_i$ each time:
    If there is $y \in U_i$ such that $\frac{r}{|U_i|}q_y(U_i,\phi_i)<Q-\sqrt{\eps}$, move $y$ into $U^0_i$.
    (Note that $y$ could be $x_i$ if its contribution becomes too small after moving other vertices into $U^0_i$).
    \end{enumerate}
    \item \emph{Group $x_i$ and newly non-exceptional vertices via comparison to $\psi_f$}:
    For each $j \in [r_f]$, let $U^j_i$ be the restriction of $U^j_{i+1}$ to $U_i\setminus\{x_i\}$. For each $z$ in $$
    B_i:=(U_{i+1}^0 \cup \{x_i\}) \cap U_i,
    $$
    add $z$ to the part $U^j_i$ such that $z$ looks most like a $\phi_f$-clone of $j$ under $\phi_i|_{U_i}$. That is, choose the $j$ such that
\begin{align}\label{eq:close}
|\{y \in U^j_i: |\phi_i(zy)| \geq 2\}|+ \sum_{j' \in [r_f]\setminus\{j\}}|\{y \in U^{j'}_i: \phi_i(zy) \neq \phi_f(j'j)\}|
\end{align}
is minimal.
\end{enumerate}

We note that after moving vertices we still have $U_i = [r] \sm U^0_i$. The next claim is also unchanged.
Its proof follows from the fact that any vertex moved into $U_i^0$ in (1) has much smaller than average contribution since the average contribution of a vertex is at least $Q-\eps$. (It is closely related to the familiar procedure in graph theory of repeatedly removing a vertex of small degree, which cannot remove too many vertices in a graph with many edges.)
\begin{secclaim}[See Claim~3.1.2,~\cite{katherine_stability}]
For all $i=f,\ldots,0$, we have $|U^0_{i}| \leq 2\sqrt{\eps}r$.
\end{secclaim}

It suffices to show that the parts $U_0^1,\ldots,U_0^{r_f}$ are essentially the ones required by the theorem (so in particular, almost all parts are subparts of one of these).
For this, we argue inductively for $i=f,\ldots,0$.
Let $G_i$ be the complete graph with vertex set $U_i$.
For each $x \in U_i$, let $j_x \in [r_f]$ be such that $x \in U_i^{j_x}$.
Then we colour each $xy \in E(G_i)$ as follows:
\begin{itemize}
\item $xy$ is \emph{red} if $j_x \neq j_y$ and $\phi_i(xy) \subsetneq \psi_f(j_xj_y)$, so there are missing colours,
\item $xy$ is \emph{blue} if either
$j_x \neq j_y$ and $\phi_i(xy) \setminus \psi_f(j_xj_y) \neq \emptyset$, or 
$j_x=j_y$ and $|\phi_i(xy)|>1$, so there are extra colours,
\item $xy$ is \emph{green} otherwise. 
\end{itemize}

The next claim is the heart of the proof. 

\begin{secclaim}[See Claim~3.1.3,~\cite{katherine_stability}]
$G_i$ has no blue edges. \label{cl:blue_edges}
\end{secclaim}

\begin{proof}[Sketch proof of claim.]
As promised, the proof is by backwards induction on $i$.
After sketching the first part of the proof, which is almost identical to~\cite{katherine_stability}, we give more details at the end in showing that any blue edge leads to a contradiction, where the argument diverges slightly.

First, the claim is true for $i=f$ as every edge of $G_f$ is green.
Fix $i < f$ and let $n_i:=|U_i|=r-|U_i^0|\geq (1-2 \sqrt{\eps})r$.
Since $G_{i+1}$ and $G_i$ only differ at $B_i$, the induction hypothesis implies that every blue edge of $G_i$ is incident with $B_i$.
Let $\bb_i := (|U^1_i|/n_i,\ldots,|U^{r_f}_i|/n_i) \in \DD^{r_f}$. 
By construction, every $x \in U_i$ satisfies $\frac{r}{|U_i|}q_x(U_i,\phi_i) \geq Q-\sqrt{\eps}$;
that is, every non-exceptional vertex has large contribution.
Since red edges at a vertex reduce its contribution, summing these contributions implies that
$$
e_{\rm red}(G_i) \leq \eps^{1/3}n_i^2
\quad\text{and}\quad
q(\psi_f,\bb_i) \geq Q-\eps^{1/3},
$$
so $\bb_i$ is close to an optimal vertex weighting.
By our choice of $\eps$ there is a vertex weighting $\ba_i'$
with $\|\ba_i' - \bb_i\|_1 \leq \gG^2$
such that $(r_f,\psi_f,\ba_i')$ is in fact optimal.
Relabelling parts and removing zero parts, we have a basic optimal solution $(\tilde{r}_i,\psi_f,\tilde{\ba}_i) \in \opt^*(\x)$, which is a slight abuse of notation since we really mean the restriction of $\psi_f$ to $[\tilde{r}_i] \subseteq [r_f]$.

Next, we show that, in $\phi_i$, every $z \in B_i$ is \emph{$\frac{\delta}{2}$-close} to being a $\psi_f$-clone of some $j \in [\tilde{r}_i]$,
meaning that one can change at most $\frac{\delta}{2}$ proportion of pairs at $z$ so that its attachment is the same as the attachment of $j$ in $\psi_f$.
Here is where we need the extension property
(\autoref{extprop}),
which states that any vertex added to a basic optimal solution with \emph{maximal} contribution $Q$ is in fact a clone of one of the existing vertices.
A version of this remains true for vertices which have \emph{almost} maximal contribution (\autoref{lm:ext}), and for a small perturbation of a basic optimal solution, as $\phi_i$ is of $\psi_f$
(see~\cite[Lemmas 2.10--2.12]{katherine_stability} which are for monochromatic patterns but generalise straightforwardly). This version can be used to show that indeed any $z \in B_i$ is $\frac{\delta}{2}$-close to being a clone of some vertex as claimed above. Owing to the way we allocated $z$ to a part in the backwards symmetrisation, this yields that $z$ is added to a part $U^{j_z}_i$ such that $z$ is $\frac{\delta}{2}$-close to a $\psi_f$-clone of $j_z$ under $\phi_i$. Since the closeness measure (\ref{eq:close}) counts non-green edges, a short calculation reveals that almost every edge in $G_i$ incident to $z$ is green:
$$
d_{\rm green}(z) \geq (1-\delta)n_i
\quad\text{for all }z \in B_i.
$$
The next step is to show that every $x \in U_i \sm B_i$ has small red degree in $G_i$ (and hence again almost every incident edge is green). Indeed, $x$ can only have blue edges incident with the small set $B_i$, and thus too many red edges mean the contribution of $x$ in $\phi_i$ is too small. We have
$d_{\rm red}(x) \leq \gG n_i$
and hence $d_{\rm green}(x) \geq (1-2\gG)n_i$ for all $x \in U_i \sm B_i$.
Altogether, we have shown that
$$
d_{\rm green}(x) \geq (1-\delta)n_i
\quad\text{for all }x \in U_i.
$$
A short calculation shows that this implies that for every $y \in U_i$ we have $q_{j_y}(\psi_f,\ba_i') \geq Q - \sqrt{\dD}$.
The extension property now implies that $\bb_i$ has the same support as $\ba_i'$ (which is the support of $\tilde{\ba}_i$ by definition); i.e.~$[\tilde{r}_i]$.
In particular, $U_i=\bigcup_{j \in [\tilde{r}_i]}U_i^j$.
To see this implication, since every vertex in $U^j_i$ with $\tilde{r}_i < j \leq r_f$ has large contribution and the parts outside of $[\tilde{r}_i]$ are small, $j$ must have large contribution to $\psi_f$ induced on $[\tilde{r}_i]$. So, by~\autoref{lm:ext}, $j$ is a clone under $\psi_f$ of some vertex $j^*$ in $[\tilde{r}_i]$. In particular, $|\psi_f(jj^*)|\leq 1$ which contradicts $\psi_f \in \Phi_{\x,2}(r_f)$.

Since $\|\bb_i-\ba'_i\|_1 < \gG^2$, and since our choice of $\mu$ implies that $\tilde{\alpha}_i \geq \mu$ for all $i \in [\tilde{r}_i]$, it follows that $|U_i^j| = \beta_{i,j}|U_i| \geq \mu n_i/2$ for all $j \in [\tilde{r}_i]$.
Until here the proof is essentially the same as in~\cite{katherine_stability}.

\medskip
We now complete the claim by comparing $\phi_i$ and the partition $\bigcup_{j \in [\tilde{r}_i]}U^j_i$ of $U_i$ to $(\tilde{r}_i,\psi_f,\tilde{\ba}_i) \in \opt^*(\x)$.
Suppose for a contradiction that there is a blue edge $y_1y_2$ between parts $U_i^{j_1},U_i^{j_2}$ where $j_1,j_2 \in [\tilde{r}_i]$ are distinct. Then there is some $c \in \phi_i(y_1y_2) \setminus\psi_f(j_1j_2)$. 
Since adding $c$ to $\psi_f(j_1j_2)$ to form $\psi_f'$ increases $q$, there is a forbidden pattern in $\psi_f'$ (this is \autoref{fact:max}), say on the vertices $j_1, j_2,j_3,\ldots, j_k$. Then as $d_{\rm green}(x)\geq (1-\dD)n_i$ for all $x \in U_i$  and $|U_i^j| \geq \mu n_i/2 \geq k\dD n_i$  for all $j \in [\tilde{r}_i]$, there are $y_3, \dots, y_k$ with $y_\ell \in U_i^{j_\ell}$ such that all the edges $y_\ell y_{\ell'}$ for $\ell, \ell' \in [k]$ are green except for $y_1y_2$. Thus, in $\phi_i$, we can find the same forbidden pattern on $y_1, \dots, y_k$ as found in $\psi_f'$ above, which contradicts $\phi_i\in \Phi_{\x,0}(r)$.

Suppose for a contradiction that there is a blue edge $y_1y_2$ inside a part $U_i^j$ for $j \in [\tilde{r}_i]$. Then $|\phi_i(y_1y_2)| \geq 2$ and the argument is very similar: To obtain $(\tilde{r}_i+1,\psi_f',\tilde{\ba}_i')$ from $(\tilde{r}_i,\psi_f,\tilde{\ba}_i)$, add a clone $j'$ of $j$ in $\psi_f$ and the colours in $\phi_i(y_1y_2)$ between $j$ and $j'$ and evenly split the previous weight of $j$ between $j$ and $j'$, i.e.\ let $\tilde{\alpha}'_{i,j}=\tilde{\alpha}_{i,j'}'=\tilde{\alpha}_{i,j}/2$. We then have $q(\tilde{r}_i+1,\psi_f',\tilde{\ba}_i')-q(\tilde{r}_i,\psi_f,\tilde{\ba}_i)\geq (\tilde{\alpha}_{i,j})^2 \log(2)/4>0$, hence $\psi_f'$ contains a forbidden pattern, say on vertices $j,j'$, $j_3, \dots, j_k$. Then as above we find $y_3, \dots, y_k$ with $y_\ell\in U_i^{j_\ell}$ with all the edges $y_\ell y_{\ell'}$ for $\ell, \ell' \in [k]$ green except for $y_1y_2$. In $\phi_i$ we can then find the same forbidden configuration on $y_1, \ldots, y_k$ as found in $\psi_f'$, which is again a contradiction.
This completes the proof of the claim.
\end{proof}

We showed in the claim that for each $i=0,\ldots,f$
there is a partition $\bigcup_{j \in [\tilde{r}_i]}U^j_i$ of $[r]$ where $\tilde{r}_i \leq r_f$ such that $\sum_{j \in [\tilde{r}_i]}||U^j_i|-\tilde{\alpha}_{i,j}n_i| \leq \gG^2 n_i$
for some $(\tilde{r}_i,\psi_f,\tilde{\ba}_i) \in \opt^*(\x)$.
We now claim that setting $(r^*,\phi^*,\ba^*):=(\tilde{r}_0,\psi_f,\tilde{\ba}_0)$, and defining the sets $Y_j:=U_0^j$ for all $j\in [r^*]$ and $Y_0:=[r]\sm \bigcup_{j\in [r^*]}Y_j$ yields a partition satisfying~\ref{SOone}--\ref{SOthree}. For part~\ref{SOone}, setting $\beta_j:=\sum_{j'\in Y_j}\alpha_{j'}$, we have that $r\beta_j=|U_0^j|=\beta_{0,j}n_0$ (recall that without loss of generality, we assumed $\alpha_i=1/r$ for all $i$) so the inequality $\sum_{j \in [\tilde{r}_i]}||U^j_i|-\tilde{\alpha}_{i,j}n_i| \leq \gG^2 n_i$ for $i=0$ implies $\sum_{j\in[r^*]}|r\beta_j/n_0-\alpha^*_j|\le \gamma^2$. Substituting $(1-2\sqrt{\eps})r \le n_0\le r$ and $\gamma, \eps \ll \nu$ yields~\ref{SOone}. For~\ref{SOtwo}, by \autoref{cl:blue_edges} for $i=0$, the auxiliary graph $G_0$ has no blue edges, that is, if $i'\in U_0^i=Y_i$ and $j'\in U_0^j=Y_j$ for $ij\in \binom{[r_f]}{2}\supseteq \binom{[r^*]}{2}$, then $\phi(i'j')\sm \phi^*(ij)=\emptyset$ (recall that $\phi_0=\phi$ and $\phi^*=\psi_f)$. 

Property~\ref{SOthree} requires a bit of extra work. We know from \autoref{cl:blue_edges} that pairs within parts see at most one colour in $\phi_i$ but not that they all see the same colour. The following paragraph is true for all intermediate partitions $\mc{U}_i$, in particular for $i=0$, which will yield the final property~\ref{SOthree}.

We start by showing that in each part $U_i^j$ with $j \in [\tilde{r}_i]$ we see at most one colour. Suppose that there are $j \in [\tilde{r}_i]$, vertices $y_1, y_2, z_1, z_2 \in U_i^j$ and distinct colours $c,c'$ such that $c \in \phi_i(y_1 y_2), c' \in \phi_i(z_1 z_2)$. Splitting part $j$ into parts $j_1,j_2$ in  $\tilde{\ba}_i$, each of size  $\tilde{\aA}_{i,j}/2$, and adding colours $c,c'$ to the edge between $j_1, j_2$ increases $q(\psi_f,\tilde{\ba}_i)$ above $Q$. Therefore this colour template must contain a forbidden pattern involving $j_1$ and $j_2$. Assume that it uses colour $c$ between $j_1,j_2$ and let $j_3, \dots, j_k$ be the other vertices of the pattern. Since $d_{\rm green}(x)\geq (1-\dD)n$ for all $x \in U_i$, we then find $y_3, \dots, y_k$ with $y_\ell \in U_i^\ell$ such that the same forbidden pattern appears on $y_1, \dots, y_k$ in $\phi_i$, a contradiction. Thus for each $j\in [\tilde{r}_i]$ there is a unique colour $c_j$ such that all edges inside $U_i^j$ see no other colour than $c_j$.

Finally, for the second part of~\ref{SOthree}, suppose that $\phi^*$ is stable inside, in which case we must show that there are no colours inside any part. Indeed, suppose there are $x,y \in U_i^j$ with $c \in \phi_i(xy)$ for some colour $c$. Since the pattern is stable inside, splitting $j$ into $j_1$ and $j_2$ as above and letting $\phi_i(j_1j_2)=\{c\}$ creates a forbidden pattern on parts $j_1, j_2, \dots, j_k$. Analogously to previous arguments we can show that this results in a forbidden pattern in $\phi_i$, a contradiction.
\end{proof}

\subsection{Proof of \autoref{th:stability_graphs}}\label{sec:proofs:stability}

The idea of this proof is to apply the multicolour regularity lemma, and the coloured regularity partition thus obtained can be approximated by an almost optimal solution. We can then apply~\autoref{lm:stability_opt} to this solution.

We use standard terminology and tools related to Szemer\'edi regularity.
Given a graph $G$, disjoint $A,B \subseteq V(G)$ 
and $0 \leq \delta \leq d \leq 1$, we say that $G[A,B]$ is
\emph{$\delta$-regular} if
$|d_G(X,Y)-d_G(A,B)|\leq\delta$ for all $X\subseteq A$, $Y \subseteq B$ with $|X|/|A|, |Y|/|B| \geq \delta$.
Recall the definition from the beginning of~\autoref{sec:stabgraphs} that $G[A,B]$ is \emph{$(\delta,d)$-regular} if it is $\delta$-regular
and $d_G(A,B) = d \pm \delta$.
We further say that $G[A,B]$ is
\emph{$(\delta, \geq\! d)$-regular} if it is $\delta$-regular and $d_G(A,B) \geq d-\delta$.
Subscripts are omitted where they are clear from the context.

We use the multicolour regularity lemma, and
the Embedding lemma, which says that if $0<\eps\ll d,1/k$ and $G$ is a sufficiently large $k$-partite graph with vertex classes $V_1,\ldots,V_k$ such that each $G[V_i,V_j]$ is $(\eps, \geq d)$-regular, then $K_k\subseteq G$.
See \cite[Theorems 1.18 and 2.1]{komlos2002regularity} for full statements and relevant definitions.

\begin{proof}[Proof of \autoref{th:stability_graphs}]
Let $\x$ be a bounded family of $s$-edge colourings of $K_k$ which has the extension property and let $\dD>0$ be given.
Let $\mu$ be the constant guaranteed by~\autoref{lm:mu}.
We may assume without loss of generality that $0 < \dD \ll \mu \ll 1$.
Let $0<1/n_0 \ll \gamma_1\ll\gamma_2\ll\eps\ll\nu \ll \gamma \ll \delta \ll \mu \ll 1/s,1/k$, and suppose that $G$ is a graph on $n \geq n_0$ vertices satisfying the theorem hypothesis. As in \cite{katherine_stability}, the multicolour regularity lemma yields a map $\mathrm{RL}$ from the set of all $\x$-free $s$-colourings $\chi: E(G)\to [s]$ to the set of pairs $(\mathcal{U},\phi)$ consisting of an equipartition $\mathcal{U} = U_1\cup\ldots\cup U_r$ of $V(G)$ and $\phi:\binom{[r]}{2}\to 2^{[s]}$ given by $\phi(ij)=\{c:\chi^{-1}(c)[U_i,U_j] \mbox{ is a }(\gamma_1,\ge \gamma_2)\mbox{-regular pair}\}$ such that
\begin{enumerate}
    \item[(R1)] at most $s\gamma_1\binom{r}{2}$ pairs of sets in $\mathcal{U}$ are $\gamma_1$-irregular with respect to at least one colour,
    \item[(R2)] between pairs $U_i,U_j$ which are $\gamma_1$-regular with respect to all colours, at most $s\gamma_2|U_i||U_j|$ edges have colours not in $\phi(ij)$
\end{enumerate}
and $r$ does not depend on $n$. Thus
we may suppose that $1/n \ll 1/r\ll \gamma_1$. Again as in \cite[Definition 4.9, Proposition 4.10]{katherine_stability} we call $(\mathcal{U}, \phi)$ \emph{popular} if its pre-image under $\mathrm{RL}$ satisfies $|\mathrm{RL}^{-1}(\mathcal{U},\phi)| \geq e^{-3\eps n^2}F(G;\x)$, and define
$\mathrm{Col}(G)$ to be the set of colourings $\chi$ mapping to a popular pair. By an adaptation of \cite[Proposition 4.10]{katherine_stability} we have that \begin{align}\label{eq:Col}
    |\mathrm{Col}(G)|\ge (1-e^{-2\eps n^2})F(G;\x).
\end{align}
In particular, the set of popular pairs is non-empty.

We now repeat \cite[Claim 4.1]{katherine_stability} to show that if $\mathrm{RL}(\chi)=(\mathcal{U},\phi)$ is popular, then setting $\ba :=(|U_1|/n,\ldots,|U_r|/n)$ yields a near-optimal feasible solution $(r,\phi,\ba)\in \feas_\x(s)$.

\begin{secclaim}\label{cl:stability_graphs:1} Suppose that $\mathrm{RL}(\chi)=(\mathcal{U},\phi)$ is popular.
Then $(r,\phi,\ba)\in \feas_\x(s)$. Moreover,
$q(\phi,\ba)\ge Q(\x)-8\eps$ 
and there are at most $s\gamma_2 n^2$ edges $xy\in E(G)$, $x\in U_i,y\in U_j$ such that $i=j$ or $\chi(xy)\not\in \phi(ij)$.
\end{secclaim}
\begin{proof}[Proof of claim.]
Feasibility is immediate: $\mathcal{U}$ is a partition of $V(G)$, so $\sum_i \alpha_i=1$. Also, if $\phi$ contained an element of $\x$ then by the Embedding lemma, $G$ coloured according to $\chi$ would contain the same element of $\x$, a contradiction.

The number of `atypical' edges, i.e.~edges in pairs specified by (R1) and (R2), is at most $s\gamma_1\binom{r}{2}\times \left(\frac{n}{r}\right)^2+\binom{r}{2}\times s\gamma_2\left(\frac{n}{r}\right)^2 \le 2s\gamma_2 n^2/3$. Together with edges internal to the $U_i$, this adds up to at most $2s\gamma_2 n^2/3+r\times \binom{n/r}{2}\le s\gamma_2 n^2$. All other edges $xy\in E(G)$ are between sets with $\chi(xy)\in \phi(ij)$, proving the last part of the claim.

Every colouring $\chi$ in $\mathrm{RL}^{-1}(\mathcal{U},\phi)$ satisfies (R1) and (R2) above, so 
\begin{align*}|\mathrm{RL}^{-1}(\mathcal{U},\phi)|&\le \underbrace{s^{r\times\left(\frac{n}{r}\right)^2}}_{\mbox{\scriptsize edges in each }U_i}\times \underbrace{\binom{\binom{r}{2}}{\le s\gamma_1\binom{r}{2}}s^{s\gamma_1\binom{r}{2}\times \left(\frac{n}{r}\right)^2}}_{\mbox{\scriptsize edges in irregular pairs}}\times \underbrace{\binom{\left(\frac{n}{r}\right)^2}{\le s\gamma_2 \left(\frac{n}{r}\right)^2}s^{s\gamma_2 \left(\binom{r}{2} \times \frac{n}{r}\right)^2}}_{\mbox{\scriptsize low-density regular pairs}}\\
&\times \prod_{ij:~\phi_{ij}\ge 1}\phi_{ij}^{e(G[U_i,U_j])}.
\end{align*}
Now, using standard properties of binomial coefficients, the first three terms multiply to at most 
$e^{\eps n^2/2}$, whereas the last term is at most $e^{n^2\sum_{ij}\alpha_i\alpha_j\log \phi_{ij}}=e^{q(\phi,\ba)\frac{n^2}{2}}$. Putting the last observations together, we get $|\mathrm{RL}^{-1}(\mathcal{U},\phi)|\le e^{(q(\phi,\ba)+\eps)\binom{n}{2}}$.
Since $(\mathcal{U}, \phi)$ is popular, we therefore have
\begin{equation}\label{eq:popular}
F (G;\x)\le e^{3 \eps n^2}|\mathrm{RL}^{-1}(\mathcal{U},\phi)| \leq e^{3\eps n^2}\times e^{(q(\phi,\ba)+\eps)\binom{n}{2}}\le e^{Q(\x)\binom{n}{2}+8\eps\binom{n}{2}},
\end{equation}
giving a matching upper bound to (\ref{eq:lb}) and hence a proof of \autoref{th:asymptotic}
(note that we have not used the extension property, which is an assumption of~\autoref{th:stability_graphs} but not~\autoref{th:asymptotic}).
Together with our assumption on $F(G;\x)$
we have $q(\phi,\ba)\ge Q(\x)-8\eps$, as required. 
\end{proof}
Let us fix a popular $(\mathcal{U},\phi)$ and its corresponding $\ba$. By \autoref{lm:stability_opt} there exists a partition $[r]=Y_0\cup\ldots\cup Y_{r^*}$ and a triple $(r^*,\phi^*,\ba^*)\in \opt^*(\x)$ which have the following properties:
\begin{itemize}
    \item[(i)] $\alpha_i^* > \mu$ for all $i \in [r^*]$ (from~\autoref{lm:mu}).
    \item[(ii)] Setting $\beta_i:=\sum_{j\in Y_i}\alpha_j$ for all $i \in [0,r^*]$ and $\bb := (\beta_1,\ldots,\beta_{r^*})$,
    we have $\|\b \beta-\b \alpha^*\|_1\le \nu$.
    \item[(iii)] For all $ij\in\binom{[r^*]}{2}, i'\in Y_i,j'\in Y_j$ we have $\phi(i'j')\subseteq \phi^*(ij)$.
    \item[(iv)] For all $i \in [r^*]$ there is a colour $c_i \in [s]$ such that $\phi(i'j') \subseteq \{c_i\}$ for every $i'j' \in \binom{Y_i}{2}$. If $\x$ is stable inside, then for all $i \in [r^*]$ and every $i'j' \in \binom{Y_i}{2}$, we have $\phi(i'j')= \emptyset$.
\end{itemize}
We now merge each group $Y_i$ of clusters of $\mathcal{U}$ to form a partition $\mathcal{W}=\{W_0,\ldots,W_{r^*}\}$ of $V(G)$ with $W_i=\bigcup_{j\in Y_i}U_j$. We have $|W_i|=\beta_i n$. We now want to show that most colourings $\chi$ in $\mathrm{RL}^{-1}(\mathcal{U},\phi)$ are regular with respect to $(\mathcal{W},\phi^*)$.

Note that $|W_0|=\beta_0 n = \sum_{i \in [r^*]}(\alpha_i^*-\beta_i^*)n \leq \|\ba^*-\bb\|_1 n \leq \nu n$. Similarly, the other sets have sizes $|W_i|=\beta_in \ge (\alpha_i^*-\nu)n\ge \mu n/2$. We also have that the number of edges in cluster $W_i$ not coloured in $c_i$ is small. More precisely, from the proof of \autoref{cl:stability_graphs:1} at most $s\gamma_2n^2$ edges $xy\in E(G)$ are within the clusters of $\mathcal{U}$ or between them with $\chi(xy)\not\in \phi(ij)$. Together with property (iv) above, this gives
\begin{equation}\label{eq:stability_graphs:1}\sum_{i\in[r^*]}(e(W_i)-|\chi^{-1}(c_i)[W_i]|)\le s\gamma_2n^2,\end{equation}
or, if the pattern is stable inside, $\sum_{i\in[r^*]}e(W_i)\le s\gamma_2n^2$.

The following claim is based on a standard calculation that shows that the probability that a random bipartite graph is not $\gamma$-regular is exponentially small in $n^2$.

\begin{secclaim}
At least $(1-e^{-\Omega(\gamma^3 n^2)})|\mathrm{RL}^{-1}(\mathcal{U}, \phi)|$ colourings $\chi$ in $\mathrm{RL}^{-1}(\mathcal{U},\phi)$ satisfy:
\begin{itemize}
    \item[($\dagger$)] for all $ij\in \binom{[r^*]}{2}$ and $c\in \phi^*(ij)$ the bipartite graph $\chi^{-1}(c)[W_i,W_j]$ is $(\gamma,\ge\! {\phi^*_{ij}}^{-1})$-regular.
\end{itemize}
\end{secclaim}

\begin{proof}[Proof of claim.] If $\chi \in \mathrm{RL}^{-1}(\mathcal{U},\phi)$ does not satisfy ($\dagger$), there exist $i^*j^*\in \binom{[r^*]}{2}$ and $c^*\in \phi^*(i^*j^*)$ such that $\chi^{-1}(c^*)[W_{i^*},W_{j^*}]$ is not $(\gamma,\ge\! {\phi^*_{ij}}^{-1})$-regular. By \cite[Proposition 4.4]{katherine_stability} this implies the existence of $X\subseteq W_{i^*}$ and $Y\subseteq W_{j^*}$ of size $\gamma|W_{i^*}|$ and $\gamma|W_{j^*}|$ respectively such that $d(\chi^{-1}(c^*)[X,Y])<(1-\gamma/2){\phi^*_{i^*j^*}}^{-1}$.
Therefore, the number of such colourings is at most
\begin{align*}
   \underbrace{s^{\nu n^2}}_{\mbox{\scriptsize edges incident to }W_0}\times \underbrace{s^{s\gamma_2n^2}}_{\substack{\mbox{\scriptsize edges in }W_i\\\chi(e) \neq c_i}}\times \underbrace{\binom{r^*}{2}s\binom{|W_{i^*}|}{\gamma|W_{i^*}|}\binom{|W_{j^*}|}{\gamma|W_{j^*}|}}_{\mbox{\scriptsize choice for }i^*j^*,c^*,X,Y}\times
   ~C\times D \times \prod_{ij\neq i^*j^*}{\phi^*_{ij}}^{|W_i||W_j|}
\end{align*}
where
\begin{align*}C&=\sum_{\ell\le({\phi^*_{i^*j^*}}^{-1}-\gamma/2s)|X||Y|} \binom{|X||Y|}{\ell}(\phi^*_{i^*j^*}-1)^{|X||Y|-\ell}\\
&\le e^{-\gamma^2\phi^*_{i^*j^*}|X||Y|/12s^2}{\phi^*_{i^*j^*}}^{|X||Y|}\mbox{~~~~~(using \cite[Corollary 4.8]{katherine_stability})}\\
&\le e^{-\gamma^2\mu^2n^2/48s}{\phi^*_{i^*j^*}}^{|X||Y|}\end{align*}
is an upper bound on the number of colour choices for $G[X,Y]$ and $D={\phi^*_{i^*j^*}}^{|W_{i^*}||W_{j^*}|-|X||Y|}$
is an upper bound on the number of colour choices for the rest of $G[W_{i^*},W_{j^*}]$.

Similarly to the previous claim, the product of the terms except $e^{-\gamma^2\mu^2n^2/48s} \prod{\phi^*_{ij}}^{|W_i||W_j|}$ is at most $e^{O(\nu n^2)}$ and we have
\begin{eqnarray*}
\prod{\phi^*_{ij}}^{|W_i||W_j|} &=& e^{q(\phi^*,\b \beta)\binom{n}{2}}
\le e^{(q(\phi^*,\ba^*)+2\nu\log(s))\binom{n}{2}}
\mbox{~~~~~(using~\autoref{lm:lipschitz})}\\
&\stackrel{(\ref{eq:lowerbd})}{\leq}& F(n;\x) \cdot e^{2\nu\log(s)\binom{n}{2}+O(n)}\\
&\leq& |\mathrm{RL}^{-1}(\mathcal{U},\phi)|\times e^{7\eps \binom{n}{2}}\times e^{3\nu\log(s)\binom{n}{2}}
\mbox{~~~~~(since $(\mathcal{U},\phi)$ is popular)}.
\end{eqnarray*}
Thus the number of colourings not satisfying ($\dagger$) is at most
$$
e^{O(\nu n^2)-\gamma^2\mu^2n^2/48s}|\mathrm{RL}^{-1}(\mathcal{U}, \phi)|\leq e^{-\Omega(\gamma^3 n^2)}|\mathrm{RL}^{-1}(\mathcal{U}, \phi)|,
$$
as required.
\end{proof}

We are now almost done. It remains to adjust $\mathcal{W}$ so that $|W_i|=\alpha_i^*n$ (recall that we suppress rounding) and $W_0$ is empty.  For each $i\ge 1$, let $V_i$ be obtained by taking any $\min\{\alpha_i^*n,|W_i|\}$ vertices from $W_i$, and then distributing the remaining vertices of $G$ arbitrarily so that $|V_i|=\alpha_i^*n$. This gives property~\ref{SGone} from the theorem statement. Recall that $|W_i|=\beta_in$ with $\|\b\beta -\ba^*\|_1\le \nu$, so $\sum_i|V_i\triangle W_i|=\|\b\beta -\ba^*\|_1n\le \nu n$. Since the parts of $\ba^*$ have size at least $\mu$ and $\nu \ll \gamma \ll \mu$, each colouring $\chi \in {\mathrm RL}^{-1}(\mathcal{U}, \phi)$ satisfying ($\dagger$) with respect to $\mathcal{W}$ now satisfies it with respect to $\mathcal{V}$ with a larger regularity parameter $2\gamma$ (see \cite[Proposition 4.5]{katherine_stability}), that is, each $\chi^{-1}(c)[V_i,V_j]$ is $(2\gamma, \ge\! {\phi^*_{ij}}^{-1})$-regular. This gives property~\ref{SGtwo} from the theorem statement. Finally, recall that by \eqref{eq:stability_graphs:1} at most $s\gamma_2n^2$ edges of $G$ are within the clusters of $\mathcal{W}$ and not in colour $c_i$, so at most $s\gamma_2n^2+\nu n^2\le \delta n^2$ are within the clusters of $\mathcal{V}$ and not in colour $c_i$, giving property~\ref{SGthree}. For patterns that are stable inside, there are simply at most $s\gamma_2n^2$ edges within clusters of $\mathcal{W}$, and hence at most $\delta n^2$ within the clusters of $\mathcal{V}$, giving the second part of property~\ref{SGthree}.

For each popular $(\mathcal{U},\phi)$, the above holds for at least $(1-e^{-\Omega(\gamma^3n^2)})|\mathrm{RL}^{-1}(\mathcal{U},\phi)|$ colourings in $\mathrm{RL}^{-1}(\mathcal{U},\phi)$. Thus, using~\eqref{eq:Col}, we get a total of 
\begin{align*}\sum_{\mathcal{U},\phi}(1-e^{-\Omega(\gamma^3n^2)})|\mathrm{RL}^{-1}(\mathcal{U},\phi)|&=(1-e^{-\Omega(\gamma^3n^2)})|\mathrm{Col}(G)|\\
&\ge (1-e^{-\Omega(\gamma^3n^2)})(1-e^{-2\eps n^2})F(G;\x)\\
&\ge (1-e^{-\eps n^2})F(G;\x)\end{align*}
colourings for which properties~\ref{SGone}--\ref{SGthree} hold.
\end{proof}

\subsection{Proof of \autoref{th:exact1}}\label{sec:proofs:exact}

\begin{proof}[Proof of \autoref{th:exact1}]
We are given constants $k,s,\dD$ and may choose additional constants $\eta, \xi, \eps,\dD_1,\dD_2,n_0$ so that, without loss of generality, we have $0<1/n_0 \ll \eta\ll \xi \ll \eps \ll \dD_1 \ll \dD_2 \ll \dD \ll \mu,1/k,1/s,1/R$, where for every $(r^*,\phi^*,\ba^*)\in\opt^*(\x)$, we have $r^* \leq R$ (by boundedness),
and $\aA^*_i \geq \mu$ for all $i \in [r^*]$ (by \autoref{lm:mu}), and also
$Q(\x) - o(1) \leq (\log F(n;\x))/\binom{n}{2} < Q(\x)+\eps$ for all $n \geq n_0$ (by \autoref{th:asymptotic}).

\begin{secclaim}[See Lemma 2.14,~\cite{katherine_exact}]\label{cl:partite}
If there is an $\x$-extremal graph which is not complete partite,
then there is an $\x$-extremal graph which is complete partite with a part of size one.
\end{secclaim}

\begin{proof}[Sketch proof of claim.]
The proof of \cite[Lemma 2.14]{katherine_exact} goes through verbatim, by property~\ref{it:S} at the beginning of this section.
The proof proceeds by \emph{symmetrising}, that is, successively replacing a vertex $u$ of an $\x$-extremal graph $G$ by a twin of another vertex $v$ which is not adjacent to $u$. The graph $G_{uv}$ obtained in this way is such that \emph{both} $G_{uv}$ and $G_{vu}$ are extremal.
Indeed, if $C(J)$ is the set of valid colourings of a graph $J$, 
we have
$F(G;\x) = \sum_{\chi \in C(G-u-v)}\chi_u \chi_v$
where e.g.~$\chi_u$ is the number of extensions of $\chi$ to $G-v$,
since $u$ and $v$ are not adjacent and every forbidden colouring in $\x$ is on a clique.
Thus $F(G_{uv};\x)=\sum_{\chi \in C(G-u-v)}\chi_v^2$ and so
$$
0 \leq \sum_{\chi \in C(G-u-v)}(\chi_u-\chi_v)^2 = F(G_{uv};\x)-2F(G;\x)+F(G_{vu};\x) \leq 0.
$$
The choice of which pair $u,v$ to symmetrise at each step can be made so that the final graph has a part of size one.
\end{proof}

Let $G$ be an $\x$-extremal graph on $n \geq n_0$ vertices.
By \autoref{cl:partite} either $G':=G$ is complete partite or there is a complete partite extremal graph $G'$ with a part of size one with the same number of vertices. Let $\mathcal{W}=W_1\cup\ldots\cup W_r$ be the partition of $V(G')$. 
Apply \autoref{th:stability_graphs} with regularity parameter $\delta_1$ to $G'$. Let $\chi:E(G')\to [s]$ be one of the $(1-e^{-\eps n^2})F(G';\x)$ `typical' $\x$-free colourings of $G'$ described in \autoref{th:stability_graphs}. Then there exist $(r^*,\phi^*,\ba^*)\in \opt^*(\x)$ and a partition $\mathcal{V} = V_1\cup \ldots\cup V_{r^*}$ of $V(G')$ satisfying properties~\ref{SGone}--\ref{SGthree}. 

Since $\x$ is hermetic and hence stable inside,~\ref{SGthree} implies that one can change at most $\dD_1{r^*}^2n^2$ adjacencies to obtain the complete $r^*$-partite graph $K[V_1,\ldots,V_{r^*}]$.
Thus our partition $W_1,\ldots,W_r$ must be close to $V_1,\ldots,V_{r^*}$:
after relabelling we have $|W_i \bigtriangleup V_i| \leq 2r^*\sqrt{s\dD_1}n$ and so $|W_i|=(\aA^*_i\pm \dD_2)n > \mu n/2$ for all $i \in [r^*]$ (in particular, $r\ge r^*$) and the union of the $r-r^*$ other parts $W_j$ has size at most $r^*\sqrt{s\dD_1} n \leq \dD_2 n$.

We now proceed to show that $G'$ is complete $r^*$-partite, which will also imply that $G=G'$ and together with the previous observations give~\ref{seone}.
Suppose for a contradiction that $r>r^*$, i.e. that $G'$ has at least one small part $W_{r^*+1}$.
Let $x\in W_{r^*+1}$ be an arbitrary vertex in the small part. We define an extension $\phi$ of $\phi^*$ to $r^*+1$ parts which records the majority colour in $\chi$ between $x$ and each of $W_1,\ldots,W_{r^*}$. More precisely, let $\phi:\binom{[r^*+1]}{2}\to 2^{[s]}$ be such that $\phi(ij):=\phi^*(ij)$ for all $ij\in \binom{[r^*]}{2}$ and $\phi(i,r^*+1):=\{c\in [s]:~d_{\chi^{-1}(c)}(x, W_i)\ge |W_i|/s\}$. 

\begin{secclaim}\label{cl:valid}
    $\phi\in \Phi_{\x,1}(r^*+1)$.
\end{secclaim}
\begin{proof}[Proof of claim.]
    We have that $\phi^*\in \Phi_{\x,2}(r^*)$ so every pair $ij$ in $\binom{[r^*]}{2}$ satisfies $\phi_{ij}=\phi^*_{ij}\ge 2$. There are $|W_i|$ edges between $x$ and $W_i$ coloured in $s$ colours, so by the pigeonhole principle at least one colour will appear $|W_i|/s$ times, thus $|\phi(i,r^*+1)|\ge 1$.

    We now need to show that $\phi$ is $\x$-free. Suppose not. Then since $\phi^*$ is $\x$-free, there is a copy of some $X\in \x$ containing the vertex $r^*+1$, say with vertex set $\{r^*+1,i_1,\ldots,i_{k-1}\}$ and colouring $\sigma$. 
    Suppose that $a\in V(X)$ is the vertex mapped to $r^*+1$. For each $j\in[k-1]$, consider the neighbourhood $N_j$ of $x$ in colour $\sigma(r^*+1,i_j)$ in $V_{i_j}$. Each $N_j$ has size $|N_j|\ge|W_{i_j}|/s-|W_{i_j}\triangle V_{i_j}|\ge |V_{i_j}|/2s$. 
    Here we are using $(|V_{i_j}|\pm1)/n=\alpha_i^*\ge \mu\gg \delta_1,1/s,1/r^*$ and the fact that $|W_{i_j}\triangle V_{i_j}|\le \delta n$. For each pair of indices $i_h,i_j$ with colour $c:=\sigma(i_h i_j)$ in $X$ we have that $c\in\phi^*(i_hi_j)$, so the bipartite graph $\chi^{-1}(c)[V_{i_h},V_{i_j}]$ is $(\delta_1,{\phi^*_{i_hi_j}}^{-1})$-regular by property~\ref{SGtwo}. 
    Then $\chi^{-1}(c)[N_h,N_j]$ is $(2s\delta_1,{\phi^*_{i_hi_j}}^{-1})$-regular. 
    The Embedding lemma then gives a copy of $X-a$ in the neighbourhood of $x$, which together with $x$ gives a copy of $X$ in $G'$. This contradicts the fact that $\chi$ is $\x$-free, thus proving the claim.
\end{proof}

However, the family $\x$ is hermetic, so no such $\phi$ exists. Thus $r=r^*$ and since $|W_i|\geq \mu n/2$ for all $i \in [r^*]$, $G'$ has no part of size 1 and thus $G=G'$. Thus we have shown that for each typical $\chi$ there is $(r^*, \phi^*, \ba^*)$ 
such that 
\begin{enumerate}        
    \item \ref{seone} holds with a smaller error term $\dD_2$, i.e. 
    $G$ is  a complete $r^*$-partite graph whose $i$-th part $W_i$ has size $(\aA^*_i \pm \dD_2)n$ for all $i \in [r^*]$, and 
    \item \ref{se0} holds, i.e.\ $(r^*, \phi^*,\ba^*)\in \opt^*(\x)$, and~\ref{se1} holds with a smaller error $\dD_1$, i.e.~$\chi^{-1}(c)[W_i,W_j]$ is $(\dD_1,{\phi^*_{ij}}^{-1})$-regular for all $ij \in \binom{[r^*]}{2}$ and $c \in \phi^*(ij)$.
\end{enumerate}
Note that thus far, $(r^*, \phi^*, \ba^*)$ may depend on $\chi$.
Observe that  $G$ being $r^*$-partite, where $r^*$ came from an arbitrary $\chi$, means that $r^*$ in fact is independent of the choice of $\chi$. 
We now argue that we may also choose $\ba^*$ independently of $\chi$ (which is at least a priori possible as (1) means each of the $\ba^*$ are all very close). For this, note that for a given $\phi^*$, we may use the same choice of $\ba^*_{\phi^*}$ for any $\chi$ that uses a triple involving ${\phi^*}$, maintaining the above properties. This gives rise to a family $\Phi \subseteq \Phi_{\x,2}(r^*)$ such that for each typical $\chi$, there is $(r^*, \phi^*, \ba^*_{\phi^*})$ with $\phi^* \in \Phi$ for which~(1) and~(2) hold for $\chi$.~\autoref{lem:uni_ba} applies to show that there is $\ba^*$ such that for all $\phi^*\in \Phi$, $(r^*,\phi^*,\ba^*) \in \opt(\x)$ and $\|\ba_{\phi^*}-\ba^*\|_1<\dD/2$. 
In fact, we get $(r^*,\phi^*,\ba^*) \in \opt^*(\x)$ since for each $i \in [r^*]$ we have $\alpha^*_{\phi^*,i}>\mu$ and $|\alpha^*_i-\alpha^*_{\phi^*,i}| < \delta/2$,
so $\alpha^*_i > \mu/2$.
Since for each typical $\chi$ and $i \in [r^*]$ we have
$||W_i|-\alpha_i^*|\leq ||W_i|-\alpha^*_{\phi^*,i}|+|\alpha^*_{\phi^*,i}-\alpha_i^*|\leq \dD_2+\dD/2\leq \dD$, we have found choices of $r^*$ and $\ba^*$ independent of $\chi$ such that the properties~\ref{seone},~\ref{se0} and~\ref{se1} hold. 
It remains to show that for most of the typical colourings $\chi$, our choice of $(r^*,\phi^*,\ba^*)$ satisfies~\ref{se2}, which is somewhat harder.

We say that a vertex $x\in G$ has \emph{large contribution} if 
$$\log F(G;\x)\ge \log F(G-x;\x)+(Q(\x)-\xi)n.$$
\begin{secclaim}[See Lemma~2.15,~\cite{katherine_exact}]
Every vertex in $G$ has large contribution.    
\end{secclaim}

\begin{proof}[Sketch proof of claim] 
Again, this proof goes through verbatim.
If $x$ did not have large contribution, then since it has many twins all necessarily with the same contribution, we would have many vertices with small contribution. A calculation using the Cauchy-Schwarz inequality shows that $\log F(G;\x) \leq \log F(G-T;\x) + (Q(\x)-\xi)|T|n$, where $T$ is the set of twins of $x$ and satisfies $\mu n/2\le |T|\le (1-\mu/2)n$ since every $\alpha_i^* \geq \mu$. Apply \autoref{th:asymptotic} to $G-T$ with parameter $\eta$ (noting that $|G-T|\ge \mu n_0/2\gg 1/\eta$) to obtain $\log F(G-T;\x) \leq (Q(\x)+\eta)\binom{n-|T|}{2}$.
Together with $\log F(G;\x) \geq Q(\x)\binom{n}{2}-O(n)$, we obtain a contradiction since $\eta\ll \xi$.
\end{proof}

Applying the claim twice yields that for every $x,y \in V(G)$ we have
\begin{equation}\label{eq:xyattach}
\log F(G;\x) \geq \log F(G-x-y;\x)+(Q(\x)-\xi)(n+n-1).
\end{equation}

We say that an $\x$-free colouring $\chi$ of $G$ is
\emph{locally good} if there is $\phi^*$ such that $(r^*,\phi^*,\ba^*) \in \opt^*(\x)$ and
\begin{enumerate}[label=(LG\arabic*)]
        \item\label{LG1} for all $ij \in \binom{[r^*]}{2}$ and $c \in \phi^*(ij)$, we have that $\chi^{-1}(c)[W_i,W_j]$ is $(\dD_1,{\phi^*_{ij}}^{-1})$-regular;
        \item\label{LG2} for all $hij \in \binom{[r^*]}{3}$ and vertices $x \in W_i$ and $y \in W_j$, and colours $c \in \phi^*(ih)$ and $c' \in \phi^*(jh)$, we have $|N_{\chi^{-1}(c)}(x,W_h) \cap N_{\chi^{-1}(c')}(y,W_h)| \geq \dD_2 |W_h|$.
    \end{enumerate}

\begin{secclaim}[See Lemma 2.11, \cite{katherine_exact}]\label{cl:perf_align}
    Every locally good valid colouring $\chi$ of $G$ is \emph{perfect},
    meaning that $\chi(xy) \in \phi^*(ij)$ for all $x \in W_i$ and $y \in W_j$ and $ij \in \binom{[r^*]}{2}$.
\end{secclaim}

\begin{proof}[Proof of claim] Let $\chi$ be such a colouring and suppose, for contradiction, that there are $x\in W_i$ and $y\in W_j$ with $\chi(xy)=c\not \in \phi^*(ij)$.
By \autoref{fact:max}, adding $c$ to $\phi^*(ij)$ creates $K\cong K_k$ coloured by $\sigma:E(K)\to[s]$ which is an element of $\x$. Note that $i,j\in V(K)$,  $\sigma(ij)=c$ and $\sigma(hf)\in \phi^*(hf)$ for all other pairs $hf\in \binom{[r^*]}{2}\setminus \{ij\}$. We now show that $xy$ lies in a copy of $K$ in $G$ coloured by $\chi$.

For every $h\in V(K)\setminus\{ i,j\}$ let $W_h':=N_{\chi^{-1}(\sigma(ih))}(x,W_h) \cap N_{\chi^{-1}(\sigma(jh))}(y,W_h)$. 
By~\ref{LG2}, we have $|W_h'| \geq \dD_2 |W_h|$. By~\ref{LG1}, for any $h,h'\in V(K)\setminus \{i,j\}$ we have that $\chi^{-1}(\sigma(hh'))[W_i,W_j]$ is $(\delta_1/\delta_2, {\phi^*_{hh'}}^{-1})$-regular. Thus by the Embedding lemma we can find a copy of $K-i-j$ coloured according to $\sigma$ with each vertex $h\in V(K)\setminus \{i,j\}$ mapped to $W_h'$. By definition of the sets $W_h'$, any such copy, together with $x$ and $y$, forms a copy of $K$ coloured according to $\sigma$.
\end{proof}

Thus, to complete the proof of~\ref{se1}, it suffices to show that most colourings are locally good.

\begin{secclaim}[See Theorem 2.12,~\cite{katherine_exact}, Claims~2.12.2--2.12.4]
    At least $(1-e^{-\eps n})F(G;\x)$ valid colourings $\chi$ of $G$ are locally good. 
\end{secclaim}

\begin{proof}[Sketch proof of claim]
We note that the locally good condition in~\cite{katherine_exact} is stronger than the one here.
Recall that for $(1-e^{-\eps n^2})F(G;\x)$ valid colourings there is $\phi^*$ (which depends on the colouring) for which~\ref{LG1} holds, so it suffices to show that $(1-e^{-\dD_1n})F(G;\x)$ valid colourings satisfy~\ref{LG2} for that same $\phi^*$.
So suppose this does not hold. 
Then, by averaging, there are two vertices $x,y$ such that at least $\binom{n}{2}^{-1}e^{-\delta_1n}F(G;\x)$ valid colourings of $G$ do satisfy~\ref{LG1} but do not satisfy~\ref{LG2} at $x,y$. By~(\ref{eq:xyattach}) the number of such colourings is at least
$$
    \binom{n}{2}^{-1}e^{-\delta_1n}\left(F(G-x-y;\x)\cdot e^{(Q(\x)-\xi)(n+n-1)}\right)\ge F(G-x-y;\x)\cdot e^{(Q(\x)-2\delta_1)(n+n-1)}.
    $$
    Again by averaging, there is a valid colouring $\chi$ of $G-x-y$ and associated colour template $\phi^*:\binom{[r^*]}{2}\to 2^{[s]}$ from~\ref{LG1}, with at least $e^{(Q(\x)-3\dD_1)(n+n-1)}$ valid extensions $\overline{\chi}$ to $G$ which do not satisfy the~\ref{LG2} condition at $x,y$.
    Averaging once more, there are $c_x,c_y\in [s]$ and $h^* \in [r^*]$ such that at  least $e^{(Q(\x)-4\dD_1)(n+n-1)}$ valid extensions of $\chi$ violate~\ref{LG2} at $x,y$ with part $W_{h^*}$ and colours $c_x,c_y$. 
    For each such extension $\overline{\chi}$, we define colour templates $\phi^*_{\overline{\chi},x},\phi^*_{\overline{\chi},y}$ on $r^*+1$ parts which are extensions of $\phi^*$ by taking popular colours at $x,y$, as follows: for every $h\in[r^*]$, let $\phi^*_{\overline{\chi},x}(hf):=\phi^*(hf)$ for all $hf \in \binom{[r^*]}{2}$, and let $\phi^*_{\overline{\chi},x}(\{h,r^*+1\}):=\{c\in[s]:~|N_{\overline{\chi}^{-1}(c)}(x)\cap W_h|\ge \delta_1|W_h|\}$, and similarly for $\phi_{\overline{\chi},y}$.
    These colour templates are valid, via an almost identical argument to \autoref{cl:valid}. 
    Choose $\phi^*_{x},\phi^*_{y}$ that appear most frequently among the $\phi^*_{\overline\chi,x}, \phi^*_{\overline\chi,y}$. The number of extensions $\overline{\chi}$ with these $\phi^*_{x},\phi^*_{ y}$ is at least $e^{(Q(\x)-5\dD_1)(n+n-1)}$.

    We next argue that this lower bound on the number of valid extensions implies that $\ext(\phi^*_{x},\ba^*)\ge Q(\x)-\sqrt{\delta}$, say (and similarly for $\phi^*_{y}$). For this implication, note that since $\phi^*_{x}$ records all colours appearing in at least a $\delta_1$ fraction and we have part sizes $|W_i|=(\alpha_i^*\pm \delta)n$, the logarithm of the number of valid extensions of $\chi$ to $G$ (or, more precisely, to $G-xy$) corresponds closely to $\ext(\phi^*_{x},\ba^*)+\ext(\phi^*_{y},\ba^*)$. This gives the desired conclusion since $\ext(\phi^*_{y},\ba^*)\le Q(\x)$ (see Proposition~2.6 in~\cite{katherine_stability}). 

    Recall that $\x$ has the strong extension property since it has the extension property and is hermetic.
    Thus we can apply~\autoref{lm:ext} to see that in $r^*+1$ is a strong clone of $i'$ under $\phi^*_x$ for some $i' \in [r^*]$, and $r^*+1$ is a strong clone of $j'$ under $\phi^*_y$ for some $j' \in [r^*]$. 
    If $i,j$ are such that $x \in W_i$ and $y \in W_j$,
    then $i=i'$ and $j=j'$ since, for example, $i \neq i'$ would imply $\emptyset=\phi^*_{x}(\{i,r^*+1\})=\phi^*_{x}(ii')=\phi^*(ii')$ which is a contradiction to $\phi^* \in \opt^*(\x)$.
    
    Finally we bound from above the number of valid extensions of $\chi$ to $\overline{\chi}$ that follow $\phi^*_{x},\phi^*_{y}$ to get a contradiction: we choose at most $s\dD_1 n$ vertices $z$ in $V(G)\setminus\{x,y\}$ where at least one of $xz,yz$ is not coloured according to the colour template $\phi^*$, and at most $\dD_2$ vertices $z$ in $W_{h^*}$ for which $(\chi(xz),\chi(yz))=(c,c')$. 
    The remaining at least $(1-\delta_2)|W_{h^*}| - s \dD_1 n > \mu n/2$ vertices $z$ in $W_{h^*}$ are coloured according to the colour template except that $(\chi(xz),\chi(yz))=(c,c')$ is forbidden. 
    Since, for most vertices, there are at most $|\phi^*(ih^*)||\phi^*(jh^*)|-1$ choices of $(\chi(xz),\chi(yz))$ as $(c,c')$ is forbidden, rather than the $|\phi^*(ih^*)||\phi^*(jh^*)|$ we expect, we obtain fewer than $e^{(Q(\x)+O(\dD))(n+n-1)} \cdot \left(\frac{s-1}{s}\right)^{\mu n/2} < e^{(Q(\x)-5\dD_1)(n+n-1)}$ extensions, a contradiction.
\end{proof}
This completes the proof of the theorem.
\end{proof}

\section{Concluding remarks}\label{sec:conclude}

In this paper we studied the \emph{generalised Erd\H{o}s-Rothschild problem} on maximising the number of $s$-edge colourings of graphs on $n$ vertices where each colouring we count must not contain any forbidden colouring in some given family $\x$ of $s$-edge coloured cliques $K_k$.
Our main results were for two specific forbidden families: the dichromatic triangle family (\autoref{th:K32})
and the family of improperly coloured cliques (\autoref{th:proper}).
Along the way, we adapt results of~\cite{katherine_stability,katherine_exact} on monochromatic colour patterns to obtain some general results (\autoref{sec:general}) which are likely to be useful for further cases.

\subsection{A more general framework}
A key property of families $\x$ in many of our arguments was boundedness: that there is at least one basic optimal solution and every basic optimal solution has a bounded number of parts.
For example, Ramsey's theorem guarantees that the monochromatic pattern $K_k^{(1)}$ is bounded (recall \autoref{sec:optk1}).

As we mentioned in \autoref{sec:opt}, our framework does not capture all of the colour patterns we would like to study. Indeed, in \autoref{lm:rainbowclique}, we showed that the family $(K^{(3)}_3,3)$ for the rainbow triangle with three colours is not bounded.

We can introduce a generalisation Problem $Q^\bullet_t(\x)$ of Problem $Q_{t}(\x)$ that extends the set of feasible solutions by allowing vertex weights (the bullet represents the loops at vertices present in the new problem).

Suppose we have $r \in \mb{N}$ and a function $\phi : \binom{[r]}{\leq 2} \to 2^{[s]}$, which maps singletons and pairs of vertices to sets of colours. Let $\bm{a}=(a_1,\ldots,a_r) \in \mb{N}^r$ be an integer vector. We define the \emph{$\bm{a}$-blow-up of $\phi$} to be the edge-multicoloured graph $G_\phi(\bm{a})$ with parts $V_1,\ldots,V_r$ such that $|V_i|=a_i$, and each $G[V_i]$ is empty if $\phi(i)=\emptyset$, and otherwise it is a clique and every edge receives all of the colours $\phi(i)$; and for every distinct $i,j \in [r]$, $G[V_i,V_j]$ has no edges if $\phi(ij)=\emptyset$, and otherwise it is a complete bipartite graph and every edge receives all of the colours $\phi(ij)$. Given a family $\x$ of forbidden colourings of $K_k$, we say that $\phi$ is \emph{$\x$-free} if the $(k,\ldots,k)$-blow up $G_{\phi}((k,\ldots,k))$ of $\phi$ is $\x$-free.

We define $\Phi^\bullet_{\x,t}(r)$ to be the set of $\x$-free colourings $\phi: \binom{[r]}{\leq 2} \to 2^{[s]}$ such that $|\phi(ij)|\geq t$ for all $i\neq j$ and each $\phi(i)$ satisfies $|\phi(i)|=0$ or $|\phi(i)| \geq t$.

\medskip
\noindent\fbox{%
    \parbox{\textwidth}{%
\textbf{Problem $Q_t^\bullet(\x)$.}

Maximise
$$
q(\phi,\ba) := \sum_{i \in [r]}\aA_i^2\log|\phi(i)| + 2\sum_{ij \in \binom{[r]}{2}: \phi(ij) \neq \emptyset}\aA_i\aA_j\log|\phi(ij)|
$$
subject to
$r \in \mb{N}$
and $\phi \in \Phi^\bullet_{\x,t}(r)$
and $\ba \in \DD^r$.
}}
\medskip

Any feasible solution of Problem $Q_{t}(\x)$ is a feasible solution of Problem $Q_{t}^\bullet(\x)$, by mapping all singletons to the empty set, so $Q_t(\x)\le Q_{t}^\bullet(\x)$.
Again we can show that
\begin{equation}
F_{\x}(n,s) \geq e^{Q_{0}^\bullet(\x)\binom{n}{2}+O(n)}.\label{eq:lb_general}
\end{equation}
To remove degenerate solutions, we define the set of \emph{basic optimal solutions} $\opt^*_{\bullet}(\x)$ to be the set of $(r^*,\phi^*,\ba^*)$ such that 
\begin{itemize}
\item $\phi^* \in \Phi^\bullet_{2}(\x)$,
\item $\aA_i^*>0$ for all $i \in [r^*]$,
\item whenever there are distinct $i,j \in [r^*]$ with $\phi^*(i)=\phi^*(ij)=\phi^*(j)$, there is $h \in [r^*]\sm\{i,j\}$ such that $\phi^*(hi) \neq \phi^*(hj)$.
\end{itemize}

We say that $\x$ is \emph{$\bullet$-bounded} if $\opt^*_{\bullet}(\x)$ is non-empty and there is some $R>0$ such that $r^* \leq R$ for all $(r^*,\phi^*,\ba^*) \in \opt^*_{\bullet}(\x)$.

We will show that $(K^{(3)}_3,3)$ is $\bullet$-bounded.
Suppose that there is $(r^*,\phi^*,\ba^*) \in \opt^*_{\bullet}(K^{(3)}_3,3)$ in which all three colours appear on pairs. Then, without loss of generality, there are distinct $h,i,j \in [r^*]$ such that $\phi^*(hi)=\{1,2\}$ and $\phi^*(hj)=\{1,3\}$.
But then $\phi^*(ij)=\emptyset$, a contradiction.
Thus, again without loss of generality, $\phi^*(ij)=\{1,2\}$ for all pairs $ij$ in $[r^*]$.
The colour 3 cannot appear in any $\phi^*(i)$, so  by optimality we have $\phi^*(i)=\{1,2\}$ for all $i \in [r^*]$. But then we must have $r^*=1$ and $\phi^*(1)=\{1,2\}$ by the third property of $\opt^*_{\bullet}$ solutions. 
So $Q_0^{\bullet}(K^{(3)}_3,3) = \log(2)$.
It was shown in \cite{balogh2019typical} that $K_n$ is the graph with the most 3-edge colourings free of rainbow triangles, and that the number of such colourings is $(\binom{3}{2}+o(1))2^{\binom{n}{2}}$, which matches $Q_0^{\bullet}(K^{(3)}_3,3)$.

The question of whether $\x$ is $\bullet$-bounded is a Ramsey-type question. 
If it has a positive answer, then this increases the motivation for studying Problem $Q^\bullet(\x)$ and it could lead us to general results that apply to an even wider class of colour patterns $\x$.

\begin{problem}
    For all integers $s \geq 2$ and $k \geq 3$ and every (symmetric) family $\x$ of $s$-edge colourings of $K_k$, is $\x$ $\bullet$-bounded?
\end{problem}

To show that the monochromatic family $(K^{(1)}_k,s)$ and the dichromatic triangle family $(K^{(2)}_3,s)$ are bounded, we showed the stronger statement that there is some $R$ for which $r<R$ for every \emph{feasible} solution $(r,\phi,\ba)$ where all multiplicities are at least 2.
As we know, for general patterns, this is not true (e.g.~for the rainbow triangle and $3$ colours), so this approach will not suffice to prove $\bullet$-boundedness.

\subsection{A more general exact result}

The `hermetic' assumption in~\autoref{th:exact1} allows us to simplify the proof in~\cite{katherine_exact} significantly,
and is sufficient to prove our main results,~\autoref{th:K32} and~\autoref{th:proper}.
However, the assumption is fairly strong. 
We see no obstacle to proving an analogue of the exact result in~\cite{katherine_exact} which would replace the assumption that $\x$ has the extension property and is hermetic
with the assumption that $\x$ has the strong extension property (recall \autoref{extprop}). 
Such a result should allow the recovery of the results in the literature on non-monochromatic patterns.
We did not attempt this in the present paper since it is already rather long and would only yield new results where we could solve the corresponding optimisation problem.

\subsection{The dichromatic triangle problem}

We have essentially solved the dichromatic triangle problem completely, for large $n$. The outstanding part concerns the set $\Rev(s)$ of optimal $r$ in the optimisation problem, which we show contains at most two values $\rev(s)=2\lfloor (W(s/e)+1)/2 \rfloor$ and $\rev(s)+2$.
Given any $s$, one can simply calculate $\gf_s(r)$ for both these values $r$ to see which is optimal.
We strongly suspect that, apart from for $s=27$, the set contains a single element. This is a number theoretic statement: we conjecture that the equation $\gf_s(r)=\gf_s(r+2)$, or equivalently, 
$$
z^{(r-1-a)(r+2)}(z+1)^{a(r+2)} = y^{(r+1-b)r}(y+1)^{br}
$$
where 
$z=\lfloor\frac{s}{r-1}\rfloor$, $y=\lfloor\frac{s}{r+1}\rfloor$,
$a=s-(r-1)z$, $b=s-(r+1)y$
has no solutions for any integer $s \neq 27$, where $r=\rev(s)$. 
It is not too hard to show that, for example, there is no other solution when $(r-1)(r+1) \ | \ s$,
in other words, no other solution to $\hf_s(r)=\hf_s(r+2)$.

In~\autoref{lm:Rsprecise}(iii) we showed that the set $S_2(r)$ of $s$ for which $r \in \Rev(s)$ is an interval, and furthermore that $S_2(r)$ and $S_2(r+2)$ overlap in at most one value of $s$. We provided some fairly weak bounds for the startpoint $\tilde s_r$ (and endpoint) of $S_2(r)$.
From~\autoref{table:1} which shows $\Rev(s)$ for small values of $s$, it seems that, as $r \to \infty$, we have $W(\tilde s_{r+2}/e) - r \to 0$.
That is, the rough value of $\tilde s_{r+2}$ when $r+2$ becomes optimal instead of $r$ is about when $W(\tilde s_{r+2}/e)=r$, 
which yields $\tilde s_{r+2}\approx re^{r+1}$. 

Even if we cannot prove that $|\Rev(s)|=1$ for all $s \neq 27$, 
one could try to show that there is nevertheless a unique extremal graph by comparing the number of colourings $c(n,r) := (C+o(1))e^{\frac{r}{r-1}\gf(s)t_r(n)}$ for $r=\rev(s),\rev(s)+2$ arising from the two potential extremal graphs $T_{\rev(s)}(n)$ and $T_{\rev(s)+2}(n)$
in the case that $\Rev(s)=\{\rev(s),\rev(s)+2\}$, for which it can be shown that $e^{\frac{r-1}{r}\gf(s)t_r(n)}$
and $e^{\frac{r+1}{r+2}\gf(s)t_{r+2}(n)}$ differ by a multiplicative factor independent of $n$. Determining the constant $C$ which depends on $r,s$ and $n~(\!\!\!\mod r)$, is straightforward for small values of the parameters but is in general difficult. Even with the favourable divisibility conditions $r \ | \ n$ and $(r-1) \ | \ s$, in which case all parts of $T_r(n)$ have the same size and all multiplicities are equal, $C$ equals the number of decompositions of the multigraph $\frac{s}{r-1}K_r$ into perfect matchings (see~\autoref{cons:general}).

\subsection{Other colour patterns}

We hope that the ideas and methods of this paper will be useful for other colour patterns, including the central case of monochromatic patterns.
For the family $(K^{(1)}_k,s)$, it is expected that, typically, every basic optimal solution $(r^*,\phi^*,\ba^*)$ satisfies that $(k-1) \ | \ r^*$,
${\phi^*}^{-1}(c) \cong T_{k-1}(r^*)$ for all $c \in [s]$ and $\ba^*$ is uniform. (Note that this is not always true; there is a single known example $(K^{(1)}_3,5)$ where there are basic optimal solutions without these properties.)
Recall that our key idea for the dichromatic triangle problem was to consider the contribution of a largest part, and to show that if $r^*$ is odd then this contribution is non-optimal. A similar idea may rule out $(k-1) \not| \ r^*$ for instances of the monochromatic problem.

We conjecture that for every pattern, every extremal graph for the generalised Erd\H{o}s-Rothschild problem is complete partite.

\begin{conj}
    For all integers $k \geq 3$ and $s \geq 2$ and every colour pattern $P$ of $K_k$, whenever $n$ is sufficiently large, every $n$-vertex $(P,s)$-extremal graph is complete partite.
\end{conj}

The results of this paper together with the results of~\cite{alon2004number,balogh2006remark,benevides2017edge} imply that we may assume $P$ is monochromatic, in which case this conjecture is implied by~\cite[Conjecture~16]{katherine_asymptotic} (which is slightly stronger).

As we have noted, results are very sporadic and there are many patterns which have not been studied. Monochromatic patterns are still of central interest. We finish with some questions about specific patterns of particular interest.

\begin{problem}
    What is the extremal graph for the following?
    \begin{enumerate}
        \item The monochromatic triangle pattern $K^{(1)}_3$ for $s \geq 8$ colours. It is believed $s=11$ could be the next most tractable case (see~\cite{katherine_exact}).
        \item The rainbow pattern $K^{(6)}_4$ for any number $s \geq 6$ of colours. It was shown in \cite{hoppen2021extension} that for $s\ge 5434$ the extremal graph is $T_{3}(n)$. We conjecture, by comparing the number of 5-colourings of $K_n$ and the number of $s$-colourings of $T_3(n)$, that this holds for $s\ge 12$, and that for $s\le 11$ the extremal graph is $K_n$.
        \item The dichromatic family $K_4^{(2)}$ as well as $K_4^{(\ge \ell)}$ and $K_4^{(\le \ell)}$ for $2\le\ell \le 6$.
    \end{enumerate}
\end{problem}

\section*{Acknowledgements}

The work leading to these results began at the workshop \emph{Young Researchers in Combinatorics} which took place 18--22 July 2022 at the ICMS in Edinburgh. We are very grateful for the excellent research environment provided during the workshop.
We would also like to thank Akshay Gupte for helpful conversations at the beginning of the project.

\bibliographystyle{abbrv}
\bibliography{biblio}

\begin{appendices}

\numberwithin{equation}{section}

\section{Appendix}\label{appendix}

\begin{proof}[Proof of \autoref{lm:hanalytic}]
To prove~(i) and (ii), note that
$$
\hf_s'(x) = \frac{\log\left(\frac{s}{x-1}\right)-x}{x^2},
$$
and letting $m(s)$ be the unique solution to $\hf_s'(m(s))=0$, we have 
$$
    s/e=(m(s)-1)e^{m(s)-1}
$$
and hence
$$
m(s)=W(s/e)+1.
$$
We also see that $\hf_s'(x)$ is positive for $x \in(1,m(s))$ and
negative for $x \in (m(s),\infty)$, so $\hf_s$ is increasing on $(1,m(s))$ and decreasing on $(m(s),\infty)$ with $m(s)$ being the unique maximum. This proves~(i).

Noting that $m(s)-1=s/e^{m(s)}$, we have
$$
\hf(s) = \hf_s(m(s)) = \frac{m(s)-1}{m(s)}\log\left(\frac{s}{s/e^{m(s)}}\right) = m(s)-1=W(s/e),
$$
proving~(ii).

For part~(iii), we only prove the case where $a,b \geq 0$, since the other case can be proved analogously (in fact, it is slightly easier).
We will show that $\hf_s(m(s)+a')-\hf_s(m(s)+b')\ge \frac{1}{16}(\log(s)+5/2)^{-2}$ where $a'=a+1/4$ and $b'=\min\{b,5/2\}$, which implies~(iii) since $a< a'< b'\leq b$ and $\hf_s$ is monotone decreasing on $[m(s),\infty)$. By the mean value theorem there is $x \in (m(s)+a',m(s)+b')$ such that
 \begin{align*}
 \hf_s(m(s)+a')-\hf_s(m(s)+b') &=
(b'-a') \frac{x-\log(\frac{s}{x-1})}{x^2}\\
&=(b'-a')\frac{-\log(\frac{s}{x-1})+\log(\frac{s}{m(s)-1})+x-m(s)}{x^2}
\\ &=(b'-a')\frac{\log(\frac{x-1}{m(s)-1})+x-m(s)}{x^2} \\
&\geq (b'-a')\frac{a'}{x^2}\ge \frac{1}{16(\log(s)+5/2)^2}
\end{align*}
using $a'\ge 1/4$, $b'-a'\ge 1/4$, $b'\le 5/2$ and the fact that $x \leq m(s)+b'\leq \log(s)+5/2$ which follows from $\log(s) \geq 2$, \eqref{eq:rs} and part (i).

For part~(iv), we only cover the case $b<0$ which is trickier. Again by the mean value theorem there is some $x\in (m(s)+b,m(s))$ such that $\hf_s(m(s))-\hf_s(m(s)+b)=
|b| \hf_s'(x).$ Using that $-b=|b|\leq m(s)-2<m(s)-1$ and $\log(1+y)\geq \frac{y}{1+y}$ for $y>-1$, we have $$
\left|\log\left(\frac{x-1}{m(s)-1}\right)\right|\le \left|\log\left(\frac{m(s)+b-1}{m(s)-1}\right)\right|\le \frac{|b|}{m(s)+b-1}\le |b|.
$$
Thus
\begin{align*}
 \hf_s(m(s))-\hf_s(m(s)+b)=  |b|\frac{\log(\frac{m(s)-1}{x-1})+m(s)-x}{x^2} \le |b|\frac{2|b|}{(m(s)-2)^2}\le \frac{8|b|^2}{(\log(s)-4)^2}, 
\end{align*}
where we used that $m(s)-2\ge \log(s)/2-2>0$ by part (i),  \eqref{eq:W} and our lower bound $s\ge 55$.
\end{proof}

\begin{proof}[Proof of \autoref{lm:gapprox}]
\autoref{lm:Rsprecise}(i) implies that $\gf(s)=\gf_s(r)$ for some $r \in\{ \rev(s),\rev(s)+2\}$. Also $\hf_s(m(s))=W(s/e)$ is the unique maximum of $\hf_s$ and $|r-m(s)|\le 2$ by definition of $\rev(s)$. From~\autoref{lm:esr} we have 
$$W(s/e)-\gf(s)=\hf_s(m(s))-\gf_s(r)\ge \hf_s(m(s))-\hf_s(r)\ge 0.$$
From~\autoref{lm:esr} and~\eqref{eq:rs} we have
$$
0 \leq e_s(r) \leq \frac{1}{4}\left(\frac{s}{\log(s)+1}-1\right)^{-2} \leq \frac{(\log(s))^2}{s^2},
$$
which yields 
$$W(s/e)-\gf(s)=\hf_s(m(s))-(\hf_s(r)-e_s(r))\le \frac{32}{(\log(s)-4)^2} + \frac{(\log(s))^2}{s^2}\le \frac{600}{(\log(s))^2}.$$
where to bound the term $\hf_s(m(s))-\hf_s(r)$ we used~\autoref{lm:hanalytic}(iv) with $b=r-m(s)$.
The upper and lower bounds on $W(s/e)-\gf(s)$ give the first part of the lemma.
Since the upper bound tends to $0$ as $s \to \infty$, and $e^{W(s/e)}=(s/e)/W(s/e)$, the second assertion follows.
\end{proof}

\begin{proof}[Proof of \autoref{lm:maxoflog}]
We use the weighted AM-GM inequality.
It implies that
\begin{align*}
x = \sum_{i \in [n]}\frac{a_i}{a}\cdot \frac{ax_i}{a_i} \geq \prod_{i \in [n]}\left(\frac{ax_i}{a_i}\right)^{a_i/a}
\end{align*}
and therefore
$$
\prod_{i \in [n]}x_i^{a_i} \leq x^a\prod_{i \in [n]}\left(\frac{a_i}{a}\right)^{a_i} = \prod_{i \in [n]}\left(\frac{xa_i}{a_i}\right)^{a_i}.
$$
Taking logs (noting that all terms in both products are positive by our assumption) completes the proof of the first part.

Suppose now that we have positive integers $x_1,\ldots,x_n$ whose sum is $x$.
It suffices to show that the maximum product $x_1\ldots x_n$ is attained whenever $|x_i-x_j| \in \{0,1\}$ for all $i,j \in [n]$.
Suppose that $x_2-x_1 \geq 2$.
Let $y_1=x_1+1$, $y_2=x_2-1$, and $y_i=x_i$ for all $i\in [3,n]$. Then $y_1,\ldots,y_n$ are positive integers whose sum is $x$ and
    \begin{align*}
        y_1\cdots y_n=(x_1x_2+x_2-x_1-1)x_3\cdots x_n\geq x_1\cdots x_n+ x_3\cdots x_n> x_1\cdots x_n.
    \end{align*}
Thus if $x_1,\ldots,x_n$ does not satisfy the condition, then we can always increase the product.
Finally, there is a unique multiset $\{x_1,\ldots,x_n\}$ for which $x_1,\ldots,x_n$ satisfies the condition,
so any such $x_1,\ldots,x_n$ must attain the maximum.
This completes the proof of the second part.
\end{proof}

\begin{proof}[Proof of \autoref{lm:fcomp}]
For $s\le 1100$ we verify the lemma directly. A script for this calculation is provided in the ancillary file \texttt{small\_s.py}. For the rest of the proof, we assume that $s>1100$. 
Let $r \in \Rmax(s)$. By \autoref{lm:Rsprecise}(iv) we have $r\ge 6$. Moreover, by \autoref{lm:Rsprecise}(i), \eqref{eq:rs} and \eqref{eq:rs_lower}, we have $\log(s)/2\le r\le \log(s)+1$. 
Let
$$
x_0 := \frac{r}{r^2-1}
\quad\text{and}\quad
\eps := \frac{e_s(r+1)}{(r-1)\log(s/r)}.
$$
By \autoref{lm:esr} and our bounds on $r$ and $s$ we have
\begin{equation}\label{eq:eps}
1/\eps
\geq 4\left\lfloor\frac{s}{r-1}\right\rfloor^{2}(r-1)\log(s/r) \geq 4 \cdot \frac{1}{2}\left(\frac{s}{r}\right)^2 \frac{5r}{6} \cdot 4 = \frac{20s^2}{3r}
\end{equation}
and so
\begin{equation}\label{eq:x0}
x_0-\eps \ge\frac{r}{r^2-1}-\frac{3r}{20s^2}>\frac{1}{r}.
\end{equation}
To prove the lemma, we need to show that
\begin{equation}\label{eq:goal}
\ff_{s,r}(x) < \max\{\gf_s(r-1),\gf_s(r+1)\}
\quad\text{for all }\quad
x \in [x_0-\eps,\tfrac{1}{r-1}).
\end{equation}

\begin{claim}
    It suffices to show that 
    \begin{equation}\label{eq:app1}
\frelax_{s,r}(x_0) \leq \max\{\hf_s(r-1), \hf_s(r+1)\}-1/(5r^2).
\end{equation}
\end{claim}

\begin{proof}[Proof of claim]
First we will show that~\eqref{eq:goal} is implied by
\begin{equation}\label{eq:suffice}
\frelax_{s,r}(x_0-\eps) < \max\{\gf_s(r-1),\gf_s(r+1)\}.
\end{equation}
To see this, let $x \in [x_0-\eps,\frac{1}{r-1})$. Then~\eqref{eq:x0} implies that there is $\lambda \in [0,1)$ such that $x=(1-\lambda)(x_0-\eps)+\lambda/(r-1)$.
By~\autoref{lm:fprops}(i), $f_{s,r}$ is convex on $[\frac{1}{r},\frac{1}{r-1}] \supseteq [x_0-\eps,\frac{1}{r-1})$, so
$\ff_{s,r}(x) \leq (1-\lambda) \ff_{s,r}(x_0-\eps)+\lambda \ff_{s,r}(\tfrac{1}{r-1})$.
But $f_s(\frac{1}{r-1}) = \gf_s(r-1)$ and by~\autoref{lm:fprops}(ii), we have $\ff_{s,r}(x_0-\eps) \leq \frelax_{s,r}(x_0-\eps)$,
so $\ff_{s,r}(x) \leq (1-\lambda) \frelax_{s,r}(x_0-\eps)+ \lambda g_s(r-1)$.
Now~\eqref{eq:suffice} implies~\eqref{eq:goal}.

Thus it remains to show that~\eqref{eq:app1} implies~\eqref{eq:suffice}.
This will follow from the fact that $\frelax_{s,r}$ is close to $\ff_{s,r}$ and $\gf_s$ is close to $\hf_s$; more specifically, it suffices to show the following three inequalities.
\begin{align}\hspace{-10pt}
\frelax_{s,r}\left(x_0-\eps\right) - \frelax_{s,r}\left(x_0\right) &\le \eps (r^2-1) + \eps(r-1)\log\left(\frac{s}{r-1}\right)+\eps^2 r^3, \label{eq:get1}\\
    \hf_s(r-1) &\leq \gf_s(r-1) + \frac{1}{4}\left\lfloor \frac{s}{r-2}\right\rfloor^{-2} \leq \gf_s(r-1)+\frac{1}{4}\left\lfloor\frac{s}{r}\right\rfloor^{-2}, \label{eq:get2}\\
   \frac{1}{5r^2} &\geq \frac1{4\lfloor s/r\rfloor ^2}+\eps(r-1) +\eps (r-1)\log \left(\frac{s}{r-1}\right)+\eps^2 r^3.\label{eq:get3}
   \end{align}

To show~\eqref{eq:get1}, we estimate $\frelax_{s,r}(x_0-\eps)$ in terms of $\frelax_{s,r}(x_0)$ and $\eps$. We have
$$
\frelax_{s,r}(x_0) = (r-2)x_0\log\left(\frac{sx_0}{1-x_0}\right) + (1-(r-1)x_0)h(x_0)
$$
and
$$
\frelax_{s,r}(x_0-\eps) \leq (r-2)x_0\log\left(\frac{sx_0}{1-x_0}\right) + (1-(r-1)(x_0-\eps))h(x_0-\eps)
$$
where $h(y):=\log\left(\frac{(1-(r-1)y)s}{1-y}\right)$.
Using the mean value theorem, for $x_0=\frac{r}{r^2-1}$ and $\eps <x_0-1/r$, we have
\begin{align*}\hspace{-20pt}
h(x_0-\eps) &\leq h(x_0) + \eps\cdot \max_{y \in (x_0-\eps,x_0)}\frac{r-2}{(1-y)(1-(r-1)y)}\\
&= h(x_0) + \eps (r-2)\frac1{(1-r/(r^2-1))(1-r/(r+1))}\le h(x_0)+\eps (r^2-1).
\end{align*}
Thus, combining inequalities, we get 
\begin{align*}\hspace{-10pt}
\frelax_{s,r}\left(x_0-\eps\right) - \frelax_{s,r}\left(x_0\right) &\leq (1-(r-1)x_0)\eps (r^2-1) + \eps(r-1)h(x_0)+\eps^2(r-1)^2(r+1) \nonumber \\
&\le \eps (r^2-1) + \eps(r-1)\log\left(\frac{s}{r-1}\right)+\eps^2 r^3,
\end{align*}
where for the last inequality we used that $h(y)$ is decreasing on $[\frac{1}{r},\frac{1}{r-1}]$, so $h(x_0) \leq h(\frac{1}{r}) = \log(\frac{s}{r-1})$.
Thus~\eqref{eq:get1} holds.

Next,~\eqref{eq:get2} follows from~\autoref{lm:esr}.

To see~\eqref{eq:get3}, 
we first recall from~\eqref{eq:eps} that $\eps \leq 3r/20s^2$.
We can bound $20r^2$ times each of the four terms on the right hand side by $1$ using $r\le \log(s)+1$ and the fact that $s$ is large, as follows:
\begin{itemize}
    \item for the first term we have $\frac{20r^2}{4\lfloor s/r\rfloor^2}\le 5\left(\frac{r^2}{s-r}\right)^2 
    \le 1$ which holds for all $s \geq 65$;
\item the second term is at most the third;
\item for the third term we have $20r^2\eps (r-1)\log \left(\frac{s}{r-1}\right)\le \frac{3r^4}{s^2}\log(s)<1$ for $s \geq 133$; and 
\item for
the fourth term we have $20r^2\eps^2r^3 < \frac{9r^7}{20s^2} < 1$, using $s>881$.
\end{itemize}
This completes the proof that~\eqref{eq:get1}--\eqref{eq:get3} hold, and hence completes the proof of the claim.
\end{proof}

It remains to show that~\eqref{eq:app1} holds.
Let
$$
a := \frac{s}{r-2}, \quad b := \frac{sr}{r^2-r-1}, \quad c := \frac{s(r-1)}{r^2-r-1}, \quad d := \frac{s}{r}.
$$
Then $a>b>c>d$.
We have
$$
\hf_s(r-1) = \frac{r-2}{r-1}\log(a) \quad \text{and} \quad \hf_s(r+1) = \frac{r}{r+1}\log(d).
$$
Now,
\begin{align*}\hspace{-30pt}
\frelax_{s,r}(x_0) 
&= \frac{r(r-2)}{(r-1)(r+1)}\log(b) + \frac{1}{r+1}\log(c)
= \begin{cases} \hf_s(r-1) + \frac{1}{r^2-1}(\log(a)-u_{r-1})\\ \hf_s(r+1) - \frac{1}{r^2-1}(\log(d)-u_{r+1})
\end{cases}
\end{align*}
where
\begin{align*}
u_{r-1} &= r(r-2)\log\left(\frac{r^2-r-1}{r^2-2r}\right) + (r-1)\log\left(\frac{r^2-r-1}{r^2-3r+2}\right)
\quad\text{and}\\
u_{r+1} &= r(r-2)\log\left(\frac{r^2}{r^2-r-1}\right) + (r-1)\log\left(\frac{r^2-r}{r^2-r-1}\right).
\end{align*}
Then using inequalities approximating $\log(1+x)$ by $x$ we have that
\begin{align*}
u_{r-1}-u_{r+1} =  r(r-2)\log\left(\frac{(r^2-r-1)^2}{(r^2-2r)r^2}\right) + (r-1)\log\left(\frac{(r^2-r-1)^2}{(r^2-3r+2)(r^2-r)}\right) \geq 0.9
\end{align*}
for all $r \geq 6$ (and indeed a Taylor expansion shows that $u_{r-1} - u_{r+1} \approx 1$ for large $r$).
Suppose that~(\ref{eq:app1}) does not hold.
Then
\begin{align*}
\log(a) &\geq u_{r-1} - (r^2-1)/(5r^2)
\quad\text{and}\quad
\log(d) \leq u_{r+1} + (r^2-1)/(5r^2).
\end{align*}
Subtracting the first inequality from the second, we have
\begin{align*}\hspace{-20pt}
\frac{2(r^2-1)}{5r^2} \geq \log\left(\frac{r-2}{r}\right) + u_{r-1} - u_{r+1} \geq \log\left(\frac{r-2}{r}\right) + 0.9
\end{align*}
which yields a contradiction for every $r \geq 6$.
Thus~(\ref{eq:app1}) holds, completing the proof of the lemma.
\end{proof}

\end{appendices}

\end{document}